\documentclass{amsart}
\usepackage{amssymb}
\usepackage[matrix,arrow,curve]{xy}
\usepackage{verbatim}
\usepackage{graphics}
\usepackage{a4wide}
\usepackage[applemac]{inputenc}
\newtheorem{theorem}{Theorem}[section]
\newtheorem{lemma}[theorem]{Lemma}
\newtheorem{remark}[theorem]{Remark}
\newtheorem{definition}[theorem]{Definition}

\newtheorem{corollary}[theorem]{Corollary}
\newtheorem{question}[theorem]{Question}

\newtheorem{proposition}[theorem]{Proposition}

\usepackage{tikz}

\def\dd{\mathrm{d}}
\def\un{\mathbf{1}}
\def\eps{\varepsilon}
\def\diam{\mathrm{diam}}

\def\Aut{\mathrm{Aut}}
\def\C{\mathbf{C}}
\def\Z{\mathbf{Z}}
\def\N{\mathbf{N}}

\def\Q{\mathbf{Q}}
\def\C{\mathbf{C}}
\def\R{\mathbf{R}}
\def\gl{\mathfrak{gl}}

\def\Hom{\mathrm{Hom}}

\def\Imm{\mathrm{Im}}
\def\Id{\mathrm{Id}}
\def\into{\hookrightarrow}
\def\onto{\twoheadrightarrow}
\def\Ad{\mathrm{Ad}\,}

\def\ii{\mathrm{i}}
\def\la{\lambda}

\date{January 12, 2017.}

\begin{document}
\centerline{}

\title[Measure theory and Classifying Spaces]{Measure theory and Classifying Spaces}

\author[I.~Marin]{Ivan Marin}
\address{LAMFA, Universit\'e de Picardie-Jules Verne, Amiens, France}
\email{ivan.marin@u-picardie.fr}
\dedicatory{to the memory of Adrien Douady}
\subjclass[2010]{}
\medskip

\begin{abstract} 
We construct classifying spaces for discrete and compact Lie groups,
with the property that they are topological groups and complete metric spaces in a natural way.
We sketch a program in view of extending these constructions.
\end{abstract}

\maketitle

\tableofcontents

\section{Introduction}

This work presents an exploration of the possibilities to use measure theory
in order to build and put some natural geometry on classifying spaces for groups.
The main observation is that the probabilistic notion of a $G$-valued
random variable provides a convenient setting in order to construct a classifying
space $L(\Omega,G)/G$ for the group $G$ with a natural (geo)metric structure
in several interesting cases. Here $G$ is a bounded metric group (for instance a discrete
group endowed with the discrete metric), $\Omega$
is the standard probability space, $L(\Omega,G)$
is the space of Borel maps $\Omega \to G$ up to neglectability, equipped with a kind of $L^1$
metric, and $G$ acts on $L(\Omega,G)$ by translations.

This construction itself
presents some difficulties, though. One of these is categorical in nature, and inherent to the
idea of building a \emph{geometric} classifying space : since we are
building metric objects, the right category for categorical operations is not necessarily the
classical topological or homotopy categories, but might have to take into account the
metric structure. 
This has notable consequences
concerning the limits, which do not lead to the same objects depending on whether
they are taken inside metric categories or inside the topology category. The general
results obtained here are lacking applications for now, but this study might be
useful when considering classifying spaces for profinite groups as well as
Malcev completions for rational homotopy theory. We hope to come back to these
questions in future work.

The main technical difficulty we are confronted with in this paper is the question of whether
the projection map $L(\Omega,G)\to L(\Omega,G)/G$ admits a local (continuous) cross-section. This property, which is granted almost for free when the classifying spaces
are locally compact, is crucial for us to apply the classical theory proving the unicity, up to a homotopy, of a classifying space.
For this reason, even when $G$ is homotopically equivalent
to a CW-complex, we are for instance not able so far to conclude on whether $L(\Omega,G)/G$ has the same property.

It is worth noticing that this question pertains to a natural question of independent interest. It is about whether, when a distribution of mass is concentrated \emph{in measure}
in some neighbourghood of a measured topological space $G$, one can assign to it some point of the space in a continuous way, with respect to a natural topology for the distribution. Indeed,
a local section of $L(\Omega,G) \to L(\Omega,G)/G$ can be
intuitively conceived as attaching to a sufficiently local distribution of mass on $G$ some element of $G$ is a $G$-equivariant way.

 When $G$ is the real line (or more generally a real vector space), a natural positive answer is provided by the center of mass, defined using the average of the coordinates, but in general the homotopy properties of the space $G$ prevent us from having such a globally defined cross-section. We deal in detail with the example of the circle $S^1$ endowed with its
 usual metric, and provide local cross-sections which assign continuously a point to such a distribution of mass, concentrated in measure in the interior of some (numerically specified) solid angle.

The case of $S^1$ fits into the case where $G$ is a compact Lie group, endowed with
a natural bi-invariant Riemannian metric. In this case, such a local cross-section exists by a general theorem of Gleason pertaining to quotients $E/G$ when $E$ is a completely regular topological space. Gleason's general proof is elegant and powerful but rather indirect. In this paper we construct explicit local sections in a geometric way, following the idea above.
Here, our geometric cross-sections are based on the (local) theory
of the Riemannian center of mass :  such a section can thus be considered as
an approximate notion of a `center of mass' for such a distribution. With this approach the neighborhood on which the local cross-section takes its arguments is intrinsically defined from the geometry of $G$ as a Riemannian manifold, and can be easily specified in each concrete case.

Our goal is then to explore the natural idea of iterating the construction. Starting from
an abelian discrete group $\Gamma$, the above construction provides another
metric group $G = L(\Omega,\Gamma)/\Gamma$, and one wonders whether
$L(\Omega,G)/G$ has the properties of a classifying space for $G$, or at least
if it provides a $K(\Gamma,2)$ -- and subsequently, whether iterating this construction
provides a natural construction of complete, metric, Eilenberg-MacLane spaces.
In order to understand these groups we need exponentiation properties
of the form $L(\Omega,L(\Omega,\bullet)) \simeq L(\Omega \times \Omega,\bullet) 
\simeq L(\Omega,\bullet)$ that we establish in section \ref{sect:exponential}. The
cross-section problem becomes much more difficult there, and we content ourselves,
after having described various avatars of these spaces,
to write down a catalogue of the questions that need to be answered in that
setting, together with some properties that link them together (section \ref{sect:iteration}).

We like to view this work as providing some indication that a \emph{geometrization program} for the homotopy
category of spaces can be envisioned, and that it can be based on measure theory and probability theory. We hope
to be able to provide more evidence in this direction in future work.

\bigskip

{\bf Acknowledgements.} 
I thank P. Vogel for several inspiring discussions on classifying spaces
and G. Godefroy for his crucial technical help on several points.
I thank 
Ivo dell'Ambrogio,
E. Deleuchi
and Livio Flaminio
for useful references,
and Alain Rivière
for many interesting discussions. 
I thank my `ergodic' colleagues S. Petite, A.-H. Fan, F. Paccaut, F. Durand, B. Testut
and T. Gauthier for their
interest and technical help and J.-Y. Charbonnel for a discussion on compact Lie groups.

\section{Preliminaries}

In this section we recall very basic material on metric and topological groups on the one hand, and standard probability spaces on the other hand.

\subsection{Metric groups, probability spaces}
 Recall (see \cite{BOURBTOP} ch. 9, \S 3) that a topological
group is metrizable as a topological space if and only if it is Hausdorff and if its neutral element admits a \emph{countable} fundamental
system of neighborhoods. In this case, $G$ can be endowed with a metric, defining the topology of the group which is left-invariant,
meaning that $d(gx,gy) = d(x,y)$ for all $g,x,y \in G$. We call such a structure $(G,d)$ a \emph{metric group}.
Of course a given metrizable topological space admits several left-invariant metrics ; in particular, it always admits a \emph{bounded} one,
that is one for which the diameter of $G$ is finite, for
it suffices to replace an arbitrary left-invariant one $d$ by $(x,y) \mapsto \min(1,d(x,y))$. If $G$ is commutative, a left-invariant distance is
clearly right-invariant.

The examples which we are primarily interested in are the following ones :
\begin{enumerate}
\item if $G$ is discrete, it is a metric group of diameter $1$ for the distance $d(x,y) = 1$ if $x \neq y$, $d(x,x)=0$.
\item $S^1 = \{ z \in \C \ | \ |z| = 1 \}$ is a metric group of diameter $2$ for the euclidean distance $d(z_1,z_2) = |z_1-z_2|$,
as well as for the arc-length distance.
\item If $G$ is the inverse limit of discrete quotients $G/G_n$ (for instance if $G$ is profinite) with $G_{n+1} \subset G_n$, then a distance is
given by $d(x,y) = 2^{- \delta(xy^{-1})}$ where $\delta : G \to \N \cup \{ + \infty \}$ is defined by $\delta(g) = \sup \{ n \ ; \ g \in G_n \}$.
This distance is ultrametric, namely $d(x,z) \leq \max( d(x,y),d(y,z) )$, it is left-invariant, and it is also right-invariant if the $G_n$
are normal in $G$. This can always be assumed if $G$ is profinite.
\item If $G$ is a compact Lie group, any invariant scalar product on its Lie algebra
endows $G$ with a bi-invariant Riemannian metric.
\end{enumerate}

When $G$ is a metric group and $H$ a subgroup of $G$, then $G/H$ is Hausdorff iff $G/H$ is metrizable, in which case
the metric is given by $d(\overline{x},\overline{y}) = \inf \{ d(x,y) \ | x \in \overline{x},y \in \overline{y} \}$
(see e.g. \cite{BOURBTOP} ch. 9 \S 3 n$^o$ 1, Remarque) ;
 if moreover
$G$ is complete, then so is $G/H$ (see \cite{BOURBTOP} ch. 9 prop. 4 (TG IX.25)).

In the sequel, we will denote by $\Omega$ a uncountable
standard Borel space (that is a uncountable measurable space isomorphic to a Borel subset of a Polish space) endowed with a continous (i.e. non-atomic) probability measure $\mu$,
sometimes denoted $\dd \mu$ when needed, and sometimes implicit (as in $\int_{\Omega}g(t) \dd t$).
It is known that such an object is unique up to isomorphism, and isomorphic for instance to every
euclidean cube $[0,1]^n$ endowed with the Lebesgue measure (see e.g. \cite{SRI} Theorem 3.4.23). The measurable subsets of $\Omega$ will be
called Borel, as well as the measurable maps $\Omega \to X$ when $X$ is a topological space, implicitely
endowed with its Borel $\sigma$-algebra.

\subsection{Squaring the probability space}

By the general argument that all non-atomic probability spaces are isomorphic to each other we have $\Omega^2 \simeq \Omega$, hence $\Omega^n \simeq \Omega$
for all $n$. The existence of such isomorphism will be at the heart of our construction in the last section of the paper.
For the convenience of the non-expert reader, and also because it is meaningful here to have direct comparisons, we recall that a direct Borel isomorphism (that is, a bijection which is Borel as well as its inverse) preserving the measure
can be given explicitely
as follows. For a given sequence $\underline{\eps} = (\eps_1,\eps_2,\dots) \in \{0,1 \}^{\N}$,
let $x_{\underline{\eps}} = \sum_{i \geq 1} \eps_i 2^{-i}$. Let $C$ the set of sequences $\underline{\eps}$ such that $\eps_i = 1$
for $i \gg 0$. Then the map $(\{0,1 \}^{\N} \setminus C) \to [0,1[$, $\underline{\eps} \mapsto x_{\underline{\eps}}$
is a bijection.
Let $D : \{ 0,1 \}^{\N} \to \{ 0,1 \}^{\N} \times \{ 0,1 \}^{\N}$ being given by $\underline{\eps} \mapsto (\underline{\eps}^o, \underline{\eps}^e)$ with $\eps^o_i = \eps_{2i-1}$, $\eps^e_i = \eps_{2i}$. This is clearly a bijection.
We let $C' = \{ \underline{\eps} \in \{0,1\}^{\N} ; \underline{\eps}^o \in C \, \mbox{or} \, \underline{\eps}^e \in C \}$.
We have $C \subset C'$. Let $E \subset [0,1]$ be the image of $\{ 0,1 \}^{\N} \setminus C'$ under $\underline{\eps} \mapsto x_{\underline{\eps}}$, and $F \subset [0,1]\times [0,1]$ the image of $E$ under
$\check{f} : x_{\underline{\eps}} \mapsto (x_{\underline{\eps}^o},x_{\underline{\eps}^e})$. We have $F = F_0 \times F_0$
where $F_0 \subset [0,1]$ is the image of $\{0,1 \}^{\N} \setminus C$ under $\underline{\eps} \mapsto x_{\underline{\eps}}$,
that is $F_0 = [0,1[$ and $F = [0,1[^2$.
Clearly $\check{f}$ is injective, hence it is a bijection $E \to F$. 

Now $[0,1[ \setminus E$ is the image under $\underline{\eps} \mapsto x_{\underline{\eps}}$
of $C' \setminus C$, which is a countable intersection of intervals, and therefore a Borel set.
Moreover, it is the intersection of a decreasing sequence of intervals, each of them having measure one half
the preceding one ; therefore, it has measure $0$.

If $g$ is a Borel isomorphism $[0,1]\setminus E \to [0,1]^2 \setminus F$, and if so is $\check{f}$,
by gluing $g$ and $\check{f}$
we would get a Borel isomorphism $f : \Omega \to \Omega^2$. Since $[0,1]\setminus E$ and $[0,1]^2 \setminus F$
both have measure $0$ it would then be enough to check that $\lambda(\check{f}(X))=\lambda(X)$
for $X$ running among any generating system the Borel $\sigma$-algebra of $E$.

Now $C' \setminus C = C_e \sqcup C_o$,
with $C_e = \{ \underline{\eps}; \underline{\eps}^o \in C \}$,
$C_o = \{ \underline{\eps}; \underline{\eps}^e \in C \}$. Then $[0,1] \setminus E = \{1\} \sqcup E_e\sqcup E_o$
with $E_e = \{ x_{\underline{\eps}} ; \underline{\eps} \in C_e\}$,
$E_o = \{ x_{\underline{\eps}} ; \underline{\eps} \in C_o\}$.
Then $x_{\underline{\eps} } \mapsto (1,x_{\underline{\eps}^e})$ maps $E_e \to \{ 1 \} \times [0,1[$
while $x_{\underline{\eps} } \mapsto (x_{\underline{\eps}^o},1)$ maps $E_o \to  [0,1[\times \{ 1 \} $.
Since $[0,1]^2 \setminus F = \{ (1,1) \} \sqcup \{ 1 \} \times [0,1[ \sqcup  [0,1[\times \{ 1 \}$
one gets this way a suitable bijection $g$.

\bigskip

We prove that $\check{f}$ is Borel.
The Borel $\sigma$-algebra of $[0,1]$ is easily checked to be generated by the intervals $[a,b[ = I_{\underline{\eps}}$,
where $a = \sum_{i=1}^{n_0} \eps_i 2^{-i}$, $b = \sum_{i=1}^{n_0} \eps_i 2^{-i} + 2^{-n_0}$ for some $n_0 \geq 2$ and $\underline{\eps}
= (\eps_1,\eps_2,\dots,\eps_{n_0}) \in \{0,1\}^{n_0}$. Therefore the Borel $\sigma$-algebra of $[0,1[^2$ is generated
by the sets of the form $I_{\underline{\eps^1}} \times I_{\underline{\eps^2}}$ for some $\underline{\eps}^k \in \{ 0,1 \}^{n_k}$, $k =1,2$.
The inverse image under $\check{f}$ of such a set is clearly the union of $2^{|n_2-n_1|}$ sets of the form $I_{\underline{\eta}}$,
with $\eta \in \{0,1\}^{2 \max(n_1,n_2)}$
and thus is a Borel set. This proves that $\check{f}$ is Borel. With the notations $\underline{\eps}^e= (\eps_2,\eps_4,\dots,\eps_{ 2\lfloor n_0/2\rfloor})$ 
and $\underline{\eps}^o= (\eps_1,\eps_3,\dots,\eps_{ 2\lceil n_0 /2\rceil-1})$, 
we have $\check{f}(I_{\underline{\eps}}) = I_{\underline{\eps}^o} \times
I_{\underline{\eps}^e}$ and therefore the inverse of $\check{f}$ is also Borel. The fact that $g$ is bi-Borel is proved in the same way.

Finally, for $\underline{\eps} \in \{0,1\}^{n_0}$ we have $\lambda(I_{\underline{\eps}})=2^{-n_0}$. Then
$\underline{\eps}^o \in \{0,1 \}^{n^o}$ and
$\underline{\eps}^e \in \{0,1\}^{n^e}$ with $n^o + n^e = n_0$ hence
$$\la( \check{f}(I_{\underline{\eps}}) =\la( I_{\underline{\eps}^o}) \times
\la(I_{\underline{\eps}^e}) = 2^{-n^o} 2^{-n^e} = 2^{-n_0} = \la(I_{\underline{\eps}})
$$
hence $\check{f}$ is measure-preserving, and this concludes the construction of an explicit isomorphism $\Omega \to \Omega \times \Omega$.

\medskip

Another way to exhibit this Borel isomorphism $\Omega \simeq \Omega^2$ is to combine
the \emph{homeomorphism} between the Cantor ternary set $\mathcal{C} = \{ \sum_i \eps_i 3^{-i} ;  \forall i \ \eps_i \in \{0,2 \} \}$
and $\mathcal{C} \times \mathcal{C}$ induced by $\underline{\eps} \mapsto (\underline{\eps}^o, \underline{\eps}^e)$ 
and the classical Borel isomorphism between $[0,1]$ and $\mathcal{C}$ given by $\sum \eps_i 2^{-i} \mapsto \sum 2\eps_i 3^{-i}$
(\cite{SRI}, example 3.3.1).

\section{A Metric Classifying space -- the discrete case}

In this section we explain the construction in case $G$ is discrete, and establish
the main properies of $L(\Omega,G)/G$. We show that it is indeed a classifying
space, and how some of the most usual properties of the classifying space can be proved directly on this explicit construction. Finally, we establish a few properties relative
to the pointset topology of $L(\Omega,G)/G$, which are especially useful when
$G$ is uncountable.

\subsection{The construction}
Let us endow the discrete group $G$ with the discrete distance, $d(x,y) = 1- \delta_{x,y}$.
We let $L(G) = L(\Omega,G)$ denote the space of Borel maps from $\Omega$ to $G$ modulo the equivalence
relation $f_1 \sim f_2$ if $\{ t \in \Omega ; f(t) \neq g(t) \}$ has zero measure. It is
a metric space for the distance
$$
d(f_1,f_2) = \int_{\Omega} d(f_1(t),f_2(t)) \dd t.
$$
If we fix an isomorphism $\Omega \simeq [0,1]$ this space admits
interesting subspaces, as in e.g. \cite{DOUADY}, \S 4.6, exercice 6, p. 243 :
\begin{itemize}
\item instead of considering all Borel maps, one may consider staircase maps $[0,1] \to G$
modulo the same equivalence relation;
\item instead of dividing out by maps whose support have measure zero, one may consider staircase maps
which are semicontinuous on one side. More precisely, we may consider the set
of maps $f : ]0,1] \to G$ such that there exists a finite increasing sequence $0=t_0< t_1< \dots < t_n = 1$ with
$f$ constant when restricted to $]t_{i-1},t_i]$ for $i \in \{1,\dots,n \}$, and let $d(f_1,f_2)$ be the sum of the lengths of
the intervals on which $f_1$ and $f_2$ differ.
\end{itemize}
Notice that an isomorphism $\Omega \to [0,1]^n$ for higher $n$ could also be used to define similar
higher dimensional analogues of the above subspaces.
All what follows, except the statements about completeness, hold true for these `smaller' subspaces.

\begin{proposition}
Let $G$ be a discrete group endowed with the discrete metric. Then $L(G)$ is a contractible
space, endowed with a free action of $G$ such that the quotient map $L(G) \to L(G)/G$ is a covering map.
\end{proposition}

\begin{proof}
We first prove that $L(G)$ is a contractible space.
Indeed, given $f \in L(G) \simeq L([0,1],G)$, let $f^{|s} \in L([0,1],G)$ for $s \in [0,1]$ be defined by $f^{|s}(t) = e$ if $s\geq t$, $f^{|s}(t) = f(t)$
otherwise. We set $H(s,f) = f^{|s}$. Then $H : I \times L([0,1],G) \to L([0,1],G)$
retracts $L(G)$ on the constant map, and is continuous, because
$$
d(H(s_1,f_1),H(s_2,f_2) = \int_{s_1}^{s_2} d_G(f_1(t),f_2(t)) dt \leq |s_2-s_1|.
$$
Moreover, $G$ acts on $L(G)$ by isometries : if $g \in G$ and $f \in L(G)$,
we let $g.f : t \mapsto g.f(t)$. This action is clearly free. We prove that $L(G) \to L(G)/G$ is a covering map,
that is, that forall $f \in L(G)$ there exists an open neighbourhood $U$ of $f$ such that,
$\forall g_1,g_2 \in G \ (g_1 \neq g_2) \Rightarrow (g_1U) \cap (g_2U) = \emptyset$. 
We can assume $g_1 = e$ and set $g = g_2$. Letting $U = \{ f' \in L(G)\  | \ d(f,f') < \frac{1}{2} \}$
we get that, if $f'', f' \in U$ with $f'' = g.f'$, we have 
$$
1 = d(f,g.f) \leq d(f,f'') + d(f'',g.f) = d(f,f'') + d(g.f',g.f)= d(f,f'') + d(f',f) < \frac{1}{2} + \frac{1}{2} ,
$$
a contradiction. This proves the claim.
\end{proof}

Note that the homotopies between $L(G)$ and $\{ * \}$ constructed in the proposition are in 1-1 correspondance with the collection of isomorphisms $[0,1] \to \Omega$
of measured spaces. We let $B(G) = L(G)/G$. 
 There is then a natural distance on $B(G) = L(G)/G$, defined by
$$
d_{B(G)}(\overline{f_1},\overline{f_2}) = \inf_{g_1,g_2 \in G} d_{L(G)}(g_1.f_1,g_2.f_2) = \inf_{g \in G} d_{L(G)}(f_1,g.f_2) 
= \inf_{g \in G} d_{L(G)}(g.f_1,f_2).
$$

Note that $L(G)$ and $B(G)$ have a natural base point, given by (the class of) the constant map equal to $e$.
The natural map $\pi_1(B(G),\overline{e}) \to G$ admits a natural inverse : to each $g \in G$
we associate the path $\tilde{\pi}_g : [0,1] \to L(G)$ defined by 
$$
\begin{array}{lccclr}
\tilde{\pi}_g(u) &:& t &\mapsto  & g & \mbox{ if } t < u \\
 & & t &\mapsto  & e & \mbox{ if } t \geq u \\
 \end{array}
$$
Then $\tilde{\pi}_g(0)$ is the constant map equal to $e$ and $\tilde{\pi}_g(1)$ is the constant map
equal to $g$. Therefore, the class of $\tilde{\pi}_g(u)$ modulo $G$ defines a loop
$[0,1] \to BG$ whose homotopy class $\pi_g$ is an explicit inverse of $\pi_1(BG,\overline{e}) \to G$.

\subsection{Historical comment}

The first ideas about this construction (in the discrete case) came during conversations
with A. Douady when I was doing my 1995 master thesis under his guidance. A version
due to him of the construction can be found in \cite{DOUADY}, where a proper subspace of
step-functions is used instead of $L(\Omega,G)$. Actually, in case $G$ is commutative, an allusion to the step-functions point of view can
be found in \cite{SEGALCOHO} appendix A, where details are given of a
construction of G. Segal of the classifying space of an (abelian) topological group,
originally defined in \cite{SEGALIHES}.
After having described a simplicial realization of it, G. Segal explains that it is
(the space of orbits of) the space of step-functions on the unit interval $[0,1]$ with values
in $G$, `but from this point of view the topology is rather obscure'. Nevertheless, G. Segal
states that there is a continous monomorphism from this topological space onto a dense subgroup of $L([0,1],G)$.
This point of view of step-functions has been used by Bourbaki in his very recent `Topologie alg\'ebrique' chapters (2016),
whose original first draft dates back to the 1970's (and which probably received significant input from A. Douady himself).
When $G$ is discrete, there is actually no loss, compared to the classical (Segal's or Milnor's) constructions
of a classifying space, in using instead the larger construction $L([0,1],G)/G$, in terms of homotopy theory :
this space is equivalent to the other ones inside the homotopy category. The gain
is that it is a complete metric space, in a fairly natural way, and actually a complete metric group.
Also, it enables us to define it as $L(\Omega,G)/G$, where
$\Omega$ is a standard probability space, since the \emph{topology} of the unit interval $I = [0,1]$
is actually totally irrelevant to its definition. We apologize for the collision with the notation $\Omega X$ 
for the loop space of a topological space $X$, but these two notations are equally well-established,
and since we are not going to use loop spaces in this paper, this should not cause any confusion.
From this remark it becomes clear that the space $L(\Omega,G)$ 
should actually be thought of as the space of \emph{random $G$-valued variables}.

\subsection{General homotopic properties}

The construction $G \leadsto L(G)$ is functorial, from the category of groups to
the category of (pointed) topological spaces and continuous maps. Indeed, if $\varphi \in \Hom(G_1,G_2)$,
we associate to each $f \in L(G_1)$ the map $\varphi \circ f \in L(G_2)$, and this
defines a metric contraction, which is an isometry if and only if $\varphi$ is injective. The
class of $\varphi \circ f$ in $B(G_2)$ only depends on the class of $f$ in $B(G_1)$, for $f'=g.f$
implies that $f,f'$ have representatives in $\Omega \to G$ (up to functions whose supports have
measure $0$) such that $\forall t \in \Omega \ f'(t) = g.f(t)$ ; therefore, 
$$
\forall t \in \Omega \ \varphi(f'(t)) = \varphi(g.f(t)) = \varphi(g)\varphi(f(t)) = \varphi(g).\varphi(f(t))
$$
hence $\varphi \circ f' = \varphi(g).\varphi\circ f$. This proves that $G \leadsto B(G)$
is also functorial, again from the category of groups to
the category of (pointed) topological spaces and continuous maps.

As is well-known from the properties of classifying spaces, the induced maps
in the homotopy category lead to bijections $B : \Hom(G_1,G_2) \to [B(G_1),B(G_2)]_*$, where $[X,Y]_*$ denotes the set of pointed homotopy classes from $X$ to $Y$.
One may check directly the injectivity : if $\varphi, \psi \in \Hom(G_1,G_2)$
yield to homotopic $B \varphi$ and $B \psi$, let $H$ be an homotopy between
both maps. Then $H$ induces an homotopy between the paths $B \varphi(\pi_g)$
and $B \psi(\pi_g)$ for all $g \in G$. But $B \varphi(\pi_g) = \pi_{\varphi(g)}$ and
$B \psi(\pi_g) = \pi_{\psi(g)}$. Then $\pi_{\varphi(g)} \sim \pi_{\psi(g)}$ iff $\varphi(g) = 
\psi(g)$ and this proves the injectivity.

In order to prove the surjectivity, we note that we have a functor $\pi_1$
from the pointed homotopy category $HoTop_*$ to the category $Gr$ of groups, for which
$\pi_1 B$ is the identity functor of the category of groups. Therefore, for
all $F : BG_1 \to BG_2$, the two maps $B \pi_1(F)$ and $F$ induce the
same morphisms on homotopy groups, and they are therefore homotopic
once we know that $BG$ has the homotopy type of a CW-complex. This follows from
the facts that it is a classifying space for $G$, that $G$ admits a classifying
space which has the homotopy type of a CW-complex, and that two classifying spaces for
the same group are homotopically equivalent.

We thus recover the fact that $B : Gr \to HoTop_*$ is fully faithful.

The classical fact that there is an homotopically trivial conjugation action of
$G$ on $BG$ is recovered in our context as follows.
For all $\varphi \in \Aut(G)$ et $f \in L(G)$, one defines
 $\varphi.f : t \mapsto \varphi(f(t))$. The map
$f \mapsto \varphi.f$ is clearly an isometry of $L(G)$.
If $g \in G$, then
$$\varphi.(g.f) = \left( t \mapsto \varphi(gf(t))
= \varphi(g)\varphi(f(t)) \right) = \varphi(g).(\varphi.f).$$
hence $\varphi.\bar{f} \in B(G)$ is well-defined for $\bar{f} \in B(G)$.
Since
$$
\varphi.(\psi.f) = \varphi. \left( t \mapsto \psi(f(t)) \right)
=  \left( t \mapsto\varphi (\psi(f(t))) \right) = (\varphi \circ \psi).f
$$
we get an action of $\Aut(G)$ on $L(G)$ which induces an
action of $\Aut(G)$ on $B(G)$.

In particular, we get in an elementary way a conjugation action of $G$ on $B(G)$
via $G \to \Aut(G)$. This is a map $G \times B(G) \to B(G)$.
We prove that this action is homotopically trivial, meaning that
it is homotopic to the trivial action. In order to do this we fix an identification $\Omega \simeq [0,1]$ and
introduce
a map $[0,1] \times G \to L(G)$, $(u,g) \mapsto g_u$
with $g_u(t) = e$ if $u \leq t$, $g_u(t) = g$ if $u > t$ and
$
F_u : G \times L(G) \to L(G)
$
defined as $F_u(g,f) : t \mapsto f(t)g_u(t)^{-1}$. It is a continuous map
 $F : [0,1] \times G \times L(G) \to L(G)$.

For every $h \in G$ we get
$F_u(g,hf)(t) = hf(t)g_u(t)^{-1}$ hence
$F_u(g,hf) = h.F_u(g,f)$. Therefore
the composite of $F$ with the natural map $L(G) \to B(G)$ 
is a continuous map
$\Psi : [0,1] \times G \times B(G) \to B(G)$.
Since $F_0(g,f) = f$ we have $\Psi_0(g,\bar{f}) = \bar{f}$
and $\Psi_0$ is the trivial action of $G$ on $BG$. Moreover
$F_1(g,f) = t \mapsto f(t)g^{-1}$, hence
$F_1(g,f)$ has the same image in $B(G)$
as the map $t \mapsto gf(t)g^{-1}$, and $\Psi_1$ is the action of
$\Ad(g) \in \Aut(G)$ on $B(G)$. This proves that 
the conjugation action of $G$ on $B(G)$ is homotopically trivial,
as claimed.

\subsection{General topology}

The proof of the following proposition is basically a copy-paste of classical arguments for usual function spaces.
\begin{proposition} $L(G)$ is complete as a metric space.
\end{proposition}
\begin{proof}
We let $(f_n)_{n \geq 0}$ denote a Cauchy sequence. For all $k \geq 1$, there exists $n_k \in \N$
such that, for all $p,q \geq n_k$, $d(f_p,f_q) \leq 1/2^{k+1}$. For every $u,v \in \N$ we set $A_{u,v}
= \{t \in \Omega ; f_{n_u}(t) \neq f_{n_v}(t) \}$ and $\Omega_m = \bigcup_{u,v \geq m} A_{u,v}$.
It is clear that $\Omega_m = \bigcup_{u \geq m} A_{u,u+1}$, and that $\Omega_{m+1} \subset I_m$. Moreover, the $A_{u,v}$
and $\Omega_m$ are measurable ; this holds true because the maps $t \mapsto (f_{n_u}(t),f_{n_v}(t))$ are measurable
because so are the $f_k$'s, and because the diagonal of $G\times G$ is a closed subset.
We have 
$$\mu(I_m) \leq \sum_{u \geq m} \mu(A_{u,u+1}) \leq \sum_{u \geq m} d(f_{n_u},f_{n_{u+1}}) \leq \sum_{u \geq m} 2^{-u-1} \leq 2^{-m}
$$
and $\Omega_{m+1} \subset \Omega_m$. From this get that $\Omega_{\infty} = \bigcap_m \Omega_m$ has measure $0$. We define then $f \in L(G)$ on the
complement of $\Omega_{\infty}$ by $f(t) = f_{n_m}(t)$ for $m$ such that $t \in \Omega_m \setminus \Omega_{m+1}$. This construction satisfies
the following : $(\forall u \geq k \ f_{n_u}(t) = f_{n_k}(t) )\Rightarrow (f_{n_k}(t) = f(t))$. Indeed, letting $k_0$ denote the minimal $k$
satisfying the assumption, we have $t \in \Omega_{k_0-1} \setminus \Omega_k$, hence $f(t) = f_{n_{k_0}}(t) = f_{n_k}(t)$.
We can now check that the subsequence $(f_{n_k})_{k \geq 1}$  converges to $f$, thus proving the convergence of $(f_n)$.
Let $\alpha >0$, and $m$ such that $\alpha \geq 2^{-m}$. Let $k \geq m$ and $t$ such that $f_{n_k}(t) \neq f(t)$. Then there exists $u \geq k$
such that $f_{n_k}(t) \neq f_{n_u}(t)$, and therefore $t \in \Omega_k \subset \Omega_m$. This proves $d(f_{n_k},f) \leq \mu(\Omega_m) \leq 2^{-m} \leq \alpha$.
and thus the claim.
\end{proof}

{\bf Example : $G = \Z/2\Z$.} In the case $G = \Z/2\Z$, elements $f \in L(G)$ are in 1-1 correspondance
with the measurable subsets $A \subset \Omega$ up to sets of measure $0$ under the correspondance $f(t) = \overline{1}$ iff $t \in A$.
Therefore, elements of $B(G)$ are in 1-1 correspondance with the unordered pairs of the form $\{ A, \Omega \setminus A \}$
up to the inherited equivalence relation. Note that $d(\overline{f},\overline{e}) = \min (\mu(A), 1 - \mu(A) \}$.

\bigskip

\begin{proposition} \label{prop:imagedenombrable} Let $D$ be a discrete metric space,
$\Omega_0$ a Borel subset of $\Omega$, and $f : \Omega_0 \to D$
a Borel map. Then the image of $f$ is countable.
\end{proposition}
\begin{proof}
Without loss of generality we can assume that $\Omega_0 = \Omega$ (up to extending $f$
by a constant), and that $f(\Omega) = D$ (up to replacing $D$ by the image of $f$).
This implies $|D| \leq |\Omega| = \mathfrak{c}$ (the cardinality of the continuum). We choose an isomorphism and identify
$\Omega$ with $[0,1]$.
Let $\iota : D \into [0,1]$
be a set-theoretic embedding. Since $D$ is discrete,
it is a Borel map, and therefore so is the function $g = \iota \circ f: [0,1] \to \R$.
It follows that its graph $\{ (t,g(t)) ; t \in [0,1] \}$ is Borel inside $[0,1] \times [0,1]$,
and therefore its projection $g([0,1])$ on the second factor is analytic in the
sense of Suslin. This implies that either $g(\Omega) = g([0,1]) \sim D $ is countable, which proves our
claim, or it has cardinality $\mathfrak{c}$ (this argument was communicated to
me by G. Godefroy). It remains to prove that $D$ cannot have cardinality $\mathfrak{c}$.
Otherwise, we would have an injective map $E \mapsto f^{-1}(E)$ from the set of all subsets
of $D$ into the set of Borel subsets of $\Omega$. Since the latter set is well-known to have
cardinality $\mathfrak{c}$, this would imply $2^{\mathfrak{c}} \leq \mathfrak{c}$, a contradiction.

\end{proof}

An immediate consequence of this proposition is the following.

\begin{corollary} \label{cor:propstar} Every $f : \Omega \to D$ as in the proposition satisfies the property
$$
(\star) \sum_{d \in D} \mu(f^{-1}(\{ d \})) = 1.
$$
\end{corollary}

This property $(\star)$ will be ubiquituous whenever $D$ is not assumed to be countable. A first example is the following.

\begin{proposition} Let $f \in L(G)$, and $x \in B(G) = L(G)/G$. If $d(\overline{f},x)<1$, then the set of elements $f' \in x$
such that $d(f,f') = d(\overline{f},x)$ is non-empty and finite.
\end{proposition}
\begin{proof}
Since $d(f,f') = d(\tilde{e},f^{-1}f')$, we can assume that $f = \tilde{e}$. So the statement is equivalent
to saying that, for every $f \in L(G)$, there exists a finite number of elements $g_{\infty} \in G$
such that $d(\tilde{g}_{\infty}, f) = \inf_{g \in G} d(\tilde{g},f)$. 
For every $g \in G$, let $\mathcal{B}_g = f^{-1}(\{g \})$. Since $\inf_{g \in G} d(\tilde{g},f) < 1$,
there exists $g \in G$ with $\mu(\mathcal{B}_g)>0$.
Since $1 = \mu(I) = \sum_{g \in G} \mu(\mathcal{B}_g)$ by $(\star)$, there is a finite number of elements
$g_1,\dots,g_m \in G$ such that $\mu(\mathcal{B}_{g_i})$ are maximal, that is to say there
is a finite number of elements $g_1,\dots,g_m \in G$ such that $d(\tilde{g}_i,f) = 
\inf_{g \in G} d(\tilde{g},f)$.
\end{proof}

Another consequence of proposition \ref{prop:imagedenombrable} is the following one.
\begin{proposition} If $G_0$ is a subgroup of the discrete group $G$, then $L(\Omega,G_0)/G_0$
embeds into $L(\Omega,G)/G$ as a closed subgroup. Moreover, $L(\Omega,G)/G$
is the union of the corresponding subgroups $L(\Omega,G_0)/G_0$ where $G_0$ runs among the countable
subgroups of $G$.
\end{proposition}
\begin{proof}
If $G_0 < G$ then the composite of the natural maps
$L(G_0) \to L(G) \to L(G)/G$ 
has kernel $G \cap L(G_0) = G_0$, and thus
factorizes injectively through a continuous map $\iota_0 : L(G_0)/G_0 \into L(G)/G$. 
Let $f \in L(\Omega,G)$. Since $f(\Omega)\subset G$ is countable, then $G_0 = \langle f(\Omega) \rangle$ is countable,
and $f \in L(\Omega,G_0) \subset L(\Omega,G)$. It follows that $L(G)$ is the union of the $L(G_0)$ of the
$G_0 < G$ for $G_0$ countable. Therefore $L(G)/G$ is the union of the $\Imm \iota_0$ for $G_0 < G$
countable. It remains to prove that $\Imm \iota_0$ is closed.

Let us assume there is a sequence $(f_n)_{n \in \N}$ such that $(\bar{f}_n) \to \bar{f}$ inside $L(\Omega,G)/G$,
and $f_n \in L(\Omega,G_0)$, with $\bar{f} \not\in \Imm \iota_0$. Let $f \in L(\Omega,G)$ be a representative of $\bar{f}$. Since $\bar{f} \not\in \Imm \iota_0$
there exists $a \in G \setminus G_0$ such that $\mu (f^{-1} \{ a \}) > 0$. But $\overline{a^{-1}.f} = \overline{f} \not\in \Imm \iota_0$
hence there exists $b \neq a$ with $b \not\in G_0$ such that $\mu (f^{-1} \{ b \}) > 0$. Let $\delta = \min (\mu (f^{-1} \{ a \}), \mu (f^{-1} \{ b \})) > 0$.
For every $h \in L(\Omega,G_0)$ and $g \in G$ we have $d(g.f,h) \geq \delta > 0$, hence $d(\overline{f},\overline{h}) \geq \delta > 0$,
which contradicts the assumption.
\end{proof}

For a given cardinal $\alpha$, let us say that a topological space $E$ is
$\alpha$-separable if it admits a dense subset of cardinality (at most) $\alpha$.

We first notice that, if $\alpha$ is infinite and $F \subset E$ is a subspace of an $\alpha$-separable
metric space $E$, then $F$ is also $\alpha$-separable. Indeed, if $X \subset E$ is dense of
cardinality $\alpha$, let us first set $F_n(x) = \{ y \in F \ | \ d(x,y) < 1/n \}$ ; then build
a set $X_n$ by choosing one element of $F_n(x)$ for each $x \in X$, whenever $F_n(x) \neq \emptyset$. We have $|X_n| \leq |X|$, hence $|Y| = |X| = \alpha$ for $Y = \bigcup_n X_n$. Since
$Y$ is clearly dense in $F$ this proves that $F$ is also $\alpha$-separable.

\begin{lemma} \label{lem:alphaseparable} If $D$ is infinite, then $L(\Omega,D)$ is $|D|$-separable, but not $\alpha$-separable
for $\alpha < |D|$. If $G$ is an infinite discrete metric group, then $L(G)/G$
is $|G|$-separable, but not $\alpha$-separable for $\alpha < |G|$.
\end{lemma}
\begin{proof}
Let $E$ be the set of (equivalence classes of) staircase maps in $L(\Omega,D)$. Because
of proposition \ref{prop:imagedenombrable} we know that $E$ is dense inside $L(\Omega,D)$. It
remains to prove that $E$ is $|D|$-separable to prove that $L(\Omega,D)$ is so. The space $E$
is the union of the subsets $\{ f \in E \ | \ f(\Omega) = F \}$ where $F$ runs among the collection $\mathcal{P}_f(D)$ of
finite subsets of $D$, which has the same cardinality of $D$ since $D$ is infinite. Since $D \times D \sim D$, we thus only need to prove that, for a finite set $F$, the space $L(\Omega,F)$
is $|D|$-separable. This holds true because it is separable (and $|D|$ is infinite).

We now choose $\alpha < |D|$, and prove that $L(\Omega,D)$ is not $\alpha$-separable. Let $\mathcal{C}$
denote the subset of (equivalence classes of) constant maps in $L(\Omega,D)$. It is a (closed) subspace of $L(\Omega,D)$, and therefore should be $\alpha$-separable if $L(\Omega,D)$ is so. It is not
the case because, if $x \in D$ is chosen outside the values of a dense subset $(f_j)_{j \in \alpha}$
of $\alpha$-elements of $\mathcal{C}$, and $\tilde{x}$ is the constant map $t \mapsto x$
then $d(x,f_j) = 1$ for all $j \in \alpha$, contradicting the density of $(f_j)_{j \in \alpha}$.

We now prove the statements about $L(G)/G$. Since we have already proved that $L(G)$
is $|G|$-separable, the same holds for its image $L(G)/G$. Let us choose now $\alpha < |G|$,
and fix an identification $\Omega \simeq [0,1]$.
For each $g \in G$, we denote $\check{g} : \Omega = [0,1] \to G$ the map $t \mapsto e$ for $t \leq 1/2$,
$t \mapsto g$ for $t > 1/2$, and $\mathcal{C}'_0 = \{ \check{g} ; g \in G \}$.
We let $\mathcal{C}'$ denote the image of $\mathcal{C}'_0$ inside $L(G)/G$. We claim
that $\mathcal{C}'$ is not $\alpha$-separable. Indeed, if $(\check{g}_j)_{j \in \alpha}$
is a collection of elements of $\mathcal{C}'_0$ with dense image in $\mathcal{C}'$,
by choosing $x \in G \setminus \{ g_j ; j \in \alpha \}$ with $x \neq e$,
we have $d(g\check{x}, g \check{g}_j) \geq 1/2$ for all $g \in G$, contradicting the density.
This concludes the proof.

\end{proof}

Notice that the subspace of $L(\Omega,D)$ made of (classes of) functions with finite images
is dense inside $L(\Omega,D)$. If $D = G$ is a discrete metric group, this subspace
is in addition stable under the action of $G$.

The next proposition has been communicated to me by G. Godefroy. It generalizes proposition \ref{prop:imagedenombrable}.

\begin{proposition} \label{prop:imageseparable} If $E$ is a metrizable space, and $f : \Omega \to E$ is Borel,
then $f(\Omega)$ is separable.
\end{proposition}

\begin{proof}
W.l.o.g. we can assume that $f$ is surjective. We recall the following classical
lemma.
\begin{lemma} If $(E,d)$ is a non-separable metric space, then there exists $r>0$ and $E_1 \subset E$ uncountable such that, for all $x,y \in E_1$ with $x \neq y$, we have $d(x,y) \geq r$.
\end{lemma}
\begin{proof}
For every $n$, let $Y_n$ a subset of $E$ such that for all $x,y \in E_1$ with $x\neq y$, we have $d(x,y) \geq 1/n$,
and which is maximal for this property. Such a $Y_n$ exists by Zorn's lemma. If $Y_n$
was countable for all $n$, then $Y = \bigcup_n Y_n$ would also be countable. Since
it is clearly dense inside $E_1$, this would contradict the non-separability of $E$. This
contradiction proves the existence of an uncountable $Y_{n_0} = E_1$ as in the statement.
\end{proof}
Applying the lemma to $E = f(\Omega)$, we get that, if $E$ is not countable, then there
exists an uncountable $J$ and a collection $(x_j)_{j \in J}$ of elements
in $E$ such that $E' = \bigsqcup_{j} B(x_j, r) \subset E$
for some $r > 0$, where $B(x_j, r) = \{ x \in E ; d(x,x_j)<r \}$. Let us set $\Omega' = f^{-1}(E')$,
and $g : \Omega' \to E'$ the restriction of $f$. It is easy to check that $\Omega'$ is a Borel subset of $E$,
and that $g$ is Borel. Moreover the map $E' \to J$, which associates $x \mapsto j$
for the unique $j$ such that $x \in B(x_j,r)$, and where $J$ is endowed with the
discrete topology, is Borel (because the $B(x_j,r)$ are open).
It follows that the composite map $\Omega' \to J$ is Borel, and surjective. By proposition \ref{prop:imagedenombrable} this proves that $J$ is countable, a contradiction that
proves the claim.

\end{proof}

\begin{corollary} \label{cor:probalindelof} Let $E$ be a metrizable space, $f : \Omega \to E$ a Borel map
and $(U_j)_{j \in J}$  a covering of $E$ by open subsets. Then there exists
a countable $J_0 \subset J$ such that
$$
\mu(f^{-1}(\bigcup_{j \in J_0} U_j))=1
$$
\end{corollary}
\begin{proof} Let $V_j = f(\Omega) \cap U_j$. The collection $(V_j)_{j \in J}$ is a covering of $f(\Omega)$
by open subsets. Since $f(\Omega)$ is separable and metric, it has the Lindel\"of property,
that is there exists a countable $J_0 \subset J$ such that $f(\Omega) = \bigcup_{j \in J_0} V_{j_0}$.
Then
$$
1 = \mu(f^{-1}(\bigcup_{j \in J} U_j)) = 
\mu(f^{-1}(\bigcup_{j \in J} V_j)) = 
\mu(f^{-1}(\bigcup_{j \in J_0} V_j))=\mu(f^{-1}(\bigcup_{j \in J_0} U_j))
$$
and this proves the claim.

\end{proof}

\section{Metric Classifying spaces -- general case and categorical aspects}

Our aim here is to extend the previous construction to the collection
of all metric groups. We first investigate categorical properties
of the functor $L(\Omega,\bullet) : X \leadsto L(\Omega,X)$ on
metric spaces, especially its behavior in terms of limits. We then extend
the construction $L(\Omega,G)/G$ to all metric groups, and establish
that this object is a classifying space for $G$ under the assumption that
the natural projection map $L(\Omega,G) \to L(\Omega,G)/G$
admits a local (continuous) cross-section. We prove that this is \emph{not}
the case when $G$ is infinite profinite, but that it is a Hurewicz fibration. We
prove that it is the case when $G$ is a compact Lie group. Finally, we investigate
the behavior of the functor $G \leadsto L(\Omega,G)/G$ with respect
to limits, and how it might be used to endow the `classifying spaces' associated
to localizations and completions with more natural (geo)metric structures.

\subsection{The functor $L(\Omega, \bullet)$ on metric spaces}
Let $E$ be a metric space. We denote $\mathcal{L}(\Omega,E)$ the
space of Borel maps $\Omega \to E$ such that $\int_{\Omega} d(f(t),c) \dd t < \infty$
for one (and then for all) $c \in E$. We endow it with the pseudo-distance
$d(f,g) = \int_{\Omega} d(f(t),g(t))\dd t$, and denote by $L(\Omega,E)$ the associated
metric space.

Let $\mathbf{Lip}$ denote the category whose objects are the metric spaces with
morphisms Lipschitz maps. If $\varphi \in \Hom_{\mathbf{Lip}}(E_0,E_1)$ with Lipschitz
constant $C > 0$, then one can associate to each $f \in \mathcal{L}(\Omega,E_0)$
a Borel map $\varphi \circ f : \Omega \to E_1$, and we have $\int_{\Omega} d(\varphi \circ f(t),\varphi(c)) \dd t \leq C \int_{\Omega} d(f(t),c) \dd t < \infty$.
This defines a functor $L(\Omega,\bullet) : \mathbf{Lip} \to \mathbf{Lip}$.

We have the following easy lemma.

\begin{lemma} \label{lemnatuprojlipsch} Let $E_0,E_1$ be two metric spaces, $X$ a probability space and $\varphi : E_0 \to E_1$ a C-Lipschitz map. Then the
naturally induced maps $\mathcal{L}(X,E_0) \to \mathcal{L}(X,E_1)$ and $L(X,E_0) \to L(X,E_1)$ are C-Lipschitz.
\end{lemma}
\begin{proof} 
Let $f,g \in \mathcal{L}(X,E_0)$. Then
$$d(\tilde{\varphi}(f),\tilde{\varphi}(g)) = \int_X d(\varphi(f(x)),\varphi(g(x))) \dd x \leqslant C \int_X d(f(x),g(x)) \dd x
= C d(f,g).
$$
\end{proof}

We will say that two metric spaces are equivalent if they are isomorphic inside $\mathbf{Lip}$ (that is,
they are bi-Lipschitz equivalent).  By the following easy lemma the equivalence class of a given compact manifold
does not depend on the choice of Riemannian metric.

\begin{lemma} Let $M$ be a compact differentiable manifold, $g_1,g_2$ two Riemannian metrics on $M$ and
$d_1,d_2$ the associated metrics. Then $(M,d_2)$ and $(M,d_1)$ are equivalent.
\end{lemma}
\begin{proof} 
By definition the tangent space of $M$ is
$T_* M = \{ (x,u) \ | \ u \in T_x M \}$. Let
$E = \{ (x,u) \in T_* M | g_1(x)(u,u) = 1 \}$, and let
us consider $P : E \to \R_+, (x,u) \mapsto g_2(x)(u,u)$.
Since $P$ is continuous and $E$ is compact there exists $C_2 >0$
which is a majorant for $P$.
Let $x,y \in M$ and $\gamma : [0,\delta_1] \to M$
be a minimizing geodesic between $x$ and $y$ for $g_1$, that we assumed
parametrized by arc-length. By definition, $\delta_1 = d_1(x,y)$.
Then, 
$$
d_2(x,y) \leqslant \int_0^{\delta_1} \sqrt{g_2(\gamma(t))(\gamma'(t),\gamma'(t))}\dd t
\leqslant \int_0^{\delta_1} C_2 \dd t = C_2 \delta_1 = C_2 d_1(x,y)
$$
which proves that the identity $(M,\delta_1) \to (M,\delta_2)$ is Lipschitz. By symmetry
this concludes the proof.
\end{proof}

\begin{lemma} \label{lem:LFclosedLE} Let $E$ be a bounded metric space and $F \subset E$ a closed subspace. The
natural inclusion $\mathcal{L}(\Omega,F) \subset \mathcal{L}(\Omega,E)$
induces a closed embedding $L(\Omega,F) \to L(\Omega,E)$. In other words,
it identifies $L(\Omega,F)$ with a closed subset of $L(\Omega,E)$.
\end{lemma}
\begin{proof} We consider a sequence $(f_n)_{n \in \N}$ of elements $f_n \in \mathcal{L}(\Omega,F)$,
and $f \in \mathcal{L}(\Omega,E)$, and we assume $d(f_n,f) \to 0$. By definition
this means $\int d(f_n(x),f(x))\dd x \to 0$. We define $\varphi_n(x) = d(f_n(x),f(x)) \leq \mathrm{diam}(E) < + \infty$.
We have $\varphi_n \in \mathcal{L}(\Omega,\R_+)$ and $\varphi_n \to 0$
in $L^1(\Omega,\R_+)$. Since $\varphi_n \leq \mathrm{diam}(E)$ this implies that
there exists a subsequence $(\varphi_{n_k})_k$ such that $\varphi_{n,k}(x)$
converges to $0$ for $x \in \Omega_1 \subset \Omega$ such
that $\mu(\Omega \setminus \Omega_1) = 0$. Therefore $d(f_{n_k}(x), f(x)) \to 0$
for $x \in \Omega_1$. Since $F$ is closed this implies $\forall x \in \Omega_1\ f(x) \in F$
and therefore the class of $f$ belongs to the image of $L(\Omega,F) \to L(\Omega,E)$
and this proves the claim.

\end{proof}

Another, more surprising category, appears here. We denote it $\mathbf{BLip}$
and call it the category of Lipschitz-or-bounded maps. Its objects are the metric spaces
and a morphism $f : A \to B$ is a uniformly continuous map which is either Lipschitz or
bounded (that is for which $f(A)$ is a bounded subset of $B$).

\begin{lemma} \label{lem:BLipUC} $\mathbf{BLip}$ is a category. $L(\Omega,\bullet)$ defines a functor $\mathbf{BLip} \to \mathbf{BLip}$.
\end{lemma}
\begin{proof} Let $A, B \in Ob(\mathbf{BLip})$. Clearly $\Id_{A} \in \Hom_{\mathbf{BLip}}(A,A)$.
Now, if $f \in \Hom_{\mathbf{BLip}}(B,C)$ and $g \in \Hom_{\mathbf{BLip}}(A,B)$,
then we need to prove that $f \circ g \in \Hom_{\mathbf{BLip}}(A,C)$.
Since $f$ and $g$ are continuous, so is $f \circ g$. If $f$ is bounded, so is $f \circ g$.
Therefore we can assume that $f$ is $C$-Lipschitz for some $C > 0$.
If $g$ is $K$-Lipshitz, then $f\circ g$ is $CK$-Lipschitz, and if $\mathrm{diam}(g(A)) < M$
for some $M >0$, then $d(f\circ g(x), f \circ g (y)) \leq C d(g(x),g(y)) \leq CM$ hence $f \circ g$
is bounded. This proves that $\mathbf{BLip}$ is a category. 

Let $\varphi \in \Hom_{\mathbf{Lip}}(A,B)$. To all $f \in \mathcal{L}(\Omega,A)$
we associate $\varphi \circ f : \Omega \to B$. Since $\varphi$ is continuous
and either Lipschitz or bounded, we have $\varphi \circ f \in \mathcal{L}(\Omega,B)$.
This induces a map $\Phi = L(\Omega,\varphi) : f \mapsto \varphi\circ f$, $L(\Omega,A) \to L(\Omega,B)$.
In case $\varphi$ is $C$-Lipschitz we already checked that $\Phi$ is $C$-Lipschitz too.
If $\varphi$ is bounded, let us choose $\beta \in B$ and consider the constant map $\tilde{\beta} : t \mapsto \beta$
in $L(\Omega, B)$. By assumption there exists $m>0$ such that $d(\varphi(x),\beta) \leq m$ for all $x \in A$.
This implies $d(\Phi(f),\tilde{\beta}) = \int_{\Omega} d(\varphi(f(t)),\beta) \dd t \leq m$ for all $f \in L(\Omega,A)$ whence $\Phi$ is bounded
too.

Let us assume that $\varphi$ is uniformly continuous and bounded, and let $\eps > 0$. Let $\delta =\max(1, \diam(\varphi(A)))$.
There exists $\eta>0$ such that, for all $x,y \in A$, we have $d(x,y)\leq \eta \Rightarrow d(\varphi(x),\varphi(y)) \leq \frac{\eps}{2 \delta}$. 

Now, for all $f,g \in \mathcal{L}(\Omega,A)$, if $d(f,g) \leq \frac{\eps}{2\delta} \eta$,
and since $d(f,g) = \int_{\Omega}d(f(t),g(t))\dd t \geq \eta \mu \{ t ; d(f(t),g(t)) > \eta \}$,
we have $\mu \{ t ; d(f(t),g(t)) > \eta \} \leq \frac{\eps}{2\delta}$. This implies that
$\mu \{ t ; d(\varphi \circ f(t),\varphi\circ g(t)) > \frac{\eps}{2\delta} \} \leq \frac{\eps}{2\delta}$.
Therefore 
$$
\int_{ \{ t ; d(\varphi(f(t)),\varphi(g(t))) \leq \frac{\eps}{2\delta}\}} 
d(\varphi(f(t)),\varphi(g(t))) \dd t \leqslant \frac{\eps}{2 \delta}  \leqslant \frac{\eps}{2}
$$
and
$$
\int_{ \{ t ; d(\varphi(f(t)),\varphi(g(t))) > \frac{\eps}{2\delta}\}} 
d(\varphi(f(t)),\varphi(g(t))) \dd t \leqslant \frac{\eps}{2 \delta} \times \diam(\varphi(A)) \leqslant \frac{\eps}{2}
$$
whence $d(\Phi(f),\Phi(g))  \leq \eps$. This proves that $\Phi$ is uniformly continuous, and
finally that $L(\Omega,\bullet)$ is a well-defined functor.
\end{proof}

Notice that the category $\mathbf{Lip}$ is not good enough for dealing e.g. with
the group of $p$-adic integers. Indeed, there is a natural $1$-parameter
family of metrics on $\Z_p$, defined by $d_{\alpha}(x,x') = \alpha^{v_p(x-x')}$,
where $\alpha \in ]0,1[$, and the identity map $(\Z_p,\alpha) \to (\Z_p, \beta)$
is Lipschitz only if $\beta \leq \alpha$. However it is always Hölderian and
therefore uniformly continuous. Since $\Z_p$ is bounded this proves
that these metric spaces $(\Z_p,d_{\alpha})$ are all isomorphic inside $\mathbf{BLip}$.
Note also that the category made of all metric spaces with morphisms the uniformly continuous
maps would not be convenient as well. Indeed, if $\Z$ is endowed with the discrete metric and
$\R$ with the usual one, then every map $\varphi : \Z \to \R$ is uniformly continuous. We choose
$\varphi = \mathrm{Id} : x \mapsto x$. Then, letting $f : [0,1] \to \Z$ be defined by $f(t) = n^2$ if $\frac{1}{n+1} < t \leq \frac{1}{n}$,
and $f(0) = 0$, then $f$ is Borel and satisfies $\int_0^1 d_{\Z}(f(t),0) \dd t = 1 < \infty$ ;
but $\varphi \circ f$ is also Borel and $\int_0^1d_{\R}(\varphi(f(t)),0)\dd t = \sum_{n=1}^{\infty}\left( \frac{1}{n}-\frac{1}{n+1} \right) n^2 = \infty$.

By restricting the class of metric spaces however it is possible to keep the collection of all uniformly continuous
functions as our morphisms. Indeed, let $\mathbf{Geod}$ denote the category of all geodesic metric spaces
with morphisms the uniformly continuous maps. We then get the following.

\begin{lemma}\label{lem:GeodUC} $L(\Omega,\bullet)$ defines a functor $\mathbf{Geod} \to \mathbf{Geod}$.
\end{lemma}
\begin{proof}
Let $E,F$ be two geodesic metric spaces and $\varphi : E \to F$ a uniformly continuous
function. We choose $e \in E$. The only thing one needs to prove is that, for a Borel map $f : [0,1] \to E$,
if $\int_{\Omega} d(f(t),e)\dd t < \infty$, then
$\int_{\Omega} d(\varphi(f(t)),\varphi(e))\dd t < \infty$.

Note that $x \mapsto d(\varphi(x),\varphi(e))$ is uniformly continuous $E \to \R_+$, since it is a composite
of a uniformly continuous map $\varphi$ with the $1$-Lipschitz map $x \mapsto d(x,\varphi(e))$.
Therefore there exists $b > 0$ such that, whenever $d_E(x,y) \leq b$, then $|d(\varphi(x),\varphi(e)) - d(\varphi(y),\varphi(e))| \leq 1$.

Let us fix $x \in E$, let $\delta = d(x,e)$, and choose
a geodesic $\gamma : [0,\delta] \to E$ from $x$ to $e$. We can assume that it is parametrized by
arclength, that is $d(\gamma(v),\gamma(u)) = length(\gamma_{|[u,v]}) = v-u$. Let $m \in \N$
such that $mb \leq \delta < (m+1) b$, and $t_k = b k/\delta$ for $0 \leq k < m$, $t_{m+1} = 1$.
Then $d(\gamma(t_{k+1}),\gamma(t_k)) = |t_{k+1} - t_k| < b$ and therefore
$d(\varphi(\gamma(t_{k+1})),\varphi(\gamma(t_k))) \leq 1$. 
It follows that
$$
d(\varphi(x),\varphi(e)) =d(\varphi(\gamma(1)),\varphi(\gamma(0))) \leq \sum_{i=0}^{m} d(\varphi(\gamma(t_i)),\varphi(\gamma(t_{i+1}))) < 1+m \leq 1+ \frac{1}{b} d(x,e)
$$
Since this holds for all $x \in E$, we get $\int_{\Omega}d(\varphi(f(t)),\varphi(e))\dd t  \leq 1 + \frac{1}{b} \int_{\Omega} d(f(t),e) \dd t < \infty$
and this proves the claim.
\end{proof}

\begin{proposition} \label{prop:LEcontractcomplete} Let $E$ be a metric space. For $a \in E$ we let $\delta_a  \in \mathcal{L}(\Omega,E)$
denote the constant map $t \mapsto a$.
\begin{enumerate}
\item $L(\Omega,E)$ is contractible, and locally contractible.
\item The map $a \mapsto \delta_a$ embeds $E$ isometrically as a closed subspace of $L(\Omega,E)$.
\item $L(\Omega,E)$ is complete iff $E$ is complete.
\item $L(\Omega,E)$ is separable iff $E$ is separable.
\end{enumerate}
\end{proposition}
\begin{proof}
 We first prove that $L(\Omega,E)$ is contractible, by proving that there is a deformation retract to any
 point $f_0 \in L(\Omega,E)$. Let us first identify $\Omega$ with $[0,1]$. Then, given $f \in L([0,1],E)$, let $f^{|s} \in L([0,1],E)$ for $s \in [0,1]$ be defined by $f^{|s}(t) = f_0(t)$ if $s\geq t$, $f^{|s}(t) = f(t)$
otherwise. We set $H(s,f) = f^{|s}$. Then $H : [0,1] \times L([0,1],E) \to L([0,1],E)$ is the retraction. It is continuous because
$$
d(H(s_1,f_1),H(s_2,f_2)) = \int_{s_1}^{s_2} d(f_1(t),f_2(t)) dt \leq |s_2-s_1| \diam(E).
$$
and preserves a given ball centered at $f_0$ of any given radius. This proves that $L(\Omega,E)$ is contractible and locally contractible.

We now prove that the subspace of $L(\Omega,E)$ made of the constant maps $\delta_a : t \mapsto a$
is closed inside $L(\Omega,E)$. Note that this map is an isometric embedding. If we have a sequence $\delta_{a_n}$ converging to $f \in L(G)$
we need to prove that there exists $a \in G$ with $f = \delta_a$.
For this we consider $\Phi : (t_1,t_2) \mapsto d(f(t_1),f(t_2))$. By Fubini's
theorem and $d(f(t_1),f(t_2)) \leq d(f(t_1),a_n) + d(a_n,f(t_2))$ we get $\int_{\Omega \times \Omega} \Phi \leq 2 d(f,\delta_{a_n})$
for all $n$, and therefore $\int_{\Omega \times \Omega} \Phi = 0$. Since $\Phi$ is non-negative and because by Fubini's theorem we have $0=\int_{t_1} (\int_{t_2} d(
f(t_1),f(t_2)) \mathrm{d} t_2)\mathrm{d} t_1$ this
proves that there exists two measurable subsets $U_1,U_2 \subset \Omega$ each one of measure $1$ such that $d(f(t_1),f(t_2)) = 0$ for all
$(t_1,t_2) \in U_1 \times U_2$. Picking $t_0 \in U_2$ and letting $a = f(t_0)$ we get
$d(f(t_1),a) = 0$ for all $t_1 \in U_1$ and therefore $d(f,\delta_a) = 0$ hence $f = \delta_a$. This proves the claim.

We now prove that $L(\Omega,E)$ is complete iff $E$ is complete. The `only if' part is a consequence of the fact that $E$ can be idenfied with a closed
isometric subspace of $L(\Omega,E)$. We now assume that $E$ is complete and
let $(f_n)_{n \geq 0}$ denote a Cauchy sequence. For all $k \geq 1$, there exists $n_k \in \N$
such that, for all $p,q \geq n_k$, $d(f_p,f_q) \leq 2^{-k}$. If we prove that the subsequence $f_{n_k}$
converges, we are done. Therefore, we may replace the sequence $(f_n)$ by this subsequence
and assume $d(f_p,f_{q}) \leq 2^{-p}$ whenever $q \geq p$. We consider now the element $J_n(t) = \sum_{k=0}^{n-1} d_G(f_k(t),f_{k+1}(t))$.
It is a nondecreasing sequence of elements in $L^1(\Omega,\R)$ and we have 
$$\int_0^1 J_n(t) \mathrm{d} t= \sum_{k=0}^{n-1} d(f_k,f_{k+1}) \leq \sum_{k=0}^{n-1} 2^{-k} \leq 2.
$$
By the monotone convergence theorem it follows that $J_n(t)$ converges almost everywhere, that is
on some $U \subset \Omega$ of measure $1$. It follows that, for all $t \in U$, the series
$\sum_{k=0}^{\infty} d(f_k(t),f_{k+1}(t))$ is convergent and therefore the sequence $f_n(t)$ is
a Cauchy sequence in $E$. Since $E$ is complete this sequence converges to some $f(t)$, and $f$ defines
an element in $L^1(\Omega,G)$ which is the pointwise limit almost everywhere of the sequence $f_n$.

By Fatou's lemma, 
$$
d(f_n,f) = \int_0^1 d(f_n(t),f(t)) \mathrm{d}t \leq \lim_{m \to +\infty} \int_0^1 d(f_n(t),f_m(t))  \mathrm{d}t \leq 2^{-n}.
$$
This proves that $(f_n)$ converges to $f$ in $L(\Omega,E)$ and proves the claim.

Let us assume that $E$ is separable.
It is clear that staircase maps are dense in $L(\Omega,E)$, and therefore so are staircase maps
whose discontinuity points are rationals, and therefore so are staircase maps
whose discontinuity points are rationals with values in some given countable dense
subset of $E$.
The converse implication is clear, for subspaces of separable topological spaces
are separable. This concludes the proof.
\end{proof}

Of course we recover from this proposition the classical fact that $L^1([0,1],\R)$ is complete and separable.

We now consider limits. For this we need to introduce another category. Let $\mathbf{Met}$ denote the
category of metric spaces and metric maps, a.k.a. 1-Lipschitz maps, and $\mathbf{Met_1}$
its full subcategory of metric spaces whose diameter is bounded by $1$. If $\mathbf{Unif}$
denotes the category of metric spaces and uniformly continuous maps, we have the
following inclusions of categories
$$
\mathbf{Met_1} \subset \mathbf{Met} \subset \mathbf{Lip} \subset \mathbf{BLip} \subset \mathbf{Unif}.
$$

\begin{proposition} \label{prop:categoryMet1} The category $\mathbf{Met_1}$ admits arbitrary limits.
\end{proposition}
\begin{proof}
Let $(X_{i},f_{ij})$ be an inverse system of objects and morphisms in $\mathbf{Met_1}$. Each $X_i, i \in I$
is a metric space $(X_i,d_i)$ with $d_i \leq 1$. We define $X$ as a set to be the inverse limit of $(X_{i},f_{ij})$.
Let $\pi_i : X \to X_i$ denote the projection maps. For each $x,y \in X$, let us denote $x_i = \pi_i(x)$, $y_i = \pi_i(y)$.
Then, for all $i \leq j$ we have $x_i = f_{ij}(x_j)$, $y_i = f_{ij}(y_j)$. Since the $f_{ij}$'s are morphisms in $\mathbf{Met_1}$,
they are $1$-Lipschitz, hence $d_i(x_i,y_i) =d_i(f_{ij}(x_j),f_{ij}(y_j)) \leq d_j(x_j,f_j)$. We define $d(x,y) = \sup_i d_i(x_i,y_i)$.
Since $d_i(x_i,y_i)\leq 1$ we have $d(x,y) \leq 1$ for all $i$. Moreover, we have clearly $d(x,y)=d(y,x)$,
and $d(x,y) = 0$ iff $\forall i \in I \ d_i(x_i,y_i) = 0$ iff $\forall i \in I \ x_i = y_i$ iff $x = y$.
Now, for $z \in X$, $z_i = \pi_i(z)$, we have $d_i(x_i,z_i) \leq d(x_i,y_i) + d(y_i,z_i) \leq d(x,y)+d(y,z)$
for all $i \in I$, hence $d(x,z) \leq d(x,y)+d(y,z)$. Thus this endows $X$ with a metric which is bounded by $1$. 
We have $\pi_i \in \Hom_{\mathbf{Met_1}}(X,X_i)$ since $d_i(\pi_i(x),\pi_i(y)) = d_i(x_i,y_i) \leq d(x,y)$.

Let us now consider $(Y,\psi_i)$ with $Y$ an object of $\mathbf{Met_1}$ and $\psi_i \in \Hom_{\mathbf{Met_1}}(Y,X_i)$
with $\psi_i = f_{ij} \circ \psi_j$ for $i \leq j$. By the universal property of the inverse limit in the category $\mathbf{Set}$
of sets we have a unique $u : Y \to X$ such that $\psi_i = \pi_i \circ u$. We need to prove that
$u \in \Hom_{\mathbf{Met_1}}(Y,X)$. We have
$$
d(u(x),u(y)) = \sup_{i \in I} d(\pi_i(u(x)),\pi_i(u(y))) = \sup_{i \in I} d(\psi_i(x)),\psi_i(y)) \leqslant \sup_{i\in I} d(x,y) = d(x,y)
$$
and this proves that $X$ is the inverse limit of the system inside $\mathbf{Met_1}$.
\end{proof}

One caveat in studying limits inside this category originates from the fact that, in the general case, even when $(\varphi_{i})_{i \in I}$
is a direct system of Borel maps $\varphi_i : \Omega \to \R_+$ with $i \leq j \Rightarrow \varphi_i \leq \varphi_j$, one
may have a strict inequality $\sup_{i \in I} \int \varphi_i < \int \sup_i \varphi_i$. A simple example is given by the set $I$ of finite
subsets of $[0,1]$ partially ordered by $\subset$ with $\varphi_i(t) = 1$ if $t \in I$ and $\varphi_i(t) = 0$ otherwise. One has $\sup \int \varphi = 0 < \int \sup \varphi_i = 1$.
The reason why we will be able to circumvent this is given by the following lemma.

\begin{lemma} \label{lem:systproj2suite} Let $(E_i)_{i \in I}$ be a directed system inside $\mathbf{Met_1}$ with (inverse) limit $E = \lim E_i$. We let 
$\pi_i : E \to E_i$
denote the projection maps. Let $f_1,f_2 : \Omega \to E$ be two Borel maps.
There exists $I_0 \subset I$ with an isomorphism $(\N, \leq) \simeq (I_0,\leq)$
given by $n \mapsto u_n$, such that the sequence of functions $t \mapsto d(\pi_{u_n} \circ f_1(t),\pi_{u_n} \circ f_2(t))$ converges
to the function $t \mapsto d(f_1(t),f_2(t))$. Moreover,
$$\sup_i \int d(\pi_{i} \circ f_1(t),\pi_{i} \circ f_2(t)) \dd t = \int \sup_i d(\pi_{i} \circ f_1(t),\pi_{i} \circ f_2(t)) \dd t
= \int  d(f_1(t), f_2(t))\dd t$$
\end{lemma}
\begin{proof}
Since $f_1,f_2$ are Borel we know that $F = f_1(\Omega) \cup f_2(\Omega)$ is separable by proposition
\ref{prop:imageseparable}. Let $(x_k)_{k \in \N}$ be a dense sequence in $F$. By induction on $n \in \N$,
we can construct 
a sequence $(u_n)_{n \in \N}$ in $I$ such that
$\forall r \leq n \ u_{r-1} \leq u_r$ and $\forall r,s \leq n \ |d(x_r,x_s)-d(\pi_{u_n}(x_r),\pi_{u_n}(x_s))| \leq 1/n$.
Indeed, for $n=0$ the two statements are void and, if these statements are assumed to hold for some given $n$,
let us consider the set $\mathfrak{X}$ of couples $(r,s) \in \N^2$ such that $r,s \leq n+1$. It is a finite set. Since
$d(x_r,x_s) = \sup_{i \in I} d(\pi_i(x_r),\pi_i(x_s))$ by construction of the inverse limit,
there exists $a_{r,s} \in I$ such that $|d(x_r,x_s) - d(\pi_{a_{r,s}}(x_r),\pi_{a_{r,s}}(x_s))| \leq 1/(n+1)$.
If $i \geq a_{r,s}$ we have $d(\pi_{a_{r,s}}(x_r),\pi_{a_{r,s}}(x_s)) \leq (\pi_{i}(x_r),\pi_{i}(x_s)) \leq d(x_r,x_s)$
hence $|d(x_r,x_s) - d(\pi_{a_{r,s}}(x_r),\pi_{a_{r,s}}(x_s))| \leq 1/(n+1)$. Since $\mathfrak{X}$ is finite and $I$
is directed there exists $a \geq a_{r,s}$ for all $(r,s) \in \mathfrak{X}$, hence
$|d(x_r,x_s) - d(\pi_a(x_r),\pi_a(x_s))| \leq 1/(n+1)$. Since $I$ is directed there exists $u_{n+1} \in I$
with $u_{n+1} \geq u_n$ and $u_{n+1} \geq a$. For this $u_{n+1}$ the two properties are satisfied and
this proves our claim by induction on $n$.

Now we prove that this sequence satisfies what we need, that is $d(\pi_{u_n} \circ f_1(t),\pi_{u_n} \circ f_2(t)) \to d(f_1(t),f_2(t))$
for all $t \in \Omega$. Let us choose $t \in \Omega$, and $\eps>0$. By density there exists $r,s \in \N$ such that
$d(f_1(t),x_r) \leq \eps/5$ and $d(f_2(t),x_s) \leq \eps/5$. Let us choose $m \geq \max(r,s)$ with $1/m \leq \eps/5$. Then, for all $n \geq m$
we have
$$
\begin{array}{lcl}
0 &\leq& d(f_1(t),f_2(t)) - d(\pi_{u_n} \circ f_1(t),\pi_{u_n} \circ f_2(t)) \\
&\leq &d(f_1(t),x_r) + d(x_r,x_s) + d(x_s,f_2(t)) - d(\pi_{u_n} \circ f_1(t),\pi_{u_n} \circ f_2(t)) \\
&\leq &2 \eps/5 + d(x_r,x_s)  - d(\pi_{u_n} \circ f_1(t),\pi_{u_n} \circ f_2(t)) \\
&\leq &2 \eps/5 + \left( d(x_r,x_s)  -d(\pi_{u_n}(x_r),\pi_{u_n}(x_s))\right)+ d(\pi_{u_n}(x_r),\pi_{u_n}(x_s))-d(\pi_{u_n} \circ f_1(t),\pi_{u_n} \circ f_2(t)) \\
&\leq &3 \eps/5 +  d(\pi_{u_n}(x_r),\pi_{u_n}(x_s))-d(\pi_{u_n} \circ f_1(t),\pi_{u_n} \circ f_2(t)) \\
&\leq &3 \eps/5 +  d(\pi_{u_n}(x_r),\pi_{u_n}(f_1(t))) + 
d(\pi_{u_n}(f_2(t)), \pi_{u_n}(x_s)) \\
&\leq &3 \eps/5 +  d(x_r,f_1(t)) + 
d(f_2(t), x_s)) \\
&\leq &5 \eps/5 = \eps \\
\end{array}
$$
and this proves the claim.

Now, it is clear that $d(f_1(t),f_2(t)) = \sup_i d(\pi_i(f_1(t)),\pi_i(f_2(t))) = \sup_n d(\pi_{u_n}(f_1(t)),\pi_{u_n}(f_2(t))
= \lim_n d(\pi_{u_n}(f_1(t)),\pi_{u_n}(f_2(t))$,
hence the last equality is a consequence of the monotone convergence theorem (or of the dominated convergence theorem).
\end{proof}

\begin{proposition} \label{prop:categoryMet1comm} 
$L(\Omega,\bullet) : \mathbf{Met_1} \to \mathbf{Met_1}$
defines a functor which commutes with directed limits
and finite products.
\end{proposition}
\begin{proof}
We know that $L(\Omega,\bullet)$ defines a functor $\mathbf{Lip} \to \mathbf{Lip}$.
Since $\mathbf{Met_1} \subset \mathbf{Lip}$ it suffices to check that $\diam(L(\Omega,E)) = \diam(E)$
whenever $E$ is bounded, which is clear, and that $L(\Omega,f)$ is $1$-Lipschitz when $f$
is  $1$-Lipschitz, and this is given by lemma \ref{lemnatuprojlipsch}.

We now prove that $L(\Omega,\bullet)$ commutes with the limits we are interested in. Let us again
consider an inverse system $(X_{i},f_{ij})$ of objects and morphisms in $\mathbf{Met_1}$,
and let $X$ be its limit in $\mathbf{Met_1}$, with projection morphisms $\pi_i : X \to X_i$.
The inverse system $(L(\Omega,X_i),L(\Omega,f_{ij}))$ also admits a limit that we denote $L$.
The morphisms $L(\Omega,\pi_i) : L(\Omega,X) \to L(\Omega,X_i)$
induce a morphism $F : L(\Omega,X) \to L$ by the universal property of the inverse limit.
We want to prove that this morphism is an isometry, hence an isomorphism. Indeed,
$$
d(F(\varphi),F(\psi)) = \sup_{i \in I} d_{L(\Omega,X_i)}(L(\Omega,\pi_i)(\varphi), L(\Omega,\pi_i)(\psi))
= \sup_{i \in I} d_{L(\Omega,X_i)}(\pi_i \circ \varphi, \pi_i \circ \psi)
$$ {} $$
= \sup_{i \in I} \int_{\Omega} d(\pi_i(\varphi(t)),\pi_i(\psi(t))) \dd t
=  \int_{\Omega} \sup_{i \in I} d(\pi_i(\varphi(t)),\pi_i(\psi(t))) \dd t
=  \int_{\Omega} d(\varphi(t),\psi(t)) \dd t 
= d(\varphi,\psi)
$$
where the non-trivial step is that $\sup \circ \int = \int \circ \sup$. But this is true in
the case of directed limits by lemma \ref{lem:systproj2suite}, and is clear in the case of finite products,
so this proves the statement.

\end{proof}

We remark that the functor $L(\Omega,\bullet) : \mathbf{Met_1} \to \mathbf{Met_1}$
does \emph{not} commute with infinite products. Indeed, let us consider the discrete metric space
$2 = \{ 0 , 1 \}$ with diameter $1$. Then its $\N$-product $2^{\N}$ inside
$\mathbf{Met_1}$ is again discrete with diameter $1$. Commutation with this product
would imply that the natural map $L(\Omega,2^{\N}) \to L(\Omega,2)^{\N}$ is an isomorphism
(that is a bijective isometry). But if we consider a bijection $f : \Omega \simeq [0,1] \simeq 2^{\N}$
provided by the binary decomposition of real numbers, this map is clearly not in the class of
a Borel map,
since the image $f(\Omega)$ of $f$ in the discrete space $2^{\N}$ is not countable. However,
all its projections are Borel maps $\Omega \to 2$ (and actually step-functions).

\begin{proposition} Inside $\mathbf{Met_1}$, an inverse limit of a collection of complete metric spaces is complete. 
\end{proposition}
\begin{proof}
Let $(X_{i},f_{ij})$ be an inverse system of objects and morphisms in $\mathbf{Met_1}$, 
and $X$ its limit, with projection morphisms $\pi_i : X \to X_i$.
Let $(x_n)_{n \geq 0}$ denote a Cauchy sequence in $X$. Then the relation
$d(\pi_i(x_n),\pi_i(x_m)) \leq d(x_n,x_m)$ implies that, for each $i$, the sequence $(\pi_i(x_n)_{n \geq 0}$
is a Cauchy sequence in $X_i$ and therefore converges to some $\tilde{x}_i \in X_i$.
Using the continuity of the $f_{ij}$'s, by passing to the limit the relation $\pi_i(x_n) = f_{ij}(\pi_j(x_n))$
we get $\tilde{x}_i = f_{ij}(\tilde{x}_j)$ and therefore we get an element $\tilde{x} \in X$
such that $\tilde{x}_i = \pi_i(x)$. It remains to prove that the sequence $x_n$
converges to $\tilde{x}$. We have
$$d(x_n,\tilde{x}) = \sup_{i \in I} d(\pi_i(x_n),\pi_i(\tilde{x}))
= \sup_{i \in I} d(\pi_i(x_n),\tilde{x}_i).
$$
Let $\eps > 0$. Since $(x_n)_{n \geq 0}$ is a Cauchy sequence, there exists $n_0 \in \N$
such that $d(x_n,x_m) \leq \eps$ for all $n,m \geq n_0$. Therefore $d(\pi_i(x_n),\pi_i(x_m))
\leq d(x_n,x_m) \leq \eps$ for all $i \in I$. This implies $d(\pi_i(x_n),\tilde{x}_i) \leq \eps$
for all $n \geq n_0$ and $i \in I$, and thus $d(x_n,\tilde{x}_i) \leq \eps$ for all $n \geq n_0$.
This proves that $x_n$ converges to $\tilde{x}$ in $X$ and that $X$ is complete.

\end{proof}

Finally, we notice the following property.

\begin{lemma} \label{lem:distssespacediscret} If $E$ is a space in $\mathbf{Met_1}$,  $F$ a discrete metric space (of diameter at most but not necessarily equal to $1$),
and $\pi \in \Hom_{\mathbf{Met_1}}(E,F)$, we set $\tilde{\pi} = L(\Omega,\pi) : L(\Omega,E) \to L(\Omega,F)$. Then we have the following.
\begin{enumerate}
\item for every $f \in L(\Omega,E)$, we have $d(\tilde{\pi}(f), F) =d(\tilde{\pi}(f),\pi(E))$.
\item Assume that $\pi$ is surjective. Then, for every $f \in L(\Omega,E)$, $d(f,\tilde{\pi}^{-1}(F)) = d(\tilde{\pi}(f),F)$
\end{enumerate}
\end{lemma}
\begin{proof} We let $N = \mathrm{diam}(F)$. For every $q_1,q_2 \in F$, we have $d(q_1,q_2) = 1/N$ iff $q_1 \neq q_2$.
First note that, since $\pi(E) \subset F$, we have $d(\tilde{\pi}(f), F) \leq d(\tilde{\pi}(f),\pi(E))$.
If we had $d(\tilde{\pi}(f), F) < d(\tilde{\pi}(f),\pi(E))$, we would have $q \in F$ with
$d(\tilde{\pi}(f), q) < d(\tilde{\pi}(f),\pi(E))$. In particular $q \not\in \pi(E)$, and therefore
$d(q,\pi(g)) = 1/N$ for all $g \in E$. It follows that $d(q,\tilde{\pi}(f)) = 1/N$. By corollary \ref{cor:propstar} there exists $g_0 \in E$
such that $\mu(\tilde{\pi}(f)^{-1}(\{ \pi(g_0) \})) = \alpha > 0$. But then $d(\pi(E),\tilde{\pi}(f)) \leq d(\pi(g_0),\tilde{\pi}(f)) \leq 1/N - \alpha < d(q,\tilde{\pi}(f))$,
and this contradiction proves (1).

We now consider (2). Let $q \in F$ and $\tilde{q} \in \pi^{-1}(\{ q \})$. We set $\Omega_1 = \tilde{\pi}(f)^{-1}(\{ q \})= f^{-1}(\pi^{-1}(\{ q \}))$
and $\Omega_2 = \Omega \setminus \Omega_1$. 
We define $g  : \Omega \to E$
by $g(t) = f(t)$ if $t \in \Omega_1$ and $g(t) = \tilde{q}$ otherwise. Since $\Omega_1$ is a Borel set and $f$ is a Borel map, we
get $g \in L(\Omega,E)$. Clearly $\tilde{\pi}(g) = q \in F \subset L(\Omega,F)$
and $d(f,g) = \mu(\Omega_2) = d(\tilde{\pi}(f),q)$. Since $g \in \tilde{\pi}^{-1}(F)$, for all $q \in F$ we get $d(\tilde{\pi}(f),q) \geq d(f,\tilde{\pi}^{-1}(F))$
and therefore $d(\tilde{\pi}(f),F) \geq d(f,\tilde{\pi}^{-1}(F))$. Conversely, for all $q \in F$ we have
$d(\tilde{\pi}(f),F) \leq d(\tilde{\pi}(f),q)$ and $d(\tilde{\pi}(f),q) = d(\tilde{\pi}(f),\tilde{\pi}(g))$ for all $g \in \tilde{\pi}^{-1}(\{ q \})$.
Since $\tilde{\pi}$ is $1$-Lipschitz this implies $d(\tilde{\pi}(f),F) \leq d(f,g)$ for all $g \in \tilde{\pi}^{-1}(F)$ hence
$d(\tilde{\pi}(f),F) \leq d(f,\tilde{\pi}^{-1}(F))$. This proves the claim.

\end{proof}

Note that the discreteness assumption on the metric is needed. Without this assumption, the following counter-example can be constructed.
Let $F = \R^2/(N\Z)^2$ with $N$ large enough (e.g. $N = 10$). We endow $F$ with the induced metric of $\R^2$, rescaled so that $\mathrm{diam}(F) = 1$.
We embed $E = (\Z/N\Z)^2$ into $F$ in the obvious way, and endow it with the induced metric. Then the inclusion $E \subset F$
belongs to $\Hom_{\mathbf{Met_1}}(E,F)$.
 Let $f : \Omega \to E$ such that 
$\mu(f^{-1}(0,0))=
\mu(f^{-1}(0,1))=\mu(f^{-1}(1,0))=\mu(f^{-1}(1,1))= 1/4$. Then $d(\tilde{\pi}(f),(1/2,1/2)) = \sqrt{2}/2N$, 
while $d(\tilde{\pi}(f), g) \geq (2 + \sqrt{2})/4N > \sqrt{2}/2N $ for all $g \in E = \pi(E)$.

\subsection{Possible classifying spaces}
We extend the previous construction to the collection of all metrizable groups. Let $G$ denote such a group,
endowed with a left-invariant distance $d$.
We define $L(G)$ as a metric space as $L(\Omega,G)$ with the notations
as the previous section, namely in the
same way as before if $G$ has bounded diameter, and if not we impose that
the elements $f$ in $L(G)$
satisfy $\int_{\Omega} d(f(t), \tilde{e}) \mathrm{d} t< \infty$, where $\tilde{e}$ is the constant map $\Omega \to G$
with value the neutral element $e \in G$.
It is similarly endowed with a free action of $G$ by isometries.

Applying proposition \ref{prop:LEcontractcomplete} we identify $G$ with a closed subset of $L(G)$, and we know that $L(G)$ is complete (resp. separable) as soon $G$ is complete
(resp. separable).

\begin{proposition} \label{prop:EGtopgroup} If $d$ is bi-invariant,
then $L(G)$ is a topological group
for the composition law $f_1 f_2 : t \mapsto f_1(t) f_2(t)$. This is in particular the case if $G$ is commutative
or if $d$ is the discrete metric.
\end{proposition}
\begin{proof}
If $d$ is invariant on both sides, we have $d(\varphi_2\psi_2,\varphi_1\psi_1)\leq d(\varphi_2\psi_2,\varphi_2\psi_1)+ d(\varphi_2 \psi_1,\varphi_1\psi_1)
= d(\psi_2,\psi_1)+d(\varphi_2,\varphi_1)$ and therefore $(\varphi,\psi) \mapsto \varphi\psi$ is (uniformly) continuous.
Moreover $d(x,y) = d(e,x^{-1}y) = d(y^{-1},x^{-1})$ and thus
the inverse map is an isometry, and therefore is continuous. This proves that
these conditions ensure that $L(G)$ is a topological group. If $G$ is commutative or if $d$ is discrete, this
condition is clearly satisfied.

\end{proof}

In that case, $G$ is a closed subgroup of $L(G)$,
and the isometric map $G \to L(G)$ given by $\delta$ is a group homomorphism. 
The quotient $B_M(G)$ is again a metric space, the distance $d_{B_M(G)}(f_1,f_2)$ being defined as the infimum
of the $d_{L(G)}(\tilde{f}_1,\tilde{f}_2)$ for $\tilde{f}_1,\tilde{f}_2$ running among the representatives of $f_1,f_2$ in $L(G)$. When $G$ is discrete we already know that $B_M(G) = B(G)$ is a classifying space for $G$.

\medskip
{\bf Example : $G = \R$.} We consider the case $G = \R$ endowed with the euclidean metric, that is $L(G)$ is the usual space $L^1(I,\R)$. Let $T : L^1(I,\R) \to L^1(I,\R)$
be defined by $f \mapsto f - \int_0^1f(t) \dd t$. We have 
$$
d(T(f_1),T(f_2)) = \int_0^1 \left| f_1(t) - \int_0^1 f_1(u)\dd u - f_2(t) + \int_0^1 f_2(u) \dd u \right| \dd t
$$
hence 
$$
d(T(f_1),T(f_2)) \leq \int_0^1 | f_1(t) - f_2(t)| \dd t + \int_0^1 \left| \int_0^1 |f_2(u)-f_1(u)| \dd u\right| \dd t = 2 d(f_1,f_2).
$$
It follows that $T$ is $2$-Lipschitz and in particular continuous. Now $T(g.f) = T(f)$ for all $f \in L^1(I,\R)$ and $g \in \R$,
hence $T$ induces a continous section $T : L^1(I,\R)/\R \to L^1(I,\R)$ of the natural projection map -- and actually,
together with the natural injection $\R \to L^1(I,\R)$,
an isomorphism of topological groups (and even of topological vector spaces)
$\R \times L^1(I,\R)/\R \to L^1(I,\R)$. As expected,
$B_M(G)$ is contractible in that case.
\medskip

The relation between a topological group $G$ and the associated discrete group in this construction 
is given by the following.
\begin{proposition} \label{prop:relationGdelta} Let $G$ be a bounded metric group, and $G^{\delta}$ the same group endowed with the discrete
metric. The identity map $G^{\delta} \to G$ induces a continuous map $B(G^{\delta}) \to L(\Omega,G)/G$
which is an algebraic isomorphism of groups if $G$ is abelian.
\end{proposition}
\begin{proof} The identity map $G^{\delta} \to G$ being $\diam(G)$-Lipschitz, by lemma \ref{lemnatuprojlipsch}
we get that the induced map $L(\Omega,G^{\delta}) \to L(\Omega,G)$ is continuous and thus
so is its composite $L(\Omega,G^{\delta}) \to L(\Omega,G)/G$ with the obvious projection map. When $G$ is abelian,
the kernel of this continuous morphism is $G^{\delta}$, and this proves the claim.
\end{proof}

A general fact is that, as soon as $G$ is not trivial, then $L(G)$ is \emph{not} locally compact. Indeed, it is enough to
check that, for every $\eps >0$, the closed ball of radius $\eps$ centered at the constant map $t \mapsto e$ is not compact.
For this, we choose some $g \in G \setminus \{ e \}$, and we set $f_n$ a sequence in $L(G)$ defined by $f(t) = g$ if
$  2 k \eps/2^n  \leq  t \leq (2k+1)\eps/2^n $ for some $0 \leq k \leq 2^{n-1}$, and $f(t) = e$ otherwise.
It is easily checked that the $f_n$ belong to the ball, and that $n \neq m \Rightarrow d(f_n,f_m) = \frac{1}{2}\eps$.
Therefore the sequence has no converging subsequence and $L(G)$ is not locally compact.

Because of that, we do not know in general the answer to the following natural question.
Recall that the existence of a local section to a natural projection $G \to G/H$
usually makes use of a local compactness property of some kind (see e.g. \cite{KARUBE}). 

\begin{question}
Under which condition on $G$ is the natural map $L(G) \to L(G)/G$ a (Hurewicz, Serre, Dold\dots) fibration  ? A quasi-fibration ?
\end{question}

In general, we have the following result.
\begin{proposition} \label{prop:genclassifiant} If $d$ is bi-invariant (in particular if $G$ is discrete or commutative), and if the projection map $L(G) \to L(G)/G$
admits a local cross-section, then $L(G)/G$ is a classifying space for $G$ (in the category of paracompact spaces).
\end{proposition}
\begin{proof} By proposition \ref{prop:EGtopgroup} we know that $L(G)$ is a topological group. Then the
projection map
$L(G) \to L(G)/G$ is a principal bundle in the sense of \cite{HUSEMOLLER} : 
the action of $G$ on $L(G)$ is continuous and free, and the translation
function $(x,g.x) \mapsto g$ is continuous, since it is given by the map $(f_1,f_2) \mapsto f_2 f_1^{-1}$ inside
the topological group $L(G)$. If the bundle $L(G) \to L(G)/G$ is numerable in the sense of Dold \cite{DOLD} (see also \cite{HUSEMOLLER}),
the contractibility of $L(G)$ implies that it is universal (see \cite{DOLD} theorem 7.5),
and therefore $L(G)/G$ is a classifying space for $G$.
Since $L(G)/G$ is paracompact (because it is metrizable) this holds true if and only if the bundle $L(G) \to L(G)/G$
is locally trivial, and this holds true by our assumption.
\end{proof}

\begin{corollary} If $G$ is a compact Lie group endowed with a bi-invariant Riemannian metric, then $L(G)/G$
is a classifying space for $G$.
\end{corollary}
\begin{proof} By the proposition it is sufficient to show the existence of a local cross-section. Since $L(G)$ is metrizable
it is completely regular, and therefore this is a consequence of Gleason's theorem \cite{GLEASON}.
\end{proof}

Unfortunately, this condition is not satisfied in general, as illustrated by the case of infinite profinite groups.

\begin{proposition} \label{prop:profiniteHurewicz} Let $G$ be a profinite group endowed with a bi-invariant
metric. Assume that $G$ contains a sequence $(g_n)$ in $G \setminus \{ 1 \}$ converging to $1$.
Then the continuous projection map $L(G) \to L(G)/G$ does not admit any local cross-section.
However, this map is a Hurewicz fibration, admitting in addition the \emph{unique} homotopy lifting property
with respect to any space.
\end{proposition}
\begin{proof} Assume we have an open neighborhood $U$ of the identity element inside $L(\Omega,G)/G$
and a section $s : U \to L(\Omega,G)$ of the projection map $p$. We can assume that $U$ is the open ball
of center $1$ and radius $\eps>0$. Let $U_0 \subset L(\Omega,G)$ be the open ball of center $1$ and
radius $\eps$. By definition of the induced metric we have $p(U_0) \subset U$, hence $\check{s} : s \circ p : U_0 \to L(\Omega,G)$
is a well-defined continuous map. Then $x \mapsto \check{s}(x)x^{-1}$ is a continuous map $U_0 \to G$.
Since $U_0$ is (arcwise) connected and $G$ is totally discontinuous it is therefore constant, equal to some $g_{\infty} \in G$.
This proves that $\check{s}(x) = g_{\infty} x$ for all $x \in U_0$. In particular, $\check{s}(g) = g_{\infty}g$ for
all $g \in G \cap U_0$. Let $(g_n)$ be a sequence of elements of $G \setminus \{ 1 \}$ converging to $1$.
For some $n_0 \geq 1$ we have $g_n \in U_0$ for all $n \geq n_0$, hence $\check{s}(g_n) = g_{\infty}g_n$ for all $n \geq n_0$.
But $\check{s}(g_n) = s(p(g_n))$ is independant on $n$, hence $g_{\infty}g_n = g_{\infty}g_{n+1}$ and $g_n = g_{n+1}$
for all $n \geq n_0$, contradicting the assumption.

We now prove the unique homotopy lifting property with respect to an arbitrary space $X$. That is,
we assume that there exists $\varphi : X \to L(\Omega,G)$ and $H : X \times [0,1] \to L(\Omega,G)/G$
satisfying the usual conditions ;  then, for every open subgroup $N \vartriangleleft G$,
the natural map $\pi_N : G \to G/N$ induces a $1$-Lipschitz morphism $L(\Omega,G) \to L(\Omega,G/N)$,
which provides a map $\varphi_N : X \to L(\Omega,G/N)$, and its composite with the projection map $L(\Omega,G/N) \to L(\Omega,G/N)/(G/N)$ 
factorizes through $L(\Omega,G)/G$, which provides a map $H_N : X \times [0,1] \to L(\Omega,G/N)/(G/N)$.
Moreover, if $N_2 \vartriangleleft G$ is open and satisfies $N \subset N_2$, we get natural compatibilities illustrated by
the commutation of the following diagram.
$$
\xymatrix{
X \ar[r] \ar[d] & L(\Omega,G) \ar[r] \ar[d] & L(\Omega,G/N) \ar[r]\ar[d]  & L(\Omega,G/N_2) \ar[d] \\
X \times [0,1] \ar[r]_H \ar@{.>}[urr]_{H_N} \ar@{.>}[urrr]_{H_{N_2}} & L(\Omega,G)/G \ar[r]  & L(\Omega,G/N)/(G/N) \ar[r]  & L(\Omega,G/N_2)/(G/N_2) 
}
$$
The plain arrows are the natural maps, and the dotted one are liftings that are \emph{uniquely} defined since
the groups $G/N$ are finite and therefore the projection maps of the form $L(\Omega,G/N) \to L(\Omega,G/N)/(G/N)$
are covering maps. This unicity implies the compatibility of the construction, namely that the diagram remains commutative
when these dotted arrows are added to it. We then can define $H : X \times [0,1] \to L(\Omega,G)$
by $H(x,t) = (H_N(x,t))_{N \in \mathcal{E}}$  where $\mathcal{E}$ is the collection of all the normal open subgroups of $G$.
It is a convenient lifting and it is the only possible one by the above arguments. It remains to prove that $H$
is continuous. But this is an immediate consequence of the continuity of the maps $H_N$, and this proves our claim.

\end{proof}

\begin{corollary} Let $G$ be a metric profinite group endowed with a bi-invariant metric. Then $L(\Omega,G)/G$ is a weak $K(G,1)$.
\end{corollary}
\begin{proof} By the proposition the map $L(\Omega,G) \to L(\Omega,G)/G$ is a fibration
with totally disconnected fiber $G$. Since $L(\Omega,G)$ is contractible, from
the homotopy long exact sequence we get $\pi_k(L(\Omega,G)/G) = 1$ for all $k \geq 2$
and a natural bijective map $\pi_1(L(\Omega,G)/G) \simeq \pi_0(G) \simeq G$
which is easily checked to be a group homomorphism. Finally $L(\Omega,G)/G$
is clearly connected. The natural continuous map $L(\Omega,G^{\delta})/G^{\delta}) \to L(\Omega,G)/G$
of proposition \ref{prop:relationGdelta}
is then easily checked to induce an isomorphism on $\pi_1$. It follows that $L(\Omega,G)/G$
is a weak $K(G,1)$.
\end{proof}

It is an intruiguing question of whether these spaces $L(\Omega,G)/G$ (when $G$ is profinite) have the homotopy type
of a CW-complex, that is whether they are actual $K(G,1)$, homotopically equivalent
to $L(\Omega,G^{\delta})/G^{\delta}$. We leave it open. We just notice, taking the example of $G = \Z_p$,
that $L(\Omega,G)/G$ is \emph{not} semi-locally simply connected. Indeed, if it were so there
would exist $\eps>0$ such that every loop inside the open ball $B$ of center the neutral
element and radius $\eps>0$ would be homotopically trivial inside $L(\Omega,G)/G$.
But choose $g \in G \setminus \{ 1 \}$ such that $d(g,1) < \eps$ and let $\gamma_u : [0,1] \to L([0,1],G)$
be defined by $\gamma_u(t) = 1$ if $t \geq u$ and $\gamma_u(t) = g$
if $t < u$. The map $u \mapsto [\gamma_u]$ defines a loop inside $B$, which is homotopically
non trivial, because it maps to $g$ under the map $\pi_1(L(\Omega,G)/G) \to \pi_0(G) = G$
provided by the long exact sequence associated to the Hurewicz fibration
$L(\Omega,G) \to L(\Omega,G)/G$. As a consequence, $L(\Omega,G)/G$
does not admit any simply connected covering, and in particular cannot be homeomorphic
to $L(\Omega,G^{\delta})/G^{\delta}$.

\subsection{Limits}

We now consider limits. For this we introduce the category $\mathbf{GrMet_1}$ of metric groups
whose diameter is bounded by $1$ with morphisms the group homomorphisms which are $1$-Lipschitz,
together with its full subcategory $\mathbf{GrBMet_1}$ of metric groups whose metric is bi-invariant.
Notice that the category $\mathbf{Gr}$ of groups and group homomorphisms embeds into 
$\mathbf{GrBMet_1}$ as a full category through the functor which endows an abstract group
with the discrete metric. We also consider the full subcategories
$\mathbf{AbMet_1}$ and $\mathbf{Ab}$ of $\mathbf{GrBMet_1}$ and $\mathbf{Gr}$, respectively,
whose objects are the abelian metric groups and abelian groups, respectively.

We say that a system of objects inside $\mathbf{Met_1}$ is \emph{thin} if for all $\eps > 0$
there exists only a finite number of objects $X$ of the system which satisfy $\diam(X) \geq \eps$.
The limit of such a (directed) system is called a thin (directed) limit.

We say that a directed system of objects $(X_i,f_{ij})_{i,j \in I}$ inside $\mathbf{Met_1}$ is \emph{ultrametric} if
for all $i,j \in I$ with $i \leq j$ then for all $x,y \in X_j$ we have $d(f_{ij}(x),f_{ij}(y)) \in \{ 0, d(x,y) \}$. In particular a directed system
inside $\mathbf{Gr} \subset \mathbf{Met_1}$ is always ultrametric. This property has the following consequence. If $X$ is the limit of the system
and the $\pi_i : X \to X_i$ are the projection maps, then for all $x,y \in X$ we have $d(\pi_i(x),\pi_i(y)) \in \{ 0 , d(x,y) \}$. 

Indeed,
since $d(x,y) = \sup_k d(\pi_k(x),\pi_k(y))$ by lemma \ref{lem:systproj2suite}, we have either $d(x,y) = 0$ in which case $0 \leq d(\pi_i(x),\pi_i(y))
\leq d(x,y) = 0$, or $d(x,y) >0$ in which case there exists $j_0 \in I$ such that $d(\pi_{j_0}(x),\pi_{j_0}(y)) = d(x,y) > 0$.
Let now $i \in I$. There exists $k \geq \max(i,j_0)$ since we have a directed system.
By the defining property of an ultrametric system we have $d(\pi_k(x),\pi_k(y)) = d(\pi_{j_0}(x),\pi_{j_0}(y)) = d(x,y)$.
Then $d(\pi_i(x),\pi_i(y)) \in \{ 0, d(\pi_k(x),\pi_k(y))\}=\{ 0, d(x,y)\}$, which proves the claim.

\begin{proposition} The categories $\mathbf{GrMet_1}$, $\mathbf{GrBMet_1}$, $\mathbf{AbMet_1}$ admit arbitrary limits.
$L(\Omega,\bullet)$ defines functors $\mathbf{GrMet_1} \to \mathbf{Met_1}$,
$\mathbf{GrBMet_1} \to \mathbf{GrBMet_1}$, $\mathbf{AbMet_1} \to \mathbf{AbMet_1}$ which commute with directed limits and finite
products.
The association $B_M : G \leadsto L(\Omega,G)/G$ induces functors $\mathbf{GrBMet_1} \to \mathbf{Met_1}$
and $\mathbf{AbMet_1} \to \mathbf{AbMet_1}$ which commute 
\begin{enumerate}
\item with finite products
\item with thin directed limits
\item with ultrametric directed limits whose projection maps are surjective
\item with directed limits where the objets belong to $\mathbf{Gr} \subset \mathbf{GrBMet_1}$ and whose projection maps are surjective.
\end{enumerate}
\end{proposition}
\begin{proof} The existence of limits is a consequence of the informal facts that $\mathbf{GrMet_1} = \mathbf{Met_1} \cap \mathbf{Gr}$
and $\mathbf{AbMet_1} = \mathbf{Met_1} \cap \mathbf{Ab}$ and that limits exist in $\mathbf{Gr}$, $\mathbf{Ab}$, $\mathbf{Met_1}$
and coincide under the forgetful functors to $\mathbf{Set}$. The fact that $L(\Omega,\bullet)$ defines functors $\mathbf{GrMet_1} \to \mathbf{Met_1}$,
$\mathbf{GrBMet_1} \to \mathbf{GrBMet_1}$, $\mathbf{AbMet_1} \to \mathbf{AbMet_1}$ which commute with directed limits and finite
products is also a straightforward consequence of proposition \ref{prop:categoryMet1comm}.

Let $G_1,G_2 \in \mathbf{GrBMet_1}$ and $\varphi \in \Hom_{\mathbf{GrBMet_1}}(G_1,G_2)$. Then the composite of $L(\Omega,\varphi)$
with the projection map $L(\Omega,G_2) \to L(\Omega,G_2)/G_2$
defines a continuous map $F : L(\Omega,G_1) \to L(\Omega,G_2)/G_2$ which factorizes through $\bar{F} : L(\Omega,G_1)/G_1 \to L(\Omega,G_2)/G_2$.
We prove that $\bar{F}$ is $1$-Lipshitz. Indeed, for $f_1,f_2 : \Omega \to G_1$ and $g \in G_1$ we have
$$
d(F(g f_1),F(f_2)) = \inf_{g_2 \in G_2} \int_{\Omega} d(g_2\varphi(g f_1(t)),\varphi( f_2(t))) \dd t
= \inf_{g_2 \in G_2} \int_{\Omega} d(g_2\varphi(g) \varphi(f_1(t)),\varphi( f_2(t))) \dd t$${}$$
= \inf_{g_2 \in G_2} \int_{\Omega} d(g_2 \varphi(f_1(t)),\varphi( f_2(t))) \dd t
\leqslant
 \inf_{g_1 \in G_1} \int_{\Omega} d(\varphi(g_1) \varphi(f_1(t)),\varphi( f_2(t))) \dd t
$${}$$= \inf_{g_1 \in G_1} \int_{\Omega} d(\varphi(g_1f_1(t)),\varphi( f_2(t))) \dd t
\leqslant \inf_{g_1 \in G_1} \int_{\Omega} d(g_1f_1(t),f_2(t)) \dd t
= \inf_{g_1 \in G_1} d(g_1.f_1,f_2)
$$
and this means $d(\bar{F}(\bar{f}_1),\bar{F}(\bar{f}_2)) \leq d(\bar{f}_1,\bar{f}_2)$ for all $\bar{f}_1,\bar{f}_2 \in L(\Omega,G_1)/G_1$. This
proves that $\bar{F}$ is a morphism in $\mathbf{Met_1}$ and therefore we get a functor $B_M : \mathbf{GrBMet_1} \to \mathbf{Met_1}$.
If $G_1,G_2$ are abelian groups, then $\bar{F}$ is clearly a group homomorphism, and thus $B_M$ can be restricted
to a functor $\mathbf{AbMet_1} \to \mathbf{AbMet_1}$.

We now investigate when the functors $B_M : \mathbf{GrBMet_1} \to \mathbf{Met_1}$ and
$B_M : \mathbf{AbMet_1} \to \mathbf{AbMet_1}$ commute with limits. 
Consider an inverse system $(G_{i},\varphi_{ij})$ of objects and morphisms in $\mathbf{GrBMet_1}$,
and let $G$ be its limit in $\mathbf{GrBMet_1}$, with projection morphisms $\pi_i : G \to G_i$. Assume that
this system is either directed or a finite product.
The inverse system $(B_M(G_i),B_M(\varphi_{ij}))$ also admits a limit that we denote $B$.
The morphisms $L(\Omega,\pi_i) : L(\Omega,G) \to L(\Omega,G_i)$ when composed
with the natural map $1$-Lipschitz map $L(\Omega,G_i) \to L(\Omega,G_i)/G_i$
factorize into $1$-Lipschitz maps $B_M(G) \to L(\Omega,G_i)/G_i$ and thus
provide morphisms in $\mathbf{Met_1}$ which
induce a morphism $F : B_M(G) \to B$ by the universal property of the inverse limit.
We need to check whether this morphism is an isometry, hence an isomorphism in $\mathbf{Met_1}$. If we have started with a system
in $\mathbf{AbMet_1}$, $F$ is clearly a group morphism, so if it is an isometry, then it is also an isomorphism in $\mathbf{AbMet_1}$.

For this, we let $\tilde{F} : L(\Omega,G) \to B$ denote the composite of $F$
with the projection map $L(\Omega,G) \to L(\Omega,G)/G = B_M(G)$.
The natural morphisms $\kappa_i : L(\Omega,G_i) \to B_M(G_i)$ provide a morphism
of inverse systems between $(L(\Omega,G_i),L(\Omega,\varphi_{ij}))$ and
$(B_M(G_i),B_M(\varphi_{ij}))$, and from this we get a $\mathbf{Met_1}$-morphism
$\rho : L \to B$ between the inverse limit $L$ of the system $(L(\Omega,G_i),L(\Omega,\varphi_{ij}))$
to $B$. Finally, in the proof of proposition \ref{prop:categoryMet1comm} we constructed a natural
map $\check{F} : L(\Omega,G) \to L$ and proved that it was an isometry. It is straightforward to check
that the natural diagram involving all these maps commutes.
$$
\xymatrix{
L(\Omega,G) \ar[d] \ar[r]^{\simeq} \ar[dr] & L \ar[d] \\
B_M(G) = L(\Omega,G)/G \ar[r] & B
}
$$
Let $\bar{x},\bar{y} \in B_M(G)$
and $x,y$ two representatives in $L(\Omega,G)$.
By definition $d(\bar{x},\bar{y}) = \inf_{g \in G} d(g.x,y)
= \inf_{g \in G} d(g,yx^{-1})$.

By definition $d(\bar{x},\bar{y}) = \inf_{g \in G} d(g.x,y)
= \inf_{g \in G} d(\check{F}(g.x),\check{F}(y))= \inf_{g \in G} d(g.\check{F}(x),\check{F}(y))$
and $d(F(\bar{x}),\bar{F}(\bar{x})(\bar{y}))=d(\tilde{F}(x),\tilde{F}(y))
=d(\rho(\check{F}(x)),\rho(\check{F}(y))).$
Therefore we only need to prove that $\inf_{g \in G} d(g.x,y) = d(\rho(x),\rho(y))$
for all $x,y \in L$. We recall that $\pi_i$ is the natural map $\pi_i : G \to G_i$ and we set
$\tilde{\pi}_i = L(\Omega,\pi_i)$. Then
$$
\inf_{g \in G} d(g.x,y) = \inf_{g \in G} \sup_i d(\tilde{\pi}_i(g.x),\tilde{\pi}_i(y))
= \inf_{g \in G} \sup_i d(\pi_i(g)\tilde{\pi}_i(x),\tilde{\pi}_i(y))
= \inf_{g \in G} \sup_i d(\pi_i(g),\tilde{\pi}_i(yx^{-1}))
$$
and
$$
d(\rho(x),\rho(y)) = \sup_i d(\kappa_i(\tilde{\pi}_i(x)),\kappa_i(\tilde{\pi}_i(y)))
=
\sup_i \inf_{g_i \in G_i} d(g_i\tilde{\pi}_i(x),\tilde{\pi}_i(y))
=
\sup_i \inf_{g_i \in G_i} d(g_i ,\tilde{\pi}_i(yx^{-1})).
$$
Therefore, we may assume $x = 1$, and set
$$A(y) =  \inf_{g \in G} \sup_i d(\pi_i(g),\tilde{\pi}_i(y)) = d(y,G)
\ \mbox{ and } \ B(y) = \sup_i \inf_{g_i \in G_i} d(g_i ,\tilde{\pi}_i(y)) = \sup_{i \in I} d(\tilde{\pi}_i(y),G_i).
$$
First note that $A(y) \geq B(y)$ for all $y \in L$. Indeed, 
$\inf_{g_i \in G_i} d(g_i ,\tilde{\pi}_i(y)) \leq d(\pi_i(g),\tilde{\pi}_i(y))$
for all $g \in G$. Therefore, $B(y) \leq \sup_i d(\pi_i(g),\tilde{\pi}_i(y))$
for all $g \in G$ and thus $B(y) \leq A(y)$.
We prove the following lemma, which proves part (2) of the proposition.

\begin{lemma} Assume that $(G_i)_{i \in I}$ is a thin directed system and $y \in L(\Omega,G)$.
 Then $d(y,G) = \sup_{i \in I} d(\tilde{\pi}_i(y),G_i)$.
\end{lemma}
\begin{proof} First note that, if the directed system $I$ is finite, then $G = G_{\sup I}$ and the assertion is obvious. Without loss of generality we now assume that $I$ is infinite.
We know that $S = \sup_i d(\tilde{\pi}_i(y),G_i) \leq d(y,G) = d$. Assume that contradiction that
$S < d$ and let $\alpha = (d-S)/5 > 0$. Since our system is thin there exists a finite $I_F \subset I$ such that  $\diam(G_i) \leq \alpha$ for all $i \not\in I_F$.
Since $d = d(y,G)$ there exists $g \in G$ such that $|d - d(y,g)| \leq \alpha$. Since $d(y,g) = \sup_{i \in I} d(\tilde{\pi}_i(y),\pi_i(g))$ there exists $i_1 \in I$
such that $|d(y,g) - d(\tilde{\pi}_{i_1}(y),\pi_{i_1}(g))| \leq \alpha$.  We now choose $i_2 \in U$ such that $|S - d(\tilde{\pi}_{i_2}(y),G_{i_2})| \leq \alpha$.
Since $I$ is directed and infinite there exists $i \in I$ with $i > \max(I_F,i_1,i_2)$. Then $\diam(G_i) \leq \alpha$,
and $d(\tilde{\pi}_{i_1}(y),\pi_{i_1}(g)) \leq d(\tilde{\pi}_{i}(y),\pi_{i}(g)) \leq d(y,g)$ hence $|d(y,g) - d(\tilde{\pi}_{i}(y),\pi_{i}(g))| \leq
|d(y,g) - d(\tilde{\pi}_{i_1}(y),\pi_{i_1}(g))| \leq \alpha$. Moreover, if $\varphi : G_i \to G_{i_2}$ denotes the transition map, 
we have $d(\tilde{\pi}_{i_2}(y),G_{i_2}) \leq d(\varphi \circ \pi_i \circ y,\varphi(G_{i})) \leq d(\tilde{\pi}_i(y),G_i) \leq S$
hence $|S - d(\tilde{\pi}_i(y),G_i)| \leq \alpha$ hence
there exists $g_i \in G_i$ such the $|S - d(\tilde{\pi}_i(y),g_i)| \leq 2 \alpha$
whence $\alpha \geq \diam(G_i) \geq d(\pi_i(g),g_i) \geq |d(\tilde{\pi}_i(y),\pi_i(g)) - d(\tilde{\pi}_i(y),g_i)|$
hence
$$
\alpha \geqslant \left| (d-S) + (d(\tilde{\pi}_i(y),\pi_i(g))-d) + (S-d(\tilde{\pi}_i(y),g_i))
\right|.
$$
On the other hand we have $(d-S) + (d(\tilde{\pi}_i(y),\pi_i(g))-d) + (S-d(\tilde{\pi}_i(y),g_i)) \geq 5 \alpha - \alpha - 2 \alpha = 2 \alpha > 0$
hence $\alpha \geq \diam(G_i) \geq 2 \alpha$, a contradiction which proves the claim.
\end{proof}
Part (1) will be proved be a consequence of the following lemma.

\begin{lemma} \label{lem:distprodfinis} Assume $G = \prod_{i \in I} G_i$ with $I$ finite.
Then $d(y,G) = \sup_{i \in I} d(\tilde{\pi}_i(y),G_i)$.
\end{lemma}
\begin{proof}
Again we know that  $S = \sup_i d(\tilde{\pi}_i(y),G_i) \leq d(y,G) = d$. Assuming by contradiction $\alpha = d-S > 0$,
let us consider some $i \in I$. Since $d(\tilde{\pi}_i(y),G_i) \leq S$, there exists $g_i \in G_i$
such that $d(\tilde{\pi}_i(y),g_i) \leq S + \frac{\alpha}{2}$. Let us consider $g = (g_i)_{i \in I} \in G$.
Then $d(y,g) = \sup_{i \in I} d(\tilde{\pi}_i(y),g_i) \leq S + \frac{\alpha}{2} < d$
and this contradiction proves the claim.
\end{proof}

It remains to consider case (3), case (4) being an immediate consequence of case (3).
Let $y : \Omega \to G$ a Borel map.
We set $\psi(t_1,t_2) = d(y(t_1),y(t_2))$.
We assume again, by contradiction, that $d = d(y,G) > S = \sup_i d(\tilde{\pi}(y),G_i)$, and
we set $\alpha = (d-S)/4$.
The map $\psi : \Omega \times \Omega \to \R_+$
is Borel as a composite of Borel maps. In particular $C = \psi^{-1}(\{ 0 \})$
is a Borel subset of $\Omega^2$. Similarly the maps $\psi_i(t_1,t_2) = d(\pi_i\circ y(t_1),\pi_i\circ y(t_2))$
and the sets $C_i = \psi_i^{-1}(\{ 0 \}) \subset \Omega\times \Omega$ are Borel.
Since $\psi_i \leq \psi$ we have $C_i \supset C$ and more generally $i \leq j \Rightarrow C_i \supset C_j$.

We have $\psi(t_1,t_2) = \sup_{i \in I} \psi_i(t_1,t_2)$ hence $C = \bigcap_i C_i$.
By lemma \ref{lem:systproj2suite} we have a sequence $u_n$ inside $I$ such that 
$\psi_{u_n}$ converges to $\psi$, hence $C = \bigcap_n C_{u_n}$, therefore
$\mu(C) = \inf_n \mu(C_{u_n}) = \inf_{i \in I} \mu(C_i)$.

There exists $i_1 \in I$ such that $\mu(C_{i_1} \setminus C) \leq \alpha^2$.
Let $\mathcal{G} = \{ g \in G \ | \ \mu(y^{-1}(\{ g \})) > 0 \}$. Clearly $\mathcal{G}$ is countable.
Let $g \in \mathcal{G}$, and $\Omega_g = y^{-1}(\{ g \})$. Over $\Omega_g \times \Omega_g \subset \Omega \times \Omega$
we have $\psi = 0$ whence $\Omega_g^2 \subset C$. Clearly $g_1 \neq g_2 \Rightarrow \Omega_{g_1}^2 \cap \Omega_{g_2}^2 = \emptyset$
and $C \supset \bigsqcup_{g \in \mathcal{G}} \Omega_g \times \Omega_g$.
It follows that $\sum_{g \in \mathcal{G}} \mu(\Omega_g)^2 \leq 1 < \infty$ hence we can assume $\sum_{g \in \mathcal{G} \setminus \mathcal{G}_0} \mu(\Omega_g)^2 \leq \alpha^2$
for some
$\mathcal{G}_0 \subset \mathcal{G}$ finite.

 If $(a,b) \in C'=C \setminus \bigsqcup_{g \in \mathcal{G}} \Omega_g \times \Omega_g$
let us set $g_0 = y(a) = y(b)$. We have $g_0 \not\in \mathcal{G}$, since $(a,b) \in \Omega_{g_0}^2$.
By definition $\mu(y^{-1}(\{ g_0 \}))=0$ hence $\mu(\{ t ; d(g_0,y(t)) = 0 \})=0$. Let $\Omega' = y^{-1}(G \setminus \mathcal{G})$.
By Fubini we have $\mu(C') = \int_{\Omega'} \mu(\{ t;  d(y(u),y(t))=0 \}) \dd u = \int_{\Omega'} 0 \dd u = 0$.

Defining similarly $\Omega_g^i = \tilde{\pi}(y)^{-1}(\{ g \})$ for $i \in I$ and $g \in G_i$, we also have that $C_i$ is the disjoint union of the $(\Omega_g^i)^2$
for all $g \in \mathcal{G}_i \subset G_i$, up to a set of measure $0$. For all $g \in G$ we have $(\Omega_{\pi_i(g)}^i)^2 \supset \Omega_g^2$ whence,
for a given $g_i \in G_i$ with $i \geq i_1$, we have
$\mu((\Omega_{g_i}^i)^2 \setminus \bigcup_{\pi_i(g)=g_i} \Omega_g^2) \leq \mu(C_i \setminus C) \leq \alpha^2$.
It follows that $\mu((\Omega_{g_i}^i)^2 \setminus C) \leq \alpha^2$. In particular, if $g_i \not\in \pi_i(\mathcal{G})$,
we have $\mu(\Omega_{g_i}^i) \leq \alpha$.

Now, for each $g \in \mathcal{G}_0$, there exists $i_g \in I$ such the $|d(\tilde{\pi}_i(y),\pi_i(g)) - d(y,g)| \leq \alpha$
for all $i \geq i_g$. Since $\mathcal{G}_0$ is finite and $I$ is directed there exists $i_2 \geq i_1$
such that $|d(\tilde{\pi}_i(y),\pi_i(g)) - d(y,g)| \leq \alpha$ for all $i \geq i_2$ and $g \in \mathcal{G}_0$.
Also, we can assume that $i_2$ is large enough so that $\pi_i$ for $i \geq i_2$ is injective
on $\mathcal{G}_0$ : if $a,b \in \mathcal{G}_0$ satisfy $a \neq b$, that is $d(a,b) > 0$,
since $d(a,b) = \sup_i d(\pi_i(a),\pi_i(b))$ there must exist such an $i_2$ such that $d(\pi_i(a),\pi_i(b)) > 0$ for all such couples $(a,b)$ and $i \geq i_2$ by the finiteness
of $\mathcal{G}_0 \times \mathcal{G}_0$.

Let $g \in G$ and $i \geq i_2 \geq i_1$. 
If $g \not\in \mathcal{G}$, then
$d(y,g) = \int_{\Omega} d(y(t),g)\dd t = \int_{\Omega } d(y(t),g)\dd t
 = \int_{\Omega \setminus \Omega_{\pi_i(g)}^i}  d(y(t),g)\dd t + \int_{\Omega_{\pi_i(g)}^i }d(y(t),g)\dd t
= \int_{\Omega \setminus \Omega_{\pi_i(g)}^i}  d(\pi_i(y(t)),\pi_i(g))\dd t + \int_{\Omega_{\pi_i(g)}^i }d(y(t),g)\dd t$
hence
$d(y,g) - d(\tilde{\pi}_i(y),\pi_i(g)) = 
-\int_{\Omega_{\pi_i(g)}^i}  d(\pi_i(y(t)),\pi_i(g))\dd t + \int_{\Omega_{\pi_i(g)}^i }d(y(t),g)\dd t$
and $|d(y,g) - d(\tilde{\pi}_i(y),\pi_i(g))| \leq 2 \mu(\Omega_{\pi_i(g)}^i ) \leq 2 \alpha$ if $i \geq i_1$.

If $g \in \mathcal{G}_0$ we already have $|d(\tilde{\pi}_i(y),\pi_i(g)) - d(y,g)| \leq \alpha$.  If $g \in \mathcal{G} \setminus \mathcal{G}_0$,
we have $\mu(\Omega_g) \leq \alpha$, and
$d(y,g) = \int_{\Omega} d(y(t),g)\dd t = \int_{\Omega \setminus \Omega_g} d(y(t),g)\dd t
 = \int_{\Omega \setminus \Omega_{\pi_i(g)}^i}  d(y(t),g)\dd t + \int_{\Omega_{\pi_i(g)}^i \setminus \Omega_g}d(y(t),g)\dd t
= \int_{\Omega \setminus \Omega_{\pi_i(g)}^i}  d(\pi_i(y(t)),\pi_i(g))\dd t + \int_{\Omega_{\pi_i(g)}^i \setminus \Omega_g}d(y(t),g)\dd t$
hence
$$0 \leqslant d(y,g) - d(\tilde{\pi}_i(y),\pi_i(g)) = 
-\int_{\Omega_{\pi_i(g)}^i}  d(\pi_i(y(t)),\pi_i(g))\dd t + \int_{\Omega_{\pi_i(g)}^i \setminus \Omega_g}d(y(t),g)\dd t$${}$$
=  \int_{\Omega_{\pi_i(g)}^i \setminus \Omega_g}d(y(t),g)\dd t
\leqslant \mu(\Omega_{\pi_i(g)}^i \setminus \Omega_g).$$
If $\pi_i(g) \not\in \pi_i(\mathcal{G}_0)$ then since $\mu((\Omega_{\pi_i(g)}^i)^2 \setminus \bigcup_{\pi_i(h) = \pi_i(g)} \Omega_h^2) \leq \alpha^2$,
we have 
$$
\mu((\Omega_{\pi_i(g)}^i)^2) = \mu((\Omega_{\pi_i(g)}^i)^2 \setminus C)  +  \mu(C \cap \Omega_{\pi_i(g)}^i)^2)
\leqslant \mu(C_i \setminus C) + \sum_{\stackrel{h \in \mathcal{G}}{\pi_i(h)=\pi_i(g)}} \mu(\Omega_h^2)
\leqslant \alpha^2 + \alpha^2
$$
hence $\mu(\Omega_{\pi_i(g)}^i) \leq \sqrt{2} \alpha \leq 2 \alpha$.
Otherwise, there exists $h \in \mathcal{G}_0$ such that $\pi_i(g)  = \pi_i(h)$. Since $\pi_i$ is injective over $\mathcal{G}_0$,
there exists only one such $h$. Therefore $d(y,\pi_i(g)) = d(y,\pi_i(h))$ and $|d(y,h) - d(\pi_i \circ y, \pi_i(g))|= |d(y,h) - d(\pi_i \circ y, h)| \leq \alpha$
by the preceeding arguments.

This proves the following : for all $i \geq i_2$, for all $g \in G$, there exists $h \in G$ such that $|d(y,h)- d(\pi_i\circ y,\pi_i(g))| \leq 2 \alpha$.

But this contradicts $d -S > 0$. Indeed, there exists $g_{i_3} \in G_{i_3}$ and such that $0 \leq S - d(\tilde{\pi}_{i_3}(y),g_{i_3}) \leq \alpha$.
Since $\pi_{i_3}$ is assumed to be surjective, there exists $g \in G$ such that $g_{i_3} = \pi_{i_3}(g)$. Let us choose $i \in I$
with $i \geq i_2$ and $i \geq i_3$. Then $0 \leq S- d(\tilde{\pi}_i(y),\pi_i(g)) \leq S-d(\tilde{\pi}_{i_3}(y),g_{i_3}) \leq \alpha$.
But we proved that there exists $h \in G$ such that 
$|d(y,h)- d(\pi_i\circ y,\pi_i(g))| \leq 2 \alpha$ hence $|S-d(y,h)| \leq |S-d(\tilde{\pi}_i(y),\pi_i(g))| + |d(\tilde{\pi}_i(y),\pi_i(g)) - d(y,h)|
\leq 3 \alpha$. Therefore $d(y,h) \leq S+3\alpha < d$ contradicting $d = \inf_{h \in G} d(y,h)$.
This concludes the proof of the proposition.

\end{proof}

We remark that the condition that the projection maps are surjective is automatically fulfilled
when the index set $I$ is countable and the transition maps are surjective (see e.g. \cite{DOUADY} exercice 2.5.4).

We now give an example where commutation does not hold, even for a finite limit when it is not directed. Consider such a limit as a closed subgroup
of $\prod_{i \in I} G_i$ with $I$ finite. If $y : \Omega \to G$ is Borel, it defines in particular a Borel map $\Omega \to \prod_{i \in I} G_i$.
Then
$$
d_{L(\Omega,G)}(y,G) = d_{L(\Omega,\prod_{i \in I}G_i)}(y,G) \mbox{ and } d_{L(\Omega,\prod_{i \in I} G_i)}(y,\prod_{i \in I} G_i) = \sup_{i \in I}
d(\tilde{\pi}_i(y),G_i)
$$
where the latter equality is a consequence of lemma \ref{lem:distprodfinis}. But in general
we have $d_{L(\Omega,\prod_{i \in I}G_i)}(y,G) = d_{L(\Omega,\prod_{i \in I} G_i)}(y,\prod_{i \in I} G_i)$.

Indeed, let us consider the following example, illustrated in figure \ref{fig:contrexlimit}.
We consider the following system in the category of groups
$$
\xymatrix{
  \Z/2N\Z \ar[dr]  & & \Z/2N\Z \ar[dl]  \\
   & \Z/2\Z \\
}
$$
where the maps $\Z/2N\Z \to \Z/2\Z$ are the reductions modulo $2$. We endow $\Z/2N\Z$ with the euclidean metric, that is the one induced by the usual metric
$d(x,y) = |x-y|$ over $\Z$, renormalized so that $d(\bar{0},\bar{1}) = 1/2N$. We have $\diam(\Z/2N\Z) = \sqrt{2}N/2N \leq 1$. We endow $\Z/2\Z$ with the only metric satisfying $\diam(\Z/2\Z) = 1/2N$. Then the above
diagram provides a system inside $\mathbf{GrBMet_1}$. We denote $G$ its limit, $G_1 = G_2 = \Z/2N\Z$, $G_3 = \Z/2\Z$. The group $G$ is a subgroup of $\prod_i G_i$.
Since
$\pi_3$ is determined by $\pi_1$ (or $\pi_2$), we can consider it as a subgroup of $G_1 \times G_2$, namely the subgroup of couples $(a,b)$
such that $a \equiv b  \mod 2$. We assume $N \gg 0$.
Let $y : \Omega \to G_1 \times G_2$ such that $\mu(f^{-1}(x))=1/4$ for $x \in \{ (1,1),(2,0),(2,2),(3,1) \}$. Clearly $y \in L(\Omega,G)$.
It is easily checked that $d(y,G) = d(y,x)$ for $x \in \{ (1,1),(2,0),(2,2),(3,1) \}$,
and for such an $x$ we have $d(y,x) = ((1/4) \times 2 + (1/2) \times \sqrt{2})/2N
= (1+\sqrt{2})/4N$. Now consider $x_0 = (2,1,0) \in G_1 \times G_2 \times G_3$. Then $d(y,x_0) = ((1/4) \times 4 ) /2N = 2/4N <(1+\sqrt{2})/4N = d(y,G)$.
This proves that $d(y,G) > d(y,\prod_{i} G_i)$ and thus provides a counterexample. Notice that $I$ is a finite system (but of course not a directed system)
and that its  projection maps are surjective.

\medskip

Finally we notice the following property of limits, that we find interesting in its own right.

\begin{proposition}
Let $G$ be the limit of the directed system $(G_i)_{i \in I}$, and $\pi_i : G \onto G_i$ the natural maps, that we assume surjective for all $i \in I$.
Let us set $\tilde{G}_i = \tilde{\pi}_i^{-1}(G_i) \subset L(\Omega,G)$. Then $G \subset L(\Omega,G)$
is equal to the intersection of the $\tilde{G}_i, i \in I$.
\end{proposition}
\begin{proof}
Since $\tilde{\pi}_i(G) = G_i$ we have $G \subset \tilde{G}_i$ for all $i \in I$ hence
$G \subset \bigcap_{i \in I} G_i$. Conversely, let $f \in \bigcap_{i \in I} G_i \subset L(\Omega,G)$. By definition, for
all $i \in I$ there exists $g_i \in G_i$ such that $\tilde{\pi}_i(f) = g_i$. Let us denote $f_{ij} : G_j \to G_i$
the transition maps. By definition, $\pi_i \circ f_{ij} = \pi_j$. Then, for all $i,j$, we have
$f_{ij}(g_j) = L(\Omega,f_{ij})(\tilde{\pi}_j(f)) = (t \mapsto f_{ij}(\pi_j(f(t)))) = (t \mapsto \pi_i(f(t))) = (t \mapsto \tilde{\pi}_i(f)(t)) = g_i$.
Therefore, the collection $g' = (g_i)_{i \in I}$ belongs to $G$. Then, by lemma \ref{lem:systproj2suite} we have 
$$
d(f,g') = \int_{\Omega} d(f(t),g') \dd t = \int_{\Omega} \sup_i d(\pi_i(f(t)),\pi_i(g'))\dd t
= \sup_i \int_{\Omega} d(\pi_i(f(t)),\pi_i(g'))\dd t $${}$$
= \sup_i \int_{\Omega} d(\pi_i(f(t)),g_i)\dd t
= \sup_i d(\tilde{\pi}_i(f),g_i) = \sup_i d(g_i,g_i) = 0
$$
and this proves the claim.
\end{proof}

\bigskip

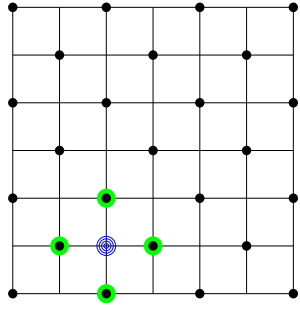
\begin{figure}
\begin{center}
\resizebox{4cm}{4cm}{
\begin{tikzpicture}
\draw (0,0) -- (0,6);
\draw (1,0) -- (1,6);
\draw (2,0) -- (2,6);
\draw (3,0) -- (3,6);
\draw (4,0) -- (4,6);
\draw (5,0) -- (5,6);
\draw (6,0) -- (6,6);
\draw (0,0) -- (6,0);
\draw (0,1) -- (6,1);
\draw (0,2) -- (6,2);
\draw (0,3) -- (6,3);
\draw (0,4) -- (6,4);
\draw (0,5) -- (6,5);
\draw (0,6) -- (6,6);
\fill (0,0) circle (0.1);
\fill (0,2) circle (0.1);
\fill (0,4) circle (0.1);
\fill (0,6) circle (0.1);
\fill (1,1) circle (0.1);
\fill (1,3) circle (0.1);
\fill (1,5) circle (0.1);
\fill (2,0) circle (0.1);
\fill (2,2) circle (0.1);
\fill (2,4) circle (0.1);
\fill (2,6) circle (0.1);
\fill (3,1) circle (0.1);
\fill (3,3) circle (0.1);
\fill (3,5) circle (0.1);
\fill (4,0) circle (0.1);
\fill (4,2) circle (0.1);
\fill (4,4) circle (0.1);
\fill (4,6) circle (0.1);
\fill (5,1) circle (0.1);
\fill (5,3) circle (0.1);
\fill (5,5) circle (0.1);
\fill (6,0) circle (0.1);
\fill (6,2) circle (0.1);
\fill (6,4) circle (0.1);
\fill (6,6) circle (0.1);
\fill[color=green] (2,0) circle (0.2);
\fill[color=green] (2,2) circle (0.2);
\fill[color=green] (1,1) circle (0.2);
\fill[color=green] (3,1) circle (0.2);
\fill[color=black] (3,1) circle (0.1);
\fill[color=black] (1,1) circle (0.1);
\fill[color=black] (2,0) circle (0.1);
\fill[color=black] (2,2) circle (0.1);
\draw[color=blue] (2,1) circle (0.05);
\draw[color=blue] (2,1) circle (0.15);
\draw[color=blue] (2,1) circle (0.1);
\draw[color=blue] (2,1) circle (0.2);

\end{tikzpicture}
}
\end{center}
\caption{A counter-example inside the discrete flat torus $(\Z/2N\Z)^2$ for $N=3$}
\label{fig:contrexlimit}
\end{figure}
\subsection{Two remarks on the construction}

A questionable in this general case is
the choice of $L^1$ instead of $L^p$ for another $p \geq 1$. Indeed, it is easily proved, by combining
the classical Minkowsky inequality with the triangular inequality for the metric in $G$, that $(f,g) \mapsto (\int_I d(f(t),g(t))^p)^{\frac{1}{p}}$
defines a metric on the set $L(G)$. When $1 \leq p < \infty$ it is also true that $L(G)$ is contractible,
the proof being the same.

We know that the usual metric spaces $L^p([0,1],\R)$ are all
homeomorphic for $1 \leq p < \infty$, by the classical result of Stanislaw Mazur
\cite{MAZUR} (see also \cite{BOURBINT}, ch. 4 \S 6 exercice 10), although they are not
uniformly homeomorphic by \cite{ENFLO}. We do not know whether the same thing holds true for an arbitrary
complete metric group $G$. Notice however that, even when $G = \R$, the classical homeomorphism from $L^p(I,\R)$ to $L^q(I,\R)$
given by $f \mapsto |f|^{\frac{p}{q}-1}.f$
is not $G$-invariant, and thus it remains unclear to us whether the $L^p([0,1],\R)$ are isomorphic to each other
as topological $\R$-spaces. While this point remains unsettled, this potentially provides an infinite family
of variations of this construction. For the groups originating from the discrete construction, of
course the choice of $p$ has no relevance.

\smallskip
Another question is about the invariance of this
construction under the symmetries of the probability space.
The group $\Gamma$ of measure-preserving Borel automorphism of $\Omega$
obviously acts on $L(\Omega,G)$, and this action preserves the action of $G$.
Therefore, it induces an action on $L(\Omega,G)/G$. This action factors
through $\overline{\Gamma} = \Gamma/\Gamma_0$ where $\Gamma_0$
is the subgroup of all Borel automorphisms which are almost surely equal to the identity,
and the induced action of $\overline{\Gamma}$ is clearly faithful, as soon as $|G| \geq 2$.
It is known that $\overline{\Gamma}$ is simple (\cite{HARADA}) and \emph{contractible}, by
a theorem of M. Keane \cite{KEANE}. Therefore, there should be no loss in terms of homotopy theory in
dividing out by $\overline{\Gamma}$. The problem however is that the action of $\overline{\Gamma}$
is \emph{not} free, already on $\Omega$, whence also on $L(\Omega,G)$.
Moreover, the action of $G$ on the quotient set $L(\Omega,G)/\Gamma$ is not free
in general : already if $G = \Z/2\Z$, choosing $ \Omega= [0,1]$ for a model, and
letting $\sigma \in \Gamma$ being $x \mapsto 1-x$, we have $\sigma.f = \overline{1}.f$
inside $L(\Omega,G)$ for $f(x) = \overline{0}$ if $x < 1/2$ and $f(x) = \overline{1}$ if $x \geq 1/2$.

The same problem shows up is one is willing to replace the set $L(\Omega,G)$
of `random variables' by their `probability law', namely the induced measure on $G$ -- 
already the case of an amenable group $G$ (e.g. a finite one, or $\Z$) provides
an example where the set of measures on $G$ admits a $G$-invariant subspace.

\subsection{Filtrations and localizations}

In this section we explore some possible metric incarnation of the usual group-theoreric
operations which are classically involved in homotopic localization processes. 

\subsubsection{Extension of metrics}
Let $G$ be a group, $N$ a normal subgroup. Suppose  we are given a bi-invariant metric $d$ on $G/N$
and a bi-invariant metric $\delta$ on $N$. We aim at extending $\delta$ to a bi-invariant metric on $G$.
A sufficient condition is given by the following lemma.

\begin{lemma} \label{lem:extdistancequotient}
Let $G$ be a group, $N$ a normal subgroup, and $\pi : G \to G/N$ the natural projection map.
Assume we are given a biinvariant metric $d$ on $G/N$
and a biinvariant metric $\delta$ on $N$. For $x,y \in G$, define 
$d^+(x,y) = d(\pi(x),\pi(y))$
if $yx^{-1} \not\in N$, and $d^+(x,y) = \delta(yx^{-1},1)$ otherwise. Assume
$\delta(x,y) \leq \inf \{ d(u,v) \ | \ u \neq v \}$. Then $d^+$ defines a bi-invariant metric on $G$,
extending $\delta$. Assume in addition that $\delta$ and $d$ both satisfy the property that equality 
in the triangle inequality
implies $x=y$ or $y=z$. Then, the
same property holds for $d^+$.
\end{lemma}

\begin{proof}
First notice that $d^+(x,y) = d(x,y) >0$ unless $xy^{-1} \in N$. In this case
$d^+(x,y) = \delta(xy^{-1},1) = 0$ iff $xy^{-1} = 0$ iff $x = y$. Therefore $d^+(x,y) = 0 \Rightarrow x = y$.
Moreover, if $g \in G$, since $N$ is normal we have 
$yx^{-1} \in N \Leftrightarrow 
yg(xg)^{-1} \in N \Leftrightarrow 
gy(gx)^{-1} \in N $,
and in this case $\delta(yx^{-1},1) = \delta(yg(xg)^{-1},1) = \delta(gy(gx)^{-1},1)$
by the biinvariance of $\delta$. If $yx^{-1} \not\in N$, then
$d^+(gx,gy) = d(\pi(g) \pi(x) ,\pi(g) \pi(y) ) = d(\pi(x),\pi(y)) = d^+(x,y)$
and 
$d^+(xg,yg) = d( \pi(x) \pi(g), \pi(y)\pi(g) ) = d(\pi(x),\pi(y)) = d^+(x,y)$.
Therefore $d^+$ is biinvariant.

Clearly $d^+(x,y) = d^+(y,x)$ for all $x,y$, since $\delta(yx^{-1},1) = \delta(1,(yx^{-1})) = \delta(1,xy^{-1}) = \delta(xy^{-1},1)$.
 Let us now consider $x,y,z \in G/N$. We prove
the triangle inequality, including the case of equality.

First assume $xz^{-1} \not\in N$.
Then $d^+(x,z) = d(\pi(x),\pi(z))$. We have three cases two consider. If $xy^{-1} \in N$, then necessarily
$yz^{-1} \not\in N$. Thus $d^+(x,y) + d^+(y,z) 
\geq d^+(y,z) = 
d(\pi(y),\pi(z)) = d(\pi(x),\pi(z)) = d^+(x,z)$.
Equality can happen only if $d^+(x,y) = 0$, that is $x = y$.
 The symmetric case $yz^{-1} \in N$ and $xy^{-1} \not\in N$ is similar.
The third case is when neither $xy^{-1}$ nor $yz^{-1}$ belong to $N$, and follows immediatly from
the properties of
$d$.

We now assume $xz^{-1} \in N$. A first case to consider is when $xy^{-1}$ and $yz^{-1}$ both
belong to $N$. If $x=y$ or $y=z$ the statement is clear and therefore we assume otherwise.
Then $d^+(x,z) = \delta(xz^{-1},1)= \delta(xy^{-1} .yz^{-1},1)
= \delta(xy^{-1} ,zy^{-1}) \leq \delta(xy^{-1},1) + \delta(zy^{-1},1) = d^+(x,y)+d^+(y,z)$.
Moreover equality holds iff  $ \delta(xy^{-1} ,zy^{-1}) \leq \delta(xy^{-1},1) + \delta(1,zy^{-1})$. If this implies
$xy^{-1} = 1$ or $1 = zy^{-1}$, then this implies $x=y$ or $y = z$.

A second case is when $xy^{-1} \not\in N$ and $yz^{-1} \not\in N$. 
In other words $\pi(x) \neq \pi(y)$ and $\pi(y) \neq \pi(z)$. Then
$d^+(x,z) = \delta(xz^{-1},1)$ is by assumption no greater than $d^+(x,y) = d(x,y)$
unless $x=y$, and no greater than $d^+(y,z) = d(y,z)$ unless $y=z$.
This proves that $d^+(x,z) \leq d^+(x,y)+d^+(y,z)$ unless $x=y=z$
in which the inequality is still true. Finally equality implies $d^+(y,z) = 0$ or $d^+(x,z) = 0$
that is $x=y$ or $y = z$.
\end{proof}

A useful lemma for rescaling a metric is the following one.

\begin{lemma} \label{lem:rescaledistanceelem}
Let $(X,d)$ be a pseudo-metric space. Let $\la > 0$. Then $d_{\la}(x,y)=\la d(x,y)$ defines a pseudo-metric, which is a metric if $d$ is a metric, and which
defines the same topology as $d$.
Let $d'(x,y) = 0$ if $x=y$ and $d'(x,y) = 1 + \la d(x,y)$ otherwise. Then $d'$ is a metric on $X$, with the additional property that $d'(x,z) = d'(x,y) + d'(y,z)$
implies $x=y$ or $y=z$. The topology of $(X,d')$ is the discrete topology. If $G$ is a (pseudo-)metric group and $d$ is bi-invariant,
then both $d'$ and $d$ are bi-invariant.
\end{lemma}
\begin{proof} The statement on $d_{\la}$ is classical. The fact that $d'(x,y) = d'(y,x)$ and $d'(x,y) = 0 \Rightarrow x=y$ is clear. Clearly
$d'(x,z) = d'(x,y)+d'(y,z)$ if $x=y$ or $y=z$. If not, $d'(x,z) = 1 + \la d(x,z) \leq 1 + \la d(x,y) + \la d(y,z) = d'(x,y)+\la d(y,z) < d'(x,y)+d'(y,z)$ and
this proves the claim about the metric $d'$. It is discrete because $d'(x,y) \leq 1/2 \Rightarrow x=y$. The property of bi-invariance is obvious.
\end{proof}

\subsubsection{Filtrations on metric groups}

Let's say that a pseudo-metric group is a group endowed with a pseudo-metric (that is, $d(x,y) = 0$ does not necessarily
imply $x=y$). This pseudo-metric is said to be bi-invariant if $d(gx,gy)=d(xg,yg)=d(x,y)$ for all $x,y,g \in G$.
We first establish the following elementary proposition.

\begin{proposition} Let $G$ be a group.
\begin{enumerate}
\item If $d$ is a bi-invariant pseudo-metric on $G$, then $\mathcal{K}(d) = \{ g \in G ; d(g,1) = 0 \}$ is a normal subgroup of $G$.
Moreover $d(x,y)$ depends only on the classes of $x$ and $y$ in $G/\mathcal{K}(d)$.
If $d_1,d_2$ are two such pseudo-metrics, $d_1 \leq d_2 \Leftrightarrow \mathcal{K}(d_2) \subset \mathcal{K}(d_1)$.
\item If $N \vartriangleleft G$ and $d$ is a bi-invariant pseudo-metric on $G/N$ with projection map $\pi : G \to G/N$,
then $d(x,y) = d(\pi(x),\pi(y))$ defines a bi-invariant pseudo-metric on $G$ with $N \subset \mathcal{K}(d)$. It is a metric iff $N = \mathcal{K}(d)$.
\item If $d$ is a bi-invariant pseudo-metric on $G$, then $d$ originates from a bi-invariant metric $\bar{d}$ on $G/\mathcal{K}(d)$
by the previous construction. This metric is simply given by $\bar{d}(u,v) = d(x,y)$ for arbitrary $x \in u, y \in v$.
\end{enumerate}
\end{proposition}
\begin{proof}
We first prove that $\mathcal{K}(d)$ is a subgroup.
Indeed, we have 
$g \in \mathcal{K}(d) \Rightarrow d(g^{-1},1) = d(1,g) = 0$,
and $g_1,g_2 \in \mathcal{K}(d) \Rightarrow d(g_1g_2,1) = d(g_1,g_2^{-1}) \leq d(g_1,1) + d(1,g_2^{-1}) = 0+0 = 0$.
It is a normal one because $g \in G , n \in \mathcal{K}(d) \Rightarrow d(gng^{-1},1) = d(n,1) = 0$. 
Moreover, if $x,y \in G$, and $n_1,n_2 \in \mathcal{K}(d)$, we have
$d(xn_1,yn_2) = d(xn_1n_2^{-1},y) = d(n_1n_2^{-1},x^{-1} y)$.
But when $n \in \mathcal{K}(d)$ and $g \in G$ we have
$d(n,g) \leq d(n,1) + d(1,g) = d(1,g)$ et
$d(n,g) \geq | d(n,1) - d(1,g)| = d(1,g)$ whence $d(n,g) = d(1,g)$.
From this we get $d(xn_1,yn_2)= d(1,x^{-1}y) = d(x,y)$, meaning that
$d(x,y)$ only depends on the classes of $x$ and $y$ modulo $\mathcal{K}(d)$.
Finally, if $d_1 \leq d_2$ then $d_2(x,y) = 0 \Rightarrow d_1(x,y) = 0$ and this proves (1).

The proof of (2) is straightforward. Let us prove (3). By (1) we know that
$d$ is induced from a map $\bar{d} : (G/N)\times(G/N) \to \R_+$ with $N = \mathcal{K}(d)$, and it immediate to check that this map
is a bi-invariant pseudo-metric, and then by (2) a bi-invariant metric.

\end{proof}

\begin{definition}
Let $G$ be a group. A sequence $(d_n)_{n \in \N}$ of pseudo-metrics on $G$ is called
weakly increasing if $d_n \leq d_{n+1}$ for all $n \in \N$.
A normal series $G = G_0 > G_1 > \dots$ with $G_i \vartriangleleft$ is called (pseudo-)\emph{metric} if
it is endowed with a collection of bi-invariant (pseudo-)metrics $d_n$ on $G/G_n$ with the property
that the induced pseudo-matrics on $G$ is weakly increasing.
\end{definition}

Note that, to each weakly increasing sequence $(d_n)_{n \in \N}$ of pseudo-metrics on $G$ one can associate
a metric normal series $(\mathcal{K}(d_n), \bar{d}_n)_{n \in \N}$ where $\bar{d}_n$ is the metric
on $G/\mathcal{K}(d_n)$ canonically deduced from $d_n$.

Let $(d_n)_{n \in \N}$ be a weakly increasing sequence $(d_n)_{n \in \N}$ of pseudo-metrics on $G$,
and $G_n$ the associated metric normal series. If it is bounded, we can define $d(x,y) = \lim d_n(x,y) = \sup d_n(x,y)$
which is a bi-invariant pseudo-metric on $G$. The system of metric groups $(G/G_n,d_n)$ actually forms
a projective system inside $\mathbf{GrBMet_1}$, because the elementary transition maps $G/G_{n+1} \to G/G_n$
are $1$-Lipschitz by assumption. Then the metric group $(G,d)$ naturally embeds into the limit of the projective
system of the $(G/G_n,d_n)$ in the category $\mathbf{GrBMet_1}$.

It is a metric iff $\mathcal{K}(d) = \bigcap_n \mathcal{K}(d_n) = \bigcap_n G_n = \{ 0 \}$.
Such a sequence is called \emph{ultrametric} if $d_n(x,y) \in \{ 0, d(x,y) \}$ for all $n \in \N$ and $x,y \in G$.
The corresponding system $(G/G_n,d_n)$ in the category $\mathbf{GrBMet_1}$ is then ultrametric.

Let $G$ be an arbitrary group, and $G = G_0 > G_1 > G_2 \dots$ a normal series. To each $\alpha \in ]0,1[$ we can
define a sequence of pseudo-metrics on $G$ by $d_n(x,y) = \alpha^r$ where $\alpha^r = \inf \{ r \in [0,n[ \ | \ xy^{-1} \in G_r \}$,
Then $\mathcal{K}(d_n) = G_n$, and $d(x,y) =  \alpha^r$ where $\alpha^r = \inf \{ r \in \N \ | \ xy^{-1} \in G_r \}$.

\subsubsection{Localizations and metric refinements}

Let $G$ a group endowed with a normal series $G_n$, and let $(\alpha_n)_{n \in \N}$ be a
decreasing sequence of positive real numbers with $\alpha_0 = 1$. We associate to this a metric
normal series by defining the pseudo-metrics $d_n(x,y) = \alpha_n$ if $xy^{-1} \not\in G_n$
and $d_n(x,y) = 0$ if $xy^{-1} \in G_n$. We call it the ultrametric metric normal series associated
to $(G_n)_n$ and $(\alpha_n)_n$.

A metric refinement of this structure is by definition the data of a collection of bi-invariant metrics $\delta_n$
on the groups $G_n/G_{n+1}$ which are bounded by $1$. We attach to it a collection of metrics
$d_n^+$ on the groups $G/G_n$, defined inductively as follows.

Assuming $d_n^+$ defined, and letting $\pi : G/G_{n+1} \to G/G_n$ the projection map,
we have a short exact sequence 
$$
1 \to G_n/G_{n+1} \to G/G_{n+1} \to G/G_n \to 1
$$
For $x,y \in G/G_{n+1}$, we define $d^+_{n+1}(x,y) = d^+_n(\pi(x),\pi(y))$
if $yx^{-1} \not\in G_n/G_{n+1}$,
and 
{}
$$
d_{n+1}^+(x,y) = \alpha_{n+1} \left( 1 + \frac{\alpha_n- \alpha_{n+1}}{\alpha_{n+1}} \delta_n(yx^{-1},1) \right)
$$

when $yx^{-1} \in (G_n/G_{n+1}) \setminus \{ 1 \}$.

\begin{proposition} $(d_n^+)_{n \in \N}$ is a weakly increasing sequence of pseudo-metrics, whose associated
normal series is $(G_n)_{n \in \N}$. It
satisfies the following properties
\begin{enumerate}
\item $d_n^+(x,y) \neq 0 \Rightarrow d_n^+(x,y) \geq \alpha_n$
\item $d_n^+(x,z) = d_n^+(x,y)+d_n^+(y,z) \Rightarrow x=y \mbox{ or } y=z $
\item The sequence $(d_n^+)_n$ is ultrametric.
\item 
If $d^+ = \sup d_n^+$,
then $d^+$ defines the topology and the uniform structure on $G$ associated to the filtration.
\end{enumerate}
\end{proposition}
\begin{proof}

The fact that $d_n^+$ defines a metric on $G/G_n$ as well as (2) is a consequence
by induction of lemmas \ref{lem:extdistancequotient} and \ref{lem:rescaledistanceelem},
since $(\alpha_n - \alpha_{n+1})/\alpha_{n+1} > 0$,
provided we can prove (1). We prove (1) by induction on $n$,
as well as the fact that $d_n^+(x,y)=0 \Rightarrow xy^{-1} \in  G_n$. 

For $n=0$ we have $d_0^+ = d_0 = 0$ since $G_0=G$ therefore the statement is true for $n=0$.
In general, assuming $d_n^+(x,y) \neq 0  \Rightarrow d_n^+(x,y) \geq \alpha_n$,
we get that, when $d_{n+1}^+(x,y) \neq 0$, by definition
either $d_{n+1}^+(x,y) = d_n^+(\pi(x),\pi(y)) \geq \alpha_n \geq \alpha_{n+1}$,
or $d_{n+1}^+(x,y) =\alpha_{n+1} ( 1 + \frac{\alpha_n- \alpha_{n+1}}{\alpha_{n+1}} \delta_n(yx^{-1},1))  \geq \alpha_{n+1}$
and this proves (1) by induction on $n$.

Similarly we prove $\mathcal{K}(d_n^+) = G_n$ by induction on $n$. When $n=0$ this is clear,
and assuming $\mathcal{K}(d_n) = G_n$ we get that $d_{n+1}^+(x,y) = 0$
implies that, either $xy^{-1} \in G_{n+1}$, or
\begin{itemize}
\item if $xy^{-1} \not\in G_n$ we have $d_{n+1}^+(x,y) = d_n^+(\pi(xy^{-1}),1) = 0$ implies $xy^{-1} \in \mathcal{K}(d_n^+) = G_n$,
a contradiction,
\item and otherwise $d_{n+1}^+(x,y) = \alpha_{n+1} ( 1 + \frac{\alpha_n- \alpha_{n+1}}{\alpha_{n+1}} \delta_n(yx^{-1},1)) \geq \alpha_{n+1} > 0$,
a contradiction.
\end{itemize}
Therefore we get $\mathcal{K}(d_n^+) = G_n$ by induction on $n$.

We prove that the sequence is ultrametric. Let $x,y \in G$ and $r = \inf \{ s ; xy^{-1} \not\in G_s \} \in \N \cup \{ + \infty \}$. We have $d_n^+(x,y) = 0$ when $n < r$,
$d_r^+(x,y) =  \alpha_{n+1} ( 1 + \frac{\alpha_n- \alpha_{n+1}}{\alpha_{n+1}} \delta_n(yx^{-1},1)) > 0$
and $d_n^+(x,y) = d_r^+(x,y)$ when $n \geq r$, whence $d_r^+(x,y) = d^+(x,y)$, which proves (3).

Actually, this proves that
$$
d^+(x,y) =  \alpha_{n+1} \left( 1 + \frac{\alpha_n- \alpha_{n+1}}{\alpha_{n+1}} \delta_n(yx^{-1},1)\right)
$$
when $yx^{-1} \in G_n \setminus G_{n+1}$. Therefore $d^+(x,y) \in [\alpha_{n+1},\alpha_n]$ when $yx^{-1} \in G_n \setminus G_{n+1}$.
Since $\alpha_n$ is (strictly) decreasing this proves that the topology (and the uniform structure) defined by the pseudo-metric $d^+$
is the same as the topology (and the uniform structure) defined by the $G_n$.

\end{proof}

A typical case is the following. Assume we have a normal series with $[G,G_n] \subset G_{n+1}$.
Then the quotients $G_n/G_{n+1}$ are abelian. 
We fix $\alpha \in ]0,1[$. Then, we start from the discrete metric of diameter $\alpha^n$ on $G/G_n$. At least four cases
are potentially interesting.
\begin{itemize}
\item The $G_n/G_{n+1}$ are (submodules of) free $\Z_p$-modules of finite rank. Then any isomorphism $G_n/G_{n+1} \simeq \Z_p^m$ determines
a metric $\delta_n$ by $\delta_n(x_1\oplus \dots \oplus x_m,0) = \max_i(\beta^{v_p(x_i)})$ for some fixed $\beta \in ]0,1[$.
\item The $G_n/G_{n+1}$ are (submodules of) finite-dimensional $\R$-vector spaces. Then any isomorphism $G_n/G_{n+1} \simeq \R^m$ determines
a metric $\delta_n$ by $\delta_n(x_1\oplus \dots \oplus x_m,0) = \max_i(d_{\R}(x_i,0))$ where $d_{\R}(x,y) = \min(1, |x-y|)$.
\item The $G_n/G_{n+1}$ are free $\Z$-modules of finite rank, and we have both possibilities
\item The $G_n/G_{n+1}$ are $\Z_p$-modules of finite rank, not necessarily free. We have isomorphisms $G_n/G_{n+1} \simeq \bigoplus_{i=1}^m \Z_p/p^{a_i}$
with $a_i \in \{ - \infty \} \cup \N$. We can define a metric $\delta_n$ by $\delta_n(x_1\oplus \dots \oplus x_m,0) = \max_i(\beta^{v_p(x_i)})$ for some fixed $\beta \in ]0,1[$,
where $v_p : \Z/p^r \Z \to \{ 0,1 ,\dots, r-1, + \infty \}$ is again the discrete valuation.
\end{itemize}

\section{Quotient fibrations}

In this section we establish properties on the behavior of $G \leadsto L(\Omega,G)$
with respect to quotient maps. In particular, we show that a quotient
map $G \to G/N$ induces a fibration $L(\Omega,G) \to L(\Omega,G/N)$
when $G \to G/N$ admits a local isometric cross-section, and that this
is the case when $G$ is discrete or a compact Lie group.

\medskip

Recall that, if $p : G \to Q$ is a continuous morphism between two topological groups
admitting a continuous cross-section, then it is a Hurewicz fibration. Indeed, if
$\psi : X \to G$, $H : X \times [0,1] \to Q$ are given with $X$ a topological space,
we then define $\tilde{H}(x,u) = \psi(x) s( \overline{\psi(x)}^{-1} H(x,u))$
where $s : Q\to G$ is the given cross-section. Since $L(G)$ and therefore $L(G/N)$ are topological groups, this
is a continuous map providing an homotopy lifting.

In the following proposition we do not assume that $G$ is separable (that is, countable).

\begin{proposition}Let $G$ be a group endowed with the discrete metric, and $N$ be a normal
subgroup of $G$. The natural map $L(G) \to L(G/N)$ is a Hurewicz fibration
admitting a continuous (even 1-Lipschitz) cross-section.
\end{proposition}
\begin{proof} Since the discrete metric is bi-invariant, $L(G)$ and $L(G/N)$ are topological
groups. Let $s : G/N \to G$ be an arbitrary set-theoretic section of the canonical map $G \to G/N$.
It is a continuous map $G/N \to G$ and therefore $\check{s} : \varphi \mapsto s \circ \varphi$
defines a map $L(G/N) \to L(G)$. Now, $\check{s}(\varphi_1)(t) = \check{s}(\varphi_2)(t)$
iff $s(\varphi_1(t)) = s(\varphi_2(t))$, and this holds as soon as $\varphi_1(t) = \varphi_2(t)$.
Therefore, $d(s(\varphi_1(t)), s(\varphi_2(t))) \leq d(\varphi_1(t),\varphi_2(t))$ for all $t \in \Omega$,
whence $d(\check{s}(\varphi_1),\check{s}(\varphi_2)) \leq d(\varphi_1,\varphi_2)$
and $\check{s}$ is 1-Lipschitz. Since $L(G) \to L(G/N)$
is a morphism of topological groups this implies that it is a Hurewicz fibration.
\end{proof}

\begin{proposition}
Let $G$ be a metric group and $N$ a closed normal subgroup of $G$. We assume that the projection map
$G \mapsto G/N$ admits a cross-section which is $C$-Lipschitz, and that $L(G)$ is a topological group for the
obvious multiplication. Then the natural map $L(G) \to L(G/N)$ is a Hurewicz fibration
admitting a $C$-Lipschitz cross-section, and $L(G/N)$ is a topological group for the induced multiplication.
\end{proposition}
\begin{proof}
We let $s$ denote such a cross-section. First note that, up to replacing $s$ by $s':g \mapsto s(\bar{e})^{-1}s(g)$ we
can assume that $s(\bar{e}) = e$ : indeed, $d(s'(x),s'(y)) = d(s(x),s(y))$ because $d$ is left-invariant.
We also notice that $\check{s} : \varphi \mapsto s \circ \varphi$ is then a $C$-Lipschitz cross-section of $L(G) \to L(G/N)$.
Finally, if $\pi : L(G) \to L(G/N)$ is the natural projection, the composition law $(\bar{\varphi},\bar{\psi}) \to \varphi \bar{\psi}$
of $L(G/N)$ can be written $\pi(s(\varphi)s(\bar{\psi}))$, and similarly $\varphi^{-1} = \pi( s(\varphi)^{-1})$ ;
therefore, these two maps are continous and $L(G/N)$ is also a topological group. 

Therefore, it only remains to prove that $L(G) \to L(G/N)$ is a Hurewicz fibration.
Let then $\psi : X \to L(G)$, $H : X \times U \to L(G/N)$ with $U = [0,1]$, $X$ a topological space.
We then define $\tilde{H}(x,u) = \psi(x) \check{s}( \overline{\psi(x)}^{-1} H(x,u))$.
Since $L(G)$ and therefore $L(G/N)$ are topological groups, this
is a continuous map providing an homotopy lifting and this proves the claim.
\end{proof}

In the same spirit, the following lemma will be useful, for covering morphisms inside $\mathbf{BLip}$.

\begin{lemma} \label{lemliftEB}
Let $\pi : E \to B$ be a covering map between two metric spaces which is either
Lipschitz or bounded.
Then the natural map $L(\Omega,E) \to L(\Omega,B)$
is onto.
\end{lemma}
\begin{proof} 
We first prove that such a map exists.
Indeed, if $f : \Omega \to E$ is Borel with $\int d(f(t),x) \dd t < \infty$
for some $x \in E$, then $\pi \circ f : \Omega \to B$ is also Borel (since $\pi$ is continuous),
and $\int d(\pi \circ f(t), \pi(x)) \dd t < K \int d(f(t),x) \dd t < \infty$ if $\pi$ is $K$-Lipshitz,
$\int d(\pi \circ f(t), \pi(x)) \dd t < \diam(\pi(E)) < \infty$ if $\pi$ is bounded.
Let $f \in \mathcal{L}(\Omega,B)$.
Because $\pi$ is a covering, there exists a covering of $B$ by open sets $(U_j)_{j \in J}$ and a collection
of open subsets $(\tilde{U}_j)_{j \in J}$ together with homeomorphisms $\sigma_j : U_j \to \tilde{U}_j$
satisfying $\pi \circ \sigma_j = \Id_{U_j}$. We let $B_j =f^{-1}(U_j)$. By corollary \ref{cor:probalindelof}, there
exists a countable subset $J_0 \subset J$ such that $\mu(\bigcup_{j \in J_0} B_k) = 1$.
Up to removing elements inside $J_0$, we can moreover assume $\mu( \bigcup_{j \in J_1} B_j) < 1$
for every $J_1 \subsetneq J_0$.

We then identify $J_0$ with an initial segment of $\N$, and define $\tilde{f}$ on $\bigcup_{j \in J_0} B_j$ by $\tilde{f}(t) = \sigma_n(f(t))$ for $n = \min \{ m \in J_0 \ | \ x \in B_m \}$,
and extend it by a constant on the neglectable complement $\Omega \setminus \bigcup_{j \in J_0} B_j$.

Letting $\Omega_n = \bigcup_{m \leq n} B_m$ we get that each $\tilde{f}_{|\Omega_n}$ is
obtained by gluing Borel maps along a finite Borel covering, and therefore is Borel. It follows that
$\tilde{f}$ is Borel on $\bigcup_m B_m$, hence also on $\Omega$. Finally, we have $\pi \circ \tilde{f}(t) = \pi(\sigma_n(f(t))) = f(t)$ for almost all $t \in \Omega$, which proves the claim.

\end{proof}

In view of the next result, we need the following elementary lemma.

\begin{lemma} \label{lem:mescupcap} Let $(A_k)$ and $(B_k)$, $1 \leq k \leq N$, be finite sequences of Borel subsets of $\Omega$,
and $\eps >0$. Assume that, for all $k$, we have $\mu(A_k \cap B_k) \geq \mu(A_k) - \frac{\eps}{2^k}$.
Then, for all $n \leq N$,
$$
\mu\left( \bigcup_{k \leq n} (A_k \cap B_k) \right) \geqslant \mu\left( \bigcup_{k \leq n} A_k \right) - \sum_{k=1}^n \frac{\eps}{2^k}.
$$
\end{lemma}
\begin{proof}
By induction on $n$, the case $n=1$ being trivial.
We write 
$ \bigcup_{k \leq n+1} (A_k \cap B_k) =  \left(\bigcup_{k \leq n} (A_k \cap B_k) \right) \sqcup C$,
with $C = (A_{n+1} \cap B_{n+1}) \setminus  \bigcup_{k \leq n} (A_k \cap B_k)$.
Then $C$ contains $(B_{n+1} \cap A_{n+1}) \setminus \bigcup_{k \leq n} A_k$.
But 
$$A_{n+1} \setminus \bigcup_{k \leq n} A_k = 
\left((B_{n+1} \cap A_{n+1}) \setminus \bigcup_{k \leq n} A_k \right) \sqcup \left(A_{n+1} \setminus \left(B_{n+1} \cup \bigcup_{k \leq n} A_k\right)\right)
$$
which implies
$$
\mu(C) \geqslant \mu\left( A_{n+1} \setminus \bigcup_{k \leq n} A_k \right) - \mu \left(A_{n+1} \setminus \left(B_{n+1} \cup \bigcup_{k \leq n} A_k\right)\right) \geqslant  \mu\left( A_{n+1} \setminus \bigcup_{k \leq n} A_k \right) - \mu \left(A_{n+1} \setminus B_{n+1} \right).
$$
Now, $A_{n+1} = (A_{n+1} \cap B_{n+1}) \sqcup (A_{n+1} \setminus B_{n+1})$ and the assumption implies
that $\mu(A_{n+1} \setminus B_{n+1}) \leq \frac{\eps}{2^{n+1}}$. This yields
$$\mu  \left( \bigcup_{k \leq n+1} (A_k \cap B_k)  \right) \geqslant \mu \left(\bigcup_{k \leq n} A_k \right) -\sum_{k \leq n} \frac{\eps}{2^k} - \mu\left( A_{n+1} \setminus \bigcup_{k \leq n} A_k \right)  - \frac{\eps}{2^{n+1}}
$$
and the conclusion follows by induction.
\end{proof}

\begin{proposition} \label{prop:globborelsect} Let $E,B$ be two metric spaces, $B$ being separable\footnote{We do not know whether this separability assumption
is necessary.}. Let $p : E \onto B$ be a continuous (surjective) map admitting local isometric sections, that is,
for all $x \in B$ there exists an open neighborhood $U$ of $x$ and a section $s : U \to E$ of $p$
which is an isometry. Then the naturally induced map $P : L(\Omega,E) \to L(\Omega,B)$
admits a global (continuous) section.
\end{proposition}

\begin{proof}
By assumption, there exists an covering $(V_n)_{n \in \N}$ of $B$ by open balls $V_j = B(x_j,\beta_j)$, $\beta_j > 0$
and isometric sections $s_j : V_j \to E$ of $B$. 
We now define
a section $\Phi$ of $L(\Omega,E) \to L(\Omega,B)$.  For this, we first choose an arbitrary
Borel map $s_{\infty} : B \to E$ (for instance a constant map).
Let $f \in L(\Omega,B)$. For $x \in \Omega$,
we define 
$n(x) = \min \{ j ; f(x) \in V_j \ \& \ \mu(f^{-1}(V_j)) > 0 \}$. 
Clearly, $n$ is well-defined outside a Borel subset $Y(f) = \Omega \setminus \bigcup_{j \in \N} f^{-1}(V_j) $ of $\Omega$ of measure $0$. For $x \in Y(f)$ we set $n(x) = \infty$. Now we can define
$\Phi(f)(x) = s_{n(x)}(f(x))$. Since all the $s_j$ are Borel, so is $\Phi(f)$.We have $P(\Phi(f))(x) = p(\Phi(f)(x)) = f(x)$
for all $x \not\in Y(f)$. Since $Y(f)$ has measure $0$, this proves that $\Phi$ is a global section of 
$P : L(\Omega,E) \to L(\Omega,B)$.
We need to prove that $\Phi$ is continuous.

\begin{lemma} \label{lem:LEBcont}
\begin{enumerate} Let $F$ be a metric space, and $f \in L(\Omega,F)$. Let $y_0 \in F$ and $\beta>0$.
\item  Assume $M = \mu(\{ x; d(f(x),y_0) < \beta \})> 0$. 
For all $\eps>0$, there exists $\eta>0$ such that, for all $g \in L(\Omega,F)$, $d(f,g) \leq \eta$
implies
$$
\left\lbrace \begin{array}{lcl}
\mu(\{ x; d(f(x),y_0) < \beta \ \& \ d(g(x),y_0) < \beta \}) & \geqslant & M(1-\eps) \\
\mu(\{ x; d(f(x),y_0) < \beta \mbox{ or } d(g(x),y_0) < \beta \}) & \leqslant & M(1+\eps) \\
\end{array}
\right.
$$
\item Assume $\mu (\{ x; d(f(x),y_0) \leq \beta \})= 0$.
For all $\eps>0$, there exists $\eta>0$ such that, for all $g \in L(\Omega,F)$, $d(f,g) \leq \eta$
implies $\mu(\{ x; d(g(x),y_0) < \beta \}) \leq \eps$.
\end{enumerate}
\end{lemma}

\begin{proof}
Let $\eps \in ]0,1/4[$.
We first prove (1).
Since $\{ x ; d(f(x) ,y_0) < \beta \} = \bigcup_{n \geq 1}\{ x ; d(f(x),y_0) < \beta - \frac{1}{n} \}$, and
$\{ x ; d(f(x) ,y_0) \leq \beta \} = \bigcap_{n \geq 1}\{ x ; d(f(x),y_0) < \beta + \frac{1}{n} \}$,
there exists $n_0$ such that $\mu( \{ x ; d(f(x), y_0) < \beta - \frac{1}{n_0} \}) \geq M (1 - \frac{\eps}{3})$
and $\mu( \{ x ; d(f(x), y_0) < \beta + \frac{1}{n_0} \}) \leq M (1 + \frac{\eps}{3})$.
On the other hand, for all $\alpha >0$, we have
$$
d(f,g) = \int d(f(x),g(x)) \dd x \geqslant \alpha \mu \left( \{ x ; d(f(x),g(x)) > \alpha \} \right)
$$
hence $\mu( \{ x; d(f(x),g(x)) > \frac{1}{3 n_0} \}) \leq M \frac{3 n_0}{M} d(f,g)$ and therefore there
exists $\eta > 0$
such that  $d(f,g) \leq \eta$ implies $\mu( \{ x; d(f(x),g(x)) > 1/3 n_0 \}) \leq M \frac{\eps}{3}$. But
$\{ x ; d(f(x), y_0) < \beta \} \cap \{ x; d(g(x), y_0) < \beta \}$ contains
$\{ x ; d(f(x), y_0) < \beta - \frac{1}{n_0} \} \setminus \{ x; d(f(x),g(x)) > 1/3n_0\}$, which has measure
at least $M(1- \frac{\eps}{3})-M \frac{\eps}{3} > M(1-\eps)$.
Likewise,  $\{ x ; d(f(x), y_0) < \beta \} \cup \{ x; d(g(x), y_0) < \beta \}$
is included inside $\{ x ; d(f(x), y_0) < \beta + \frac{1}{n_0} \} \cup \{ x; d(f(x),g(x)) > 1/3 n_0 \}$,
which has measure at most $M(1+\eps/3) + M \eps/3 < M(1+\eps)$.
We now prove (2).
Since $\{ x ; d(f(x),y_0) \leq \beta \} = \bigcap_{n \geq 1}\{ x ; d(f(x),y_0) < \beta + 1/n \}$,
there exists $n_0$ such that $\{ x ; d(f(x),y_0) < \beta + 1/n_0 \}$ has measure at most $\eps/2$. Then
if $\eta$ is small enough, $d(f,g) \leq \eta$ implies $\mu (\{ x; d(f(x),g(x) ) \geq 1/n_0 \}) \leq \eps/2$.
Since $\{ x; d(g(x), y_0) < \beta \}$ is contained inside $\{ x ; d(f(x),y_0) < \beta + 1/n_0 \} \cup
\mu \{ x; d(f(x),g(x) ) \geq 1/n_0 \} $, and since this one as measure at most $ \eps/2 + \eps/2 = \eps$, we get (2).

\end{proof}
Let  $\eps>0$, and 
$f_0,f \in L(\Omega,B)$.
For each $n \in \N$ we can choose $\beta'_n<\beta_n$ such that, setting $V'_n = B(x_n, \beta'_n)$,
we have $\mu( f_0^{-1} (V'_{n} )) > 
\mu(f_0^{-1} (V_{n} ))- \frac{\eps}{10.2^n}$ and $\mu( f_0^{-1}(\partial  V'_n)) = 0$.
Up to a set of measure at most $\eps/10$,
$(f_0^{-1}(V'_n))_n$ is a Borel covering of $\Omega$.
We denote $(V'_{n_k})_{k \geq 1}$ the thiner 
subsequence of $(V'_n)$ such that $\mu( f_0^{-1} (V'_{n_k} )) > 0$ for all $k$ ; this means that, if $n_k < r < n_{k+1}$, then $\mu(f_0^{-1}(V'_r)) = 0$. Up to a set of measure at most $\eps/10$,
$(f_0^{-1}(V'_{n_k}))_k$ is again a Borel covering of $\Omega$.

Since $\Omega$ has finite measure,
there exists an integer $K$ such that
$$
\mu\left(\bigcup_{k > K}  f_0^{-1} V'_{n_k}\setminus \bigcup_{k \leq K}  f_0^{-1} V'_{n_k} \right) \leq \eps/10.
$$
Such a $K$ being fixed, by lemma \ref{lem:LEBcont} we can choose $\eta>0$ such that $d(f,f_0) \leq \eta$
implies $\mu (f_0^{-1}(V'_{n_k}) \cap f^{-1}(V'_{n_k}) )\geq \mu(f_0^{-1} V'_{n_k}) - \frac{\eps}{10 . 2^k}$
and $\mu (f_0^{-1}(V'_{n_k}) \cup f^{-1}(V'_{n_k})) \leq \mu(f_0^{-1} V'_{n_k}) + \frac{\eps}{10 . 2^k}$
for all $k \leq K$.

Then, by lemma \ref{lem:mescupcap}, we have $\mu( \bigcup_{k \leq K} (f_0^{-1}(V'_{n_k}) \cap f^{-1}(V'_{n_k})))
\geq \mu(\bigcup_{k \leq K} f_0^{-1}(V_{n_k})) - \frac{\eps}{10} \geq 1 - 3 \eps/10$. 
For $x \in \Omega$, $n(x)$ is well-defined almost surely. Similarly, we define (almost surely) $n'(x)$,
this time with respect with the $f^{-1}(V'_n)$'s. We then prove that $n(x) = n'(x)$ outside a set of small measure.
First of all, up to removing a set of measure at most $3 \eps/10$, we can assume
$x \in  \bigcup_{k \leq K} (f_0^{-1}(V_{n_k}) \cap f^{-1}(V_{n_k}))$,
and moreover that $n(x) \in \{ n_1,\dots, n_K \}$. We assume by contradiction that $n'(x) > n(x) = n_r$. Then $x$
belongs to $f_0^{-1} V_{n_r} \setminus f^{-1} V_{n_r}$,
hence to $f_0^{-1} V_{n_r} \setminus (f_0^{-1} V_{n_r} \cap f^{-1} V_{n_r})$,
which has measure at most $\frac{\eps}{10. 2^r}$. From this we deduce that $n'(x) \leq n(x)$ outside a set of
measure $3 \eps/10 + \eps/10 = 4 \eps/10$. 
By lemma \ref{lem:LEBcont}, we can moreover assume $\eta\leq \eps/10$ small enough,
so that $x \not\in \mu(f_0^{-1}(V_s))$ for all
 $s < n_K$ with $s \not\in \{ n_1,\dots, n_K \}$ outside a set of measure $\eps/10$.
 Therefore, for all $x$ outside a set of measure $4 \eps/10 + \eps/10 = 5 \eps/10$, there exists
 $r,s$ such that $n_s = n'(x) \leq n(x) = n_r $.
If $n_s < n_r$, then $n(x) \in f^{-1} V_{n_r} \setminus f_0^{-1}(V_{n_r})$, which has measure at most $2 \eps/10. 2^{r}$.
From this we deduce that $n(x) = n'(x)$ outside a set $Z_2$ of measure at most $5 \eps/10 + 2 \eps/10 = 7 \eps/10$.

Then
 $$
 d(\Phi(f_0),\Phi(f)) = \int  d(\Phi(f_0)(x),\Phi(f)(x))\dd x \leqslant 7 \eps/10 +
\int_{\Omega \setminus Z_2} d(\Phi(f_0)(x),\Phi(f)(x))\dd x
 $$
that is
 $$
 d(\Phi(f_0),\Phi(f)) \leqslant 7 \eps/10 +
\sum_{1 \leq r \leq K} \int_{\stackrel{x \in \Omega \setminus Z_2}{n(x)=n'x)=r}} d(s_r(f_0(x)),s_r(f(x)))\dd x
$${}$$
\leqslant 7 \eps/10 +
\sum_{1 \leq r \leq K} \int_{\stackrel{x \in \Omega \setminus Z_2}{n(x)=n'x)=r}} d(f_0(x),f(x))\dd x
= 7 \eps/10 +
 \int_{\Omega \setminus Z_2} d(f_0(x),f(x))\dd x
$$
and this is smaller than or equal to
$7 \eps/10 + d(f_0,f) \leqslant 7 \eps/10 + \eps/10 = 8 \eps/10<\eps
 $
 which proves the continuity of $\Phi$.
\end{proof}

This proposition can be immediately applied to covering maps. Since we could not find
an easy reference, we prove this now.

\begin{proposition} \label{prop:existlocalisosectcovering}
Let $E$ be a metric space, $\Gamma$ a group acting by isometries on $E$ such that
the natural projection $p : E \to B= E/\Gamma$ is a covering map. Then $p$ admits local
isometric sections, w.r.t. the induced metric on $B$.
\end{proposition}
\begin{proof}
Let us choose $x \in B$, and $\tilde{x} \in p^{-1}(\{ x \})$. There exists $\beta > 0$ such that the open ball $U$
with center $x$ and radius $\beta$ satisfies $p^{-1}(U) \simeq_{\Phi} U \times F$, for
some discrete space $F$ (actually a necessarily discrete quotient of $\Gamma$). We let $* \in F$
such that $\Phi(\tilde{x}) = (x,*)$. Since $F$ is discrete, letting $V = \Phi^{-1}( U \times \{ * \})$ we
have $\alpha = \inf \{ d(x,y) \ | \ x \in V, y \in p^{-1}(U) \setminus V \} > 0.$
Finally, $V$ being open, there exists $\gamma > 0$ such that $V$ contains the open ball of center $\tilde{x}$
and radius $\gamma$. Let $\delta = \min(\alpha/3, \beta,\gamma)$ and $W$ the open ball of center $\tilde{x}$
and radius $\delta$. Since $\delta \leq \gamma$ we have $W \subset V$. Then the restriction $p_W$
of $p$ to $W$ is an homeomorphism $W \to p(W)$. Let $s_W$ be the converse of $p_W$. We first
prove that $p(W)$ is equal to the open ball $C$ with center $x$ and radius $\delta$. Indeed,
let us choose $y \in p(W)$. Since $d(x,y) \leq d(s_W(x),s_W(y)) = d(\tilde{x}, s_W(y)) < \delta$
 by definition of the induced metric we have $y \in C$, whence $p(W) \subset C$.
 
 Conversely, let $y \in C$. We know $C \subset U$ since $\delta < \beta$. By definition of
 the induced metric, there exists $\check{x} \in p^{-1}(\{ x \})$, $\check{y} \in p^{-1}(\{y \}) \subset p^{-1}(U)$
 such that $d(\check{x},\check{y}) \leq d(x,y) + \alpha/3$. SInce $\Gamma$ is acting by isometries
 we may assume $\check{x} = \tilde{x}$. Then $d(\check{x},\check{y}) < \delta + \alpha/3 \leq 2 \alpha/3 < \alpha$
 and $\check{y} \in p^{-1}(U)$ implies $\check{y} \in V$. Since $p^{-1}(\{ y \}) \cap V$ has cardinality $1$,
 we have $d(x,y) = d(\tilde{x},\check{y})$ hence $d(\tilde{x},\check{y}) < \delta$ and $\check{y} \in W$, whence
 $y \in p(W)$ and $C = p(W)$.
 
 Now, if $y_1,y_2 \in C = p(W)$, we have $d(y_1,y_2) \leq d(s_W(y_1),s_W(y_2))$ by definition
 of the induced metric. Assume by contradiction that $d(y_1,y_2) < d(s_W(y_1),s_W(y_2))$.
 Then there should exist $\tilde{y}_k \in p^{-1}(y_k)$, $k = 1,2$, such that
 $d(\tilde{y}_1,\tilde{y_2}) < d(s_W(y_1),s_W(y_2))$. Since $\Gamma$ acts by isometries we can assume
 $\tilde{y}_1 = s_W(y_1)$. But $d(s_W(y_1),s_W(y_2)) < 2 \delta \leq 2 \alpha/3 < \alpha$,
 and $d(s_W(y_1), \tilde{y}_2) < \alpha$ implies $\tilde{y}_2 \in V$. Since $V \cap p^{-1}(\{ y_2 \}) = \{ s_W(y_2) \}$ this proves $\tilde{y}_2 = s_W(y_2)$, a contradiction. This proves that $s_W$ is an isometry and the proposition.

\end{proof}

\begin{lemma} Let $G$ be a connected compact Lie group endowed with a bi-invariant metric, and $N$ a closed normal subgroup of $G$.
Then the projection map $G \to G/N$ admits local isometric cross-sections.
\end{lemma}

\begin{proof}
We need only look at a neighborhood of the neutral element of $G/N$.
Let $\mathfrak{G}$ denote the Lie algebra of $G$, and $\mathfrak{N} \subset \mathfrak{G}$ the Lie algebra of $N$.
The Lie algebra of $G/N$ is canonically identified with $\mathfrak{G}/\mathfrak{N}$. The bi-invariant metrics on $G$
and $G/N$ are associated to invariant scalar products on $\mathfrak{G}$ and $\mathfrak{G}/\mathfrak{N}$. Let $\mathfrak{H}$ denote the orthogonal inside
$\mathfrak{G}$ of $\mathfrak{N}$ with respect to such a scalar product. Since this scalar
product is invariant it is a Lie subalgebra (even a Lie ideal). The canonical Lie algebra morphism $\mathfrak{G} \to \mathfrak{G}/\mathfrak{N}$ restricts to a linear isometry $\varphi : \mathfrak{H} \to \mathfrak{G}/\mathfrak{N}$.
Now, there is $\eps>0$ such that the exponential map defines diffeomorphisms $B_{\mathfrak{G}/\mathfrak{N}}(0,\eps) \to B_{G/N}(1,\eps)$
and $B_{\mathfrak{G}}(0,\eps) \to B_G(1,\eps)$ which coincide with the Riemannian exponential map. Then $\exp \circ \varphi^{-1} \circ \exp^{-1} : B_{G/N}(0,\eps) \to B_{G}(0,\eps)$ defines an isometric local cross-section. Indeed, it is clearly a local cross-section and,
assuming without restriction of generality $\mathfrak{G} \subset \gl_N(\R)$ for some $N$ and letting $CH$ denote the Campbell-Hausdorff (Lie) series (which is convergent if $\eps$ is chosen small enough)
satisfying $\exp(u) \exp(v) = \exp CH(u,v)$, we have
$$
\begin{array}{clcl}
 &d_G(\exp \circ \varphi^{-1} \circ \exp^{-1} x,\exp \circ \varphi^{-1} \circ \exp^{-1} y) &\stackrel{(1)}{=}
&d_G( \exp ( \varphi^{-1} ( \exp^{-1} x))\exp (- \varphi^{-1} (\exp^{-1} y)), 1) \\
\stackrel{(2)}{=}& d_G( \exp (CH( \varphi^{-1} \circ \exp^{-1} x,\varphi^{-1} \circ \exp^{-1} y^{-1})), 1)
&\stackrel{(5)}{=}& d_{\mathfrak{G}} (CH( \varphi^{-1} \circ \exp^{-1} x, \varphi^{-1}\circ \exp^{-1} y^{-1})), 0) \\
\stackrel{(3)}{=}& d_{\mathfrak{G}} (\varphi^{-1}( CH(  \exp^{-1} x,\exp^{-1} y^{-1})), 0)&
=& d_{\mathfrak{H}} (\varphi^{-1}( CH(  \exp^{-1} x, \exp^{-1} y^{-1})), 0)\\
\stackrel{(4)}{=}& d_{\mathfrak{G}/\mathfrak{N}} (CH(  \exp^{-1} x, \exp^{-1} y^{-1})), 0)&
\stackrel{(5)}{=}& d_{G/N} (\exp CH(   \exp^{-1} x, \exp^{-1} y^{-1}), 1)\\
\stackrel{(2)}{=}& d_{G/N} (xy^{-1}, 1) \stackrel{(1)}{=} d_{G/N}(x,y).
\end{array}
$$
by using the right invariance of the metric (1), the definition of the Campbell-Hausdorff series (2),
the fact that $\varphi^{-1}$ is a morphism of Lie algebras (3) and an isometry (4), and the fact
that the Lie exponential coincides with the Riemannian exponential (5).
\end{proof}

\section{A $K(\Z,2)$ from $S^1$-valued random variables}

In this section we exhibit a family of local sections of the projection map
$L(\Omega,S^1) \to L(\Omega,S^1)/S^1$. This family is quite elementary
and the neighborhood on which it is defined is very concretely determined.
This proves that $L(\Omega,S^1)/S^1$ is a classifying space for $S^1$
and therefore a $K(\Z,2)$ without having to use Gleason's theorem or
any other sophisticated machinery.

\bigskip

To a Borel map $f : \Omega \to S^1$ we associate $F(x) = \int d(f(t),e^{\ii \pi x}) \dd t$, where $d$ is the geodesic metric
on $S^1$ with diameter $1$ (that is, this is the arc length $d(e^{\ii \pi\theta_1},e^{\ii \pi\theta_2}) = |\theta_1-\theta_2|$ if 
for instance $\theta_1,\theta_2 \in ]0,1[$). We have $F \in C(\R/(2\Z),\R_+)=C(S^1,\R_+)$.
Clearly, $f \mapsto F$ is  $1$-Lipschitz for the metric induced by the uniform norm, since
$$
\left| \int d(f_1(t),u) \dd t - \int d(f_2(t),u)\dd t \right| 
\leqslant
 \int \left|d(f_1(t),u) \dd t - \int d(f_2(t),u)\right| \dd t 
\leqslant
 \int d(f_1(t),f_2(t))  \dd t.
 $$

We now assume
$F(0) = \int d(f(t),1) \dd t < \alpha$, for some well-chosen $\alpha>0$. For all $\beta > 0$ we then have
$$
\alpha > \int_{d(f(t),1)\geq \beta} d(f(t),1)\dd t \geqslant \beta \mu\{ t; d(f(t),t) \geq \beta \}
$$
that is $\mu\{ t; d(f(t),1) \geq \beta \} < \alpha/\beta$, which implies
$\mu\{ t; d(f(t),1) < \beta \} \geq 1- \alpha/\beta$.

\begin{lemma} \label{lem:croissancecercle} Assume $x \in ]-1,1[$, and that $F(0) = \int d(f(t),1) \dd t < \alpha$. Then
\begin{enumerate}
\item If $|x| \geq \frac{1}{2}$, then $F(x)  \geq \frac{1}{4} - \alpha$
\item If $|x| \leq \frac{1}{16}$, then $F(x)  \leq \frac{1}{8} +16 \alpha$
\item if $\frac{1}{16} \leq y \leq x \leq \frac{1}{2}$ and $\alpha < 1/64$, then $F(x) - F(y) >0$
\item if $\frac{-1}{2} \leq x \leq y \leq \frac{-1}{16}$ and $\alpha < 1/64$, then $F(x) - F(y) >0$
\end{enumerate}
\end{lemma}
\begin{proof}
Let us choose $x \in ]-1,1[$. We assume $|x| \geq \frac{1}{2}$. If $d(f(t),1) \leq \frac{1}{4}$, then
$d(f(t),e^{\ii \pi x}) \geq \frac{1}{4}$. This implies
$$
|x| \geq \frac{1}{2} \Rightarrow F(x) = \int d(f(t),e^{\ii \pi x}) \dd t \geqslant \frac{1}{4} \mu \{ t ; d(f(t),1) < \frac{1}{4} \} \geqslant \frac{1}{4} ( 1 - 4\alpha)
$$
which proves (1).
We now assume $|x| \leq \frac{1}{16}$. If $d(f(t),1) \leq 1/16$ then $d(f(t),e^{\ii \pi x}) \leq 2 \times \frac{1}{16} = \frac{1}{8}$. Otherwise, we have  $d(f(t),e^{\ii \pi x}) \leq 1$. 
Since
$$
F(x) = \int_{d(f(t),1) < \frac{1}{16}} d(f(t),e^{\ii \pi x}) \dd t + \int_{d(f(t),1) \geq \frac{1}{16}} d(f(t),e^{\ii \pi x}) \dd t
$$
we get
$$
|x| \leq \frac{1}{16} \Rightarrow F(x) \leqslant \mu\{ t; d(f(t),1) \geq \frac{1}{16} \} + \frac{1}{8}\mu \{ d(f(t),1) < \frac{1}{16} \} < 16 \alpha
+ \frac{1}{8}
$$
which proves (2). 
\begin{figure}
\begin{center}
\resizebox{8cm}{!}{\includegraphics{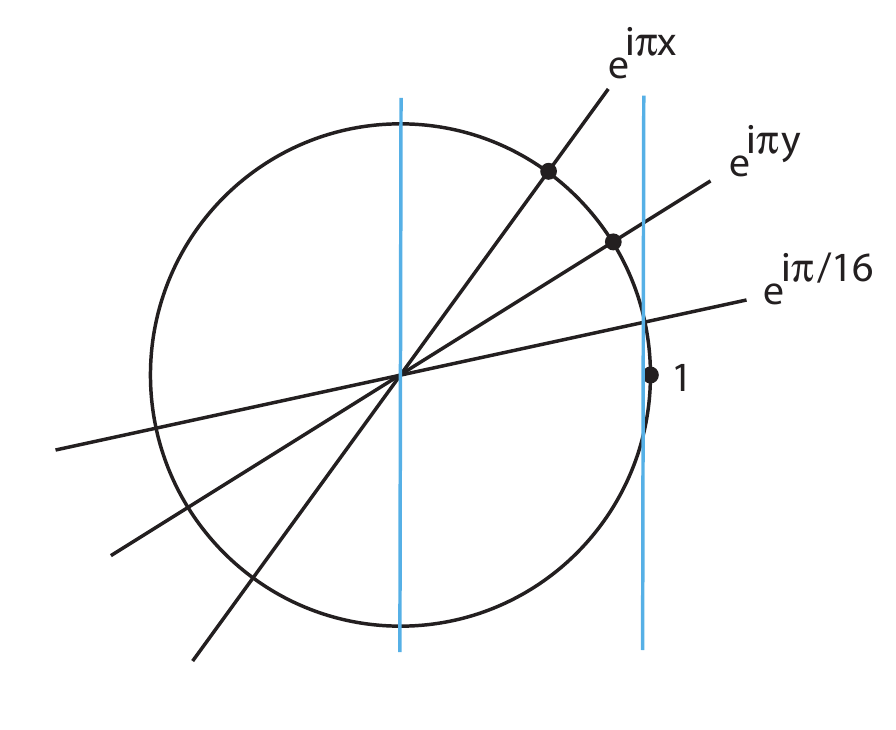}}
\end{center}
\caption{Proof of lemma \ref{lem:croissancecercle} : $F(x)-F(y)>0$ when $\frac{1}{16} \leq y \leq x \leq \frac{1}{2}$
and $\alpha < 1/64$}
\label{fig:croissancecercle}
\end{figure}
Then, $d(f(t),e^{\ii \pi x}) = d(f(t),e^{\ii \pi y}) + d(e^{\ii \pi y},e^{\ii \pi x}) = d(f(t),e^{\ii \pi y}) + y-x$
as soon as $f(t)$ belongs to the half-circle containing $1$ bounded by $e^{\ii \pi y}$ and $-e^{\ii \pi y}$ (see figure \ref{fig:croissancecercle}).
In particular, if $d(f(t),1) \leq 1/32$, then $d(f(t),e^{\ii \pi x}) = d(f(t),e^{\ii \pi y}) + x-y$.
It follows that
$$
F(x) = \int_{d(f(t),1) \leq 1/32} d(f(t),e^{\ii \pi y}) \dd t + (x-y) \mu \{ t ; d(f(t),1) \leq 1/32\} + \int_{d(f(t),1) > 1/32} d(f(t),e^{\ii \pi x}) \dd t
$$hence$$
F(x)= F(y) + (x-y) \mu \{ t ; d(f(t),1) \leq 1/32\}  + \int_{d(f(t),1) > 1/32} \left(d(f(t),e^{\ii \pi x})-d(f(t),e^{\ii \pi y})\right) \dd t
$$
and finally
$F(x) = F(y)+(x-y) \Delta(x,y)$ with
$$
\Delta(x,y) = \mu \{ t ; d(f(t),1) \leq 1/32\}  + \int_{d(f(t),1) > 1/32} \frac{d(f(t),e^{\ii \pi x})-d(f(t),e^{\ii \pi y})}{x-y} \dd t.
$$
We have $\mu \{ t ; d(f(t),1) \leq 1/32\}  \geq 1 - 32 \alpha$,
and
$$
\left| \frac{d(f(t),e^{\ii \pi x})-d(f(t),e^{\ii \pi y})}{x-y} \right|= 
 \left| \frac{d(f(t),e^{\ii \pi x})-d(f(t),e^{\ii \pi y})}{d(e^{\ii \pi x},e^{\ii \pi y})}\right|\left|\frac{d(e^{\ii \pi x},e^{\ii \pi y})}{x-y}\right| \leqslant \left|\frac{d(e^{\ii \pi x},e^{\ii \pi y})}{x-y} \right| \leqslant 1
$$
by the triangular inequality and the definition of the geodesic distance on $S^1$.
Therefore, $|\Delta(x,y) -1| \leq 32 \alpha + \mu \{ t ;d(f(t),1)> 1/32\} \leq 64 \alpha$
hence $\Delta(x,y) > 0$ as soon as $\alpha < 1/64$ and this proves (3). We now prove (4).
Let $z \mapsto \overline{z}$ denote the complex conjugation. It induces an isometry of $S^1$.
Then $F(x) = \int d(f(t),e^{\ii \pi x}) \dd t = \int d(\overline{f(t)},e^{-\ii \pi x}) \dd t$. Since $t \mapsto \overline{f(t)}$
belongs to $L(\Omega,S^1)$ and $\int d(\overline{f(t)},1) \dd t = \int d(f(t),1)\dd t < \alpha< 1/64$, by applying (3)
to $\overline{f}$ and $\frac{1}{16} \leq -y \leq -x \leq \frac{1}{2}$ we get
$F(x) - F(y)>0$  and (4).
\end{proof}

In particular, if $\alpha < 3/640$, we get the following property. 
This condition implies that $\frac{1}{8} + 16 \alpha < \frac{1}{5} < \frac{1}{4}- \alpha$ and $\alpha < 1/64$.
Therefore, there exists exactly two points $e^{\ii \pi s_-(f)},e^{\ii \pi s_+(f)} \in S^1$ such that
$F(s_{\pm}(f)) = \frac{1}{5}$, with $\frac{-1}{2} < s_-(f)  < \frac{-1}{16}$
and $\frac{1}{16} < s_+(f)  < \frac{1}{2}$. This defines two maps $s_-, s_+ : B(1,\alpha) \to ]-1,1[$
where $B(1,\alpha)$ is the open ball in $L(\Omega,S^1)$ with center $1$ and radius $\alpha$.

\begin{lemma} \label{lem:continuoussectionS1} The maps $s_+,s_- : B(1,\alpha) \to ]-1,1[$ are continuous.
\end{lemma}
\begin{proof} Let $f_0 \in B(1,\alpha)$. We prove that $s_+$ is continuous at $f_0$, the proof for
$s_-$ being similar. Let $\eps > 0$. We want to prove that there exists $\delta > 0$ such that,
is $d(f_0,f) \leq \delta$, then $|s_+(f) - s_+(f_0)| \leq \eps$.
First of all, let $\delta_1 > 0$ be such that $d(f_0,f) \leq \delta_1 \Rightarrow f \in B(1,\alpha)$. We let
$F_0,F \in C(S^1,\R_+)$ be the maps associated to $f,f_0$.  We set $\alpha^u = \frac{1}{4} - \alpha$,
$\alpha^d = \frac{1}{8} +16 \alpha$.
Obviously 
$\alpha^d < \frac{1}{5} < \alpha^u$.
Now $F_0$ induces a bicontinuous bijection from  $[\frac{1}{16} , \frac{1}{2}]$ to its image. Let $\Phi_0$
denote the converse map. Its range contains $[\alpha^d, \alpha^u]$. By definition, $\Phi_0(1/5) = s_+(f_0)$. By (uniform) continuity,
there exists $\delta_2$ such that $|x-y| \leq \delta_2 \Rightarrow |\Phi_0(x) - \Phi_0(y)| \leq \eps$.

Without loss of generality, we can assume that $\delta_2$ is small enough so that $\alpha^d < \frac{1}{5} - \delta_2/2 < \frac{1}{5}+ \delta_2/2 < \alpha^u$.

We choose $0< \delta < \min(\delta_1,\delta_2/2)$.
Since $f \mapsto F$ is 1-Lipschitz, we have $d(F_0,F) \leq d(f_0,f) \leq \delta$. We set 
$c_{\pm} = \Phi_0(\frac{1}{5} \pm \delta)$. Note that $c_{\pm} \in [\alpha^d, \alpha^u]$.
By definition $F_0(c_{\pm}) = \frac{1}{5} \pm \delta$ 
hence $F(c_+) \leq \frac{1}{5} \leq F(c_-)$. It follows that
$s_+(f) \in [c_+,c_-]$. Now $2\delta < \delta_2$ implies that $|\Phi_0(\frac{1}{5} + \delta)-\Phi_0(\frac{1}{5} - \delta)| \leq \eps$ that is $|c_+ - c_-| \leq \eps$. Since $s_+(f_0) = \Phi_0(\frac{1}{5}) \in [c_+,c_-]$ this implies
$|s_+(f)-s_+(f_0)|  \leq \eps$ and the conclusion.
\end{proof}
\begin{remark}
\end{remark}
Instead of the maps $s_+,s_-$, one may be willing to consider the place where $F$ reaches a minimum.
This does not work, first of all because the minimum may be reached at several places, but also, maybe more crucially,
because of the following phenomenon. For simplicity let us consider the gentle case where the map $F$ is convex in a neighborhood of $1$,
for instance if $f$ is identically $0$ in a neightborhood of $-1$. In that case, $F(x) = \int_0^1|\varphi(t)-x| \dd t$
for some $\varphi : [0,1] \to ]-\beta,\beta[ \subsetneq ]-1,1[$, and the place where $F$ reaches a minimum is necessarily a closed interval,
of the form $[s_0^-(f), s_0^+(f)]$, where $f(t) = \exp(\ii \pi \varphi(t))$. We claim that the maps $s_0^-$ and $s_0^+$ are \emph{not} continuous, and
illustrate this by the following example. 
\begin{figure}
\begin{center}
\resizebox{4cm}{!}{\includegraphics{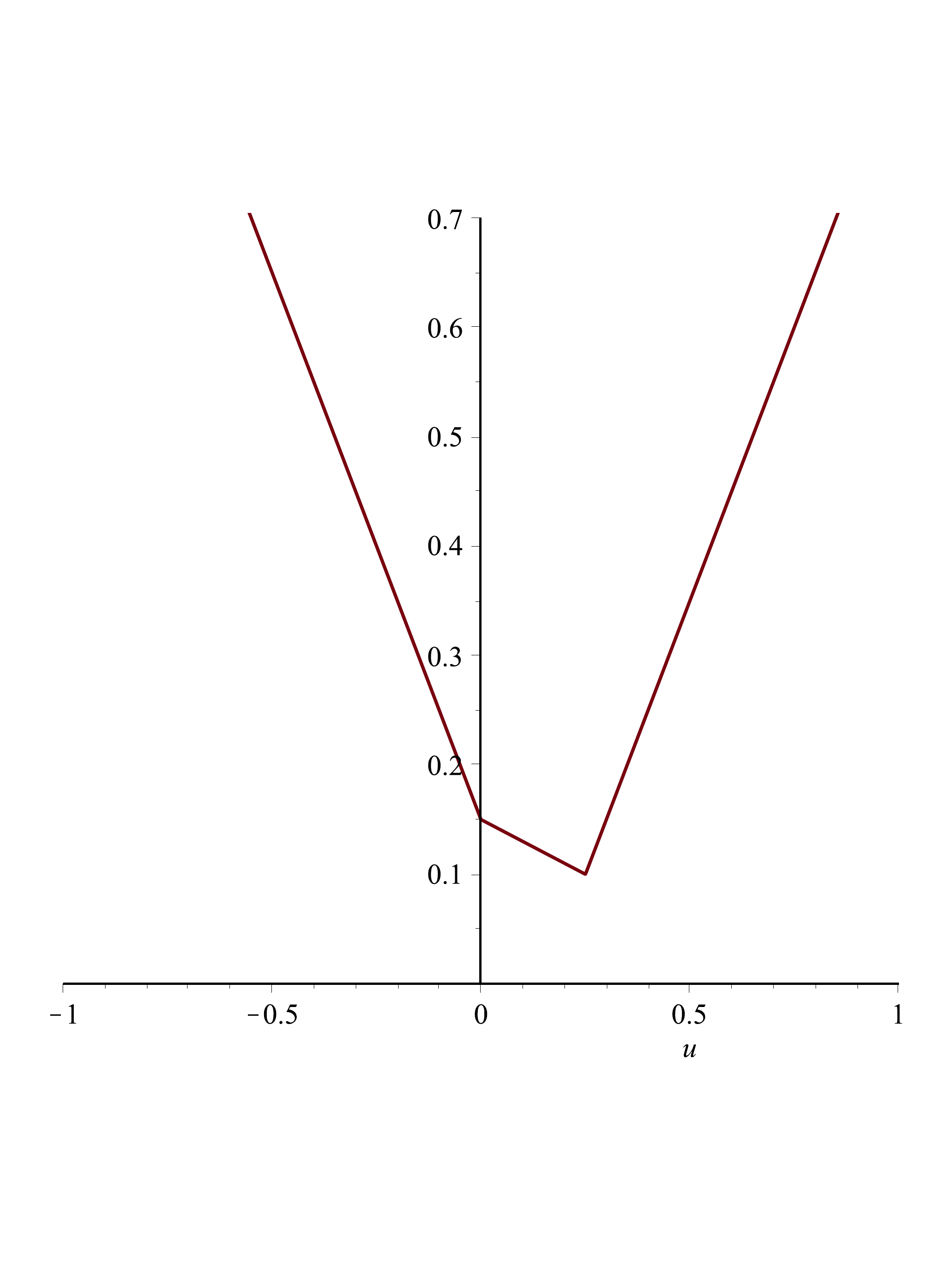}}
\resizebox{4cm}{!}{\includegraphics{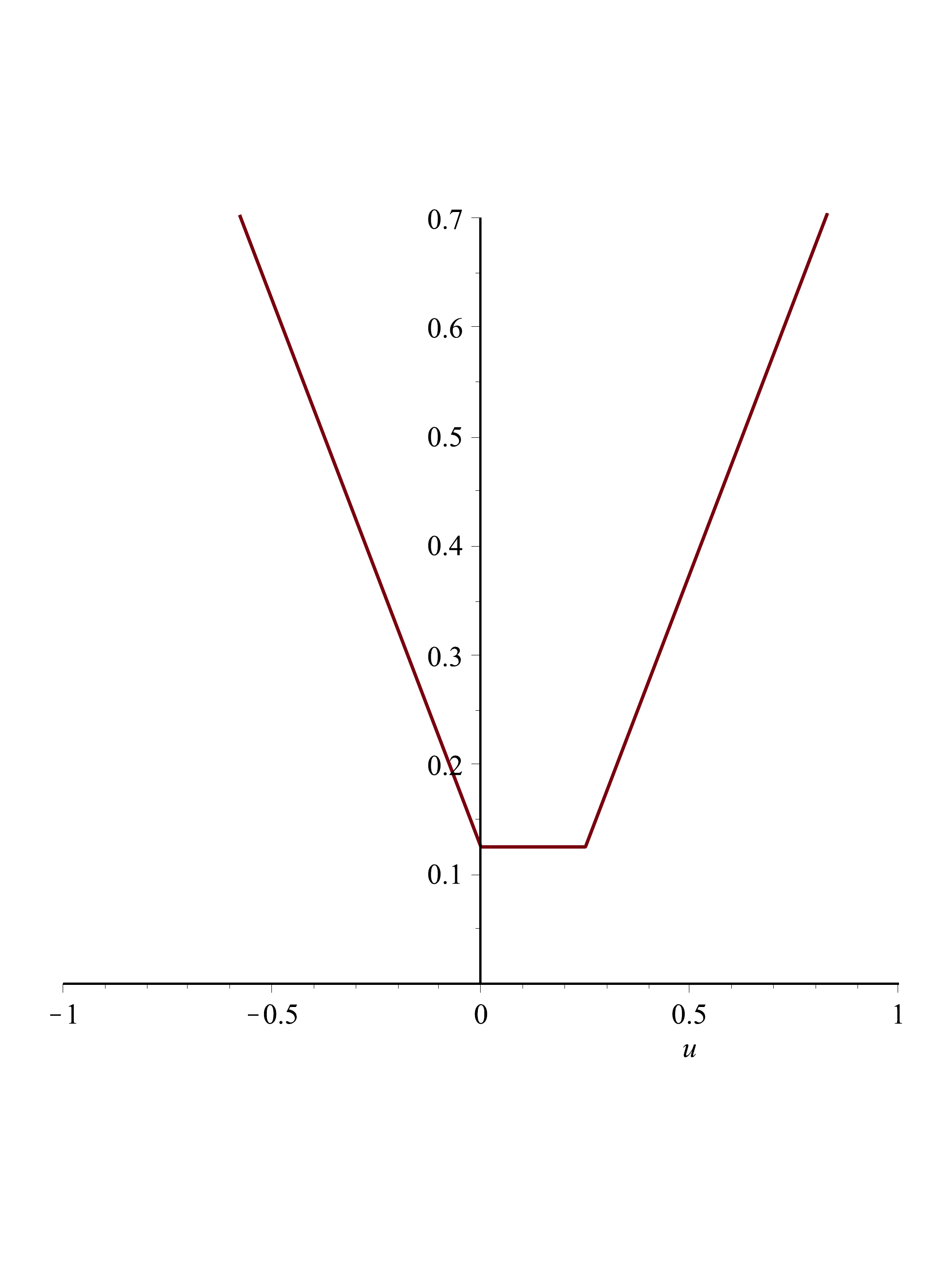}}
\resizebox{4cm}{!}{\includegraphics{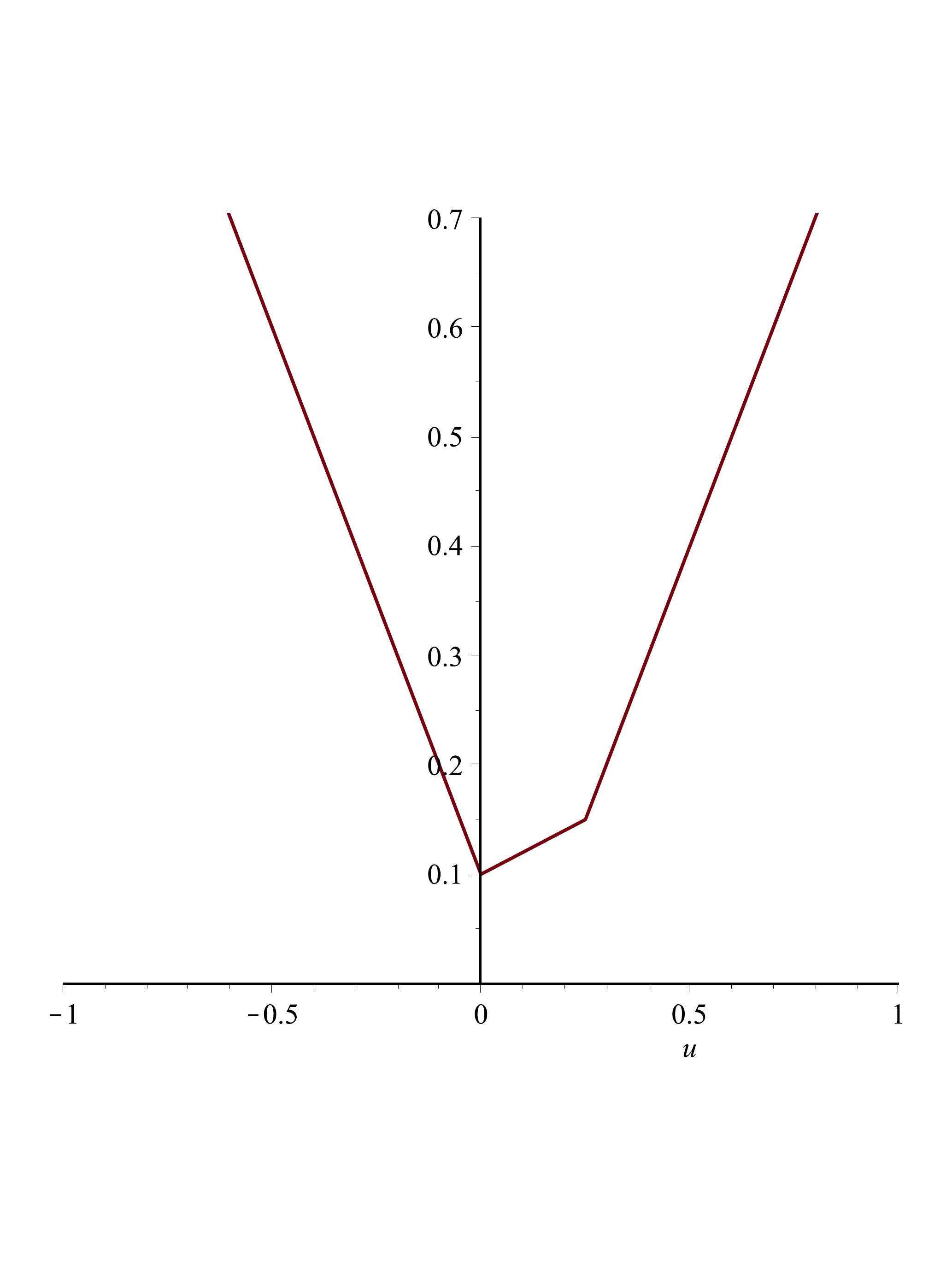}}
\end{center}
\caption{Non-continuity of $\lambda \mapsto s_0^{\pm}(f_{\la})$ : $\la =\frac{1}{2} - \frac{1}{10},\frac{1}{2},\frac{1}{2} + \frac{1}{10}$}
\label{fig:remnoncont}
\end{figure}
Let us set $\varphi_{\la}(t) = (1/4) \un_{[\la,1]}$ when $\la \in [0,1]$. Then, in a neighbourghood of $1$, we have $F_{\la}(u) = \int_0^1|\varphi_{\la}(t) - u|\dd t
= \la |u| + (1- \la)|(1/4) - u|$. The map
$\la \mapsto f_{\la}$ is continuous, but (see figure \ref{fig:remnoncont})
\begin{itemize}
\item if $\la < 1/2$, then $s_0^-(f_{\la}) = s_0^+(f_{\la}) = 1/4$
\item if $\la > 1/2$, then $s_0^-(f_{\la}) = s_0^+(f_{\la}) = 0$
\item if $\la = 1/2$, then $s_0^-(f_{\la}) =0$ and $ s_0^+(f_{\la}) = 1/4$.
\end{itemize}
Therefore, the maps $\la \mapsto s_0^{\pm}(f_{\la})$ are not continuous.

\begin{proposition} \label{prop:existcrosssectionS1} The projection map $L(\Omega,S^1) \to L(\Omega,S^1) /S^1$
admits (continuous) local sections.
\end{proposition}
\begin{proof}
Let $\alpha \in ]0,3/640[$ as before. We set $S_+(f) = \exp(\ii \pi s_+(f))$,
thus defining a continuous map $S_+ : B(1,\alpha) \to S^1$. Let $f \in B(1,\alpha)$
and $z = e^{\ii \pi \delta}$ with $\delta \in ]-1,1]$ such that $z. f \in B(1,\alpha)$. Let $F$ and $H$
denote the continuous maps $\R \to \R_+$ attached to $f$ and $z.f$. We have
$H(x) = \int d(z.f(t),e^{\ii \pi x}) \dd t = \int d(f(t),e^{\ii \pi (x-\delta)}) \dd t = F(x-\delta)$ for all $x \in \R$.
Since $H(\delta) = F(0) < \alpha$ we have $|\delta| < 1/16$. It follows that
$x + \delta \in ]0,1[$ for all $x \in [1/16,1/2]$. Since $H(s_+(f)+\delta) = F(s_+(f)) = 1/5$ this implies
$s_+(z.f) = s_+(f) + \delta$ hence $S_+(z.f) = z.S_+(f)$. Let $g \mapsto \bar{g}$ denote the projection
map $L(\Omega,S^1) \to L(\Omega,S^1) /S^1$. The image of $B(1,\alpha)$ is the open ball
$B(\bar{1},\alpha)$. Let us associate to $f \in B(1,\alpha)$ the map $T(f) = S_+(f)^{-1} f \in L(\Omega,S^1)$.
If $z.f \in B(1,\alpha)$, then $T(z.f) = T(f)$. This proves that $T$
defines a map  $\tilde{T} : B(\bar{1},\alpha) \to L(\Omega,S^1)$ such that $\overline{\tilde{T}(\bar{f})} = \bar{f}$.
By lemma \ref{lem:continuoussectionS1} this map is continuous, therefore it is a local continuous section near $\bar{1}$. A local section in a neighbourghood of any given $\bar{f_0} \in L(\Omega,S^1)/S^1$ is then given by $\bar{f} \mapsto \tilde{T}( \bar{f} \bar{f}_0^{-1}) \bar{f}_0$.
\end{proof}

\begin{corollary}  $L(\Omega,S^1)/S^1$ is a classifying space for the topological group $S^1$
inside the category of paracompact spaces.
\end{corollary}
\begin{proof}
By proposition \ref{prop:existcrosssectionS1} we know that the projection map
 $L(\Omega,S^1) \to L(\Omega,S^1) /S^1$ admits a local cross-section. Since $S^1$
 is commutative, proposition \ref{prop:genclassifiant} implies the conclusion.

\end{proof}
\begin{corollary} $L(\Omega,S^1)/S^1$ is a $K(\Z,2)$.
\end{corollary}
\begin{proof} By the previous corollary we know that $L(\Omega,S^1)/S^1$ is
a classifying space for $S^1$. Since it is known how to built a paracompact classifying
space for $S^1$, which has the homotopy type of a CW-complex and is a $K(\Z,2)$, $L(\Omega,S^1)/S^1$,
being homotopy equivalent to it, has itself  the homotopy type of a CW-complex and is a $K(\Z,2)$.
\end{proof}

\section{Classifying spaces for compact Lie groups}

In this section we prove that a local cross-section of the projection
map $L(\Omega,G) \to L(\Omega,G)/G$ when $G$ is a compact Lie
group can be obtained by following the geometric idea of assigning
continously an almost-center of mass to the map $f \in L(\Omega,G)$.
This provides an alternative proof to the existence of a local
cross-section, that does not use Gleason's theorem, but uses instead
Karcher's theory of a Riemannian center of mass.

\medskip

\subsection{Riemannian preliminaries}

Here we recall a few basic facts on Riemannian manifolds.
Our textbook reference for the material used here is \cite{JOST}.
We recall that a $n$-dimensional Riemannian manifold $M$ is
naturally endowed with an intrinsic metric $d(x,y)$, defined by the minimal length
of a geodesic joining $x$ and $y$. A geodesic of minimal length will
be called a (length-)minimizing geodesic. By the Hopf-Rinow theorem, $M$ is complete
as a metric space iff all its bounded-closed subsets are compact iff every two points can
be joined by a minimizing geodesic iff for all $p \in M$ there is an
exponential map $\exp_p : T_p M \to M$ satisfying $d(\exp_p(v),p) = \| v \|$, where $\| . \|$
is the norm on $T_p M$ defined by the riemannian structure. Such an exponential map is a diffeomorphism
when restricted to some ball of radius $\rho_p(M)$, where $\rho_p(M) > 0$ is known as the radius of injectivity
of $M$ at $p$. A complete Riemannian manifold is compact iff it is has finite diameter for the intrinsic metric.
We denote $diam(M)$ this diameter.

For every point $p \in M$ there exists a neighborhood $U$ of $p$ which is geodesically convex,
meaning that every two points can be joined by exactly one minimizing geodesic, whose
support lies inside $U$. Moreover, there exists a positive $r>0$ such that all balls
of center $p$ and radius $< r$ are geodesically convex. The supremum of such $r$'s
is called the convexity radius at $p$ (see \cite{GHL}, 2.90).

We also recall that a manifold is orientable iff its tangent bundle can be trivialized iff
it admits a nowhere-vanishing volume form. Such a volume form defines a measure
on $M$, and two such measures can be deduced one from the other by multiplying
by some nowhere-vanishing continuous map $M \to \R_+^{\times}$.

Finally, a crucial fact for us will be that, for every $p \in M$, there exists an
open neighborhood $U$ of $p$ which is (geodesically) convex and on which,
for each $a \in U$, the map
$x \mapsto d(x,a)^2$
is \emph{strictly convex} (see for example \cite{JOST}, theorem 4.6.1).
We recall that the (strict) convexity of a real-valued function $f$ on a geodesically convex
subset $U$ of $M$ means that, for every geodesic $\gamma : [0,\alpha] \to U$,
the map $f_{|\gamma} : t \mapsto f \circ \gamma$ is (strictly) convex.

\subsection{The Riemannian center of mass}

We need to use Karcher's notion of the Riemannian center of mass,
and we also need to establish the following `continuous' version of it. We denote $E \Delta F = (E \setminus F) \cup (F \setminus E)$.

\begin{proposition} \label{prop:masscont} Let $M$ be a Riemannian manifold, and $p_0 \in M$. We
assume $M$ endowed with a volume form $\dd p$ which is non-vanishing in a neighborhood of $p_0$. There
exists an open neighborhood $U$ of $p_0$ with the following property.
For every Borel subset $E$ of $U$ of positive measure, there exists a unique $c(E) \in M$
on which the map $x \mapsto \int_E d(x,p)^2 \dd p$ has minimal value.
Moreover, $c(E) \in U$, and the map $E \mapsto c(E)$ is continuous, from the set of Borel subsets of $U$
which have positive measure, endowed with the (pseudo-)distance $d(E,F) = \int_{E \Delta F} \dd p$,
to $U$.
\end{proposition}

We recalled earlier that there exists a  neighborhood $U'$ of $p_0$ such which is (geodesically) convex and on which,
for each $a \in U'$, the map
$x \mapsto d(x,a)^2$
is strictly convex. We can assume that $U'$ is small enough so that the volume form $\dd p$ is non-vanishing over $U$.
If $E \subset U'$ has positive measure, then the
map $x \mapsto \int_E d(x,p)^2 \dd p$ is also strictly convex on $U'$, and therefore
there exists a unique $c(E) \in U'$ on which it takes minimal value. This is Karcher's argument for the Riemannian center
of mass.

Assume that $U'$ contains the ball $B(p_0,A)$ for some $A > 0$. We can assume that $A$
is smaller than the convexity radius of $p_0$, namely every ball $B(p_0,A')$ with $A' \leq A$ is geodesically
convex. Let $U$ be the ball $B(p_0,A/3)$, and $E \subset U$.
Then, for each $x \not\in U$, and $p \in B(p_0,A/3)$ then we have $d(x,p) \geq d(x,p_0) - d(p_0,p) \geq A - A/3 = 2A/3 > d(p_0,p)$
whence $\int_E d(x,p)^2 \dd p > \int_E d(p_0,p)^2 \dd p$, which proves that $c(E)$ is the only minimum of $x \mapsto \int_E d(x,p)^2 \dd p$
not only over $U$, but over $M$.

We denote $\mu_M$ the measure on $M$ associated to the volume form $\dd p$. We now introduce the following topological spaces,
where $X$ denotes a geodesically convex open subset of $M$ :
\begin{itemize}
\item the set $\mathcal{B}(X)$ of Borel subsets of $X$ with positive measure endowed with the pseudo-distance $d(E,F) = \mu_M(E \Delta F)$
\item the set $E(X)$ of real-valued bounded continuous functions admitting a unique minimum on $X$ for the topology induced by $\| . \|_{\infty}$
\item its subspace $E_c(X) \subset E(X)$ of strictly convex functions on $X$.
\end{itemize}

\begin{lemma} \label{lem:contBUtoECU}
The map $\mathcal{B}(U) \to E_c(U)$ defined by
$$
E \mapsto \left( x \mapsto \int_M d(x,p)^2 \un_E(p) \dd p \right)
$$
is continuous (and actually $diam(M)^2$-Lipschitz).
\end{lemma}
\begin{proof}
If $E$ is non-empty we know that $F_E : x \mapsto \int_M d(x,p)^2 \un_U(p) \dd p$
is strictly convex. Since $M$ is bounded, $F_E$ is also bounded.
Now, $F_E$ is continuous because $x \mapsto d(x,p)^2$ is a $2 diam(M)$-Lipschitz
map $M \to \R_+$, and therefore $|F_E(x) - F_E(y)| \leq 2 diam(M) \mu_M(E) d(x,y)$
for all $x,y \in M$. We now prove that $E \mapsto F_E$ is continuous. We know
that
$$
|F_E(x) - F_{E'}(x)| \leq \int_M d(x,p)^2|\un_E(p) - \un_{E'}(p)| \dd p \leq diam(M)^2 \int_M |\un_E(x) - \un_{E'}(x)| \dd p
=diam(M)^2 \mu_M(E\Delta E')
$$
hence $E \mapsto F_E$ is $diam(M)^2$-Lipschitz.
\end{proof}

Proposition \ref{prop:masscont} is then an immediate consequence of lemma \ref{lem:contBUtoECU} and of the following one.

\begin{lemma} \label{lem:convcont}
For every $f \in E_c(U)$ there exists a unique $m(f) \in U$ on which $f$ has minimal value.
For all $f \in E_c(U)$, there exists $r_{max} > 0$ such that the map
$\delta \mapsto \inf \{ f(x) ; d(x,m(f)) \geq \delta \}$ is strictly increasing $[0,r_{\max}[ \to \R_+$,
and continuous.
Moreover, the map
$E_c(U) \to U$, $f \mapsto m(f)$
is continuous.
\end{lemma}

In order to prove this lemma, we first recall the following very general fact (valid on arbitrary metric spaces) :
\begin{lemma} \label{lem:minlip}
The map $E(U) \to \R$ defined by $f \mapsto \min f$ is 1-Lipschitz.
\end{lemma}
\begin{proof}
Let $f,f_0 \in E(U)$ and $|| f - f_0 ||_{\infty} = a$. 
For every $x \in U$ we have
$f(x) \geq f_0(x) - a \geq \min(f_0)-a$ hence
$\min(f) \geq \min(f_0) - a$ and $\min(f_0) - \min(f) \leq a$. Symmetrically
we have $\min(f) - \min(f_0) \leq a$ and therefore $|\min(f_0) - \min(f)| \leq || f - f_0||$,
which concludes the proof.
\end{proof}

We can then proceed with the
\begin{proof} ({\it of lemma \ref{lem:convcont}})
The existence and uniqueness of $m(f)$ is immediate from the strict convexity of $f \in E_c(U)$ (and the locally compactness as well as the
geodesic convexity of $U$).
Let $f_0 \in E_c(U)$ and $m_0 =m(f_0) \in U$. We choose $r_{max}>0$ so that
$$
r_{max} \leq  \sup \{ \alpha ; \forall x \in M \ d(x,m_0) \leq \alpha \Rightarrow  x \in U \ \& \ \exists x \in U \ d(x,m_0) = \alpha \}.
$$
and so that $r_{max}$ is smaller than the injectivity radius (of the exponential map) at $m_0$.

For $0 \leq \delta< r_{max}$ we denote $F(\delta) = \inf \{ f_0(x) ; d(x,m_0) \geq \delta \}$. We have $F(0) = f_0(m_0) = \min f_0$.
Let $0 \leq \delta < r_{max}$ and  $x_1 \in U$ such that $d(x_1,m_0) > \delta$. Let $\gamma$ be a minimizing geodesic
from $m_0$ to $x_1$.
Then $\gamma$ contains $x_2$ such that $d(m_0,x_2) = \delta$. 
Since $f_0$ is strictly convex and has its minimum at $m_0$, we know that 
$(f_0)_{|\gamma}$ is strictly increasing, hence $f_0(x_2) < f_0(x_1)$. It follows that
$F(\delta) = \inf \{ f_0(x) ; d(x,m_0) = \delta \}$.

Let us choose $\delta_1,\delta_2$ with $0 \leq \delta_2 < \delta_1$. 
Since $f_0$ is continuous, $\{ f_0(x) : d(x,m_0) = \delta_1 \}$ is bounded and closed and therefore compact. 
Therefore there exists $x_1 \in M$ with $d(m_0,x_1) = \delta_1$
such that $f_0(x_1) = F(\delta_1)$.  Let $\gamma$ be a minimizing geodesic
from $x_1$ to $m_0$, and $x_2$ on $\gamma$ with $d(x_2,m_0) = \delta_2$.
By the same argument as before we get $f_0(x_2) < f_0(x_1) = F(\delta_1)$
and therefore $F(\delta_2) \leq f_0(x_2) < F(\delta_1)$. This proves that $F$ is strictly increasing.

We now prove that $F$ is continuous.
 Let $\delta_{\infty} \in [0,r_{max}[$.
We prove that $F$ is continuous at $\delta_{\infty}$. If not, there would be $\alpha > 0$ a sequence $(\delta_n)_{n \in \N}$ in $[0,r_{max}[$
converging to $\delta_{\infty}$ such that $|F(\delta_n)-F(\delta_{\infty})| \geq \alpha > 0$.
Let $x_n \in U$ such that $d(m_0,x_n) = \delta_n$ and $f_0(x_n) = F(\delta_n)$. Since $(\delta_n) \to \delta_{\infty} < r_{max}$
we know $\forall n \ \delta_n \leq r'$ for some $r'  < r_{max}$. By our assumption on $r_{max}$ the ball $\{ x \in U; d(x,m_0) \in [0,r'] \}$
is compact, and therefore 
there exists $x_{\infty} \in U$ and a subsequence $(x_{n_k})_{k \geq 0}$ in $U$ converging to $x_{\infty}$.
By continuity of the distance function we have $d(x_{\infty},m_0) = \lim_k \delta_{n_k} = \delta_{\infty}$.
Since $F(\delta_{n_k}) = f_0(x_{n_k}) \to f_0(x_{\infty}) \geq F(\delta_{\infty})$, we can assume
up to replacing $(\delta_n)_{n \in \N}$ by a subsequence that $\forall n \ F(\delta_n) \geq F(\delta_{\infty})$.
This implies $\delta_n \geq \delta_{\infty}$.

Let now $y_{\infty} \in U$ such that $d(y_{\infty},m_0) = \delta_{\infty}$ and $f(y_{\infty}) = F(\delta_{\infty})$. By continuity
of $f_0$ at $y_{\infty}$ we know that there exists $\eta > 0$ such that $d(y_{\infty},x) \leq \eta$ implies $|F(\delta_{\infty}) - f_0(x)| \leq \alpha/2$.
Let $\gamma : t \mapsto \exp_{m_0} (tv)$ the minimizing geodesic from $m_0$ to $y_{\infty}$, with $\| v \| = 1$. Since $\delta_{\infty} < r_{max}$
we know that $\gamma(t) \in U$ for $\delta_{\infty} < t < r_{max}$.
 Moreover, $d(\gamma(t),m_0) = t$
inside this range, because $U$ is geodesically convex and $r_{max}$ is smaller than the radius of injectivity at $m_0$.
Since $\gamma$ is continuous and $y_{\infty} = \gamma(\delta_{\infty})$ there exists $t_0$ inside this range such that $d(\gamma(t),y_{\infty})
\leq \eta$,and therefore $|F(\delta_{\infty}) - f_0(\gamma(t))| \leq \alpha/2$,  for all $t \in [\delta_{\infty},t_0]$.
But for $n$ large enough, we have $\delta_n \in ]\delta_{\infty},t_0]$ hence 
$$
F(\delta_{\infty})+ \alpha \leqslant F(\delta_n) \leqslant f_0(\gamma(\delta_n)) \leqslant F(\delta_{\infty}) + \frac{\alpha}{2}
$$
and this contradiction proves the continuity of $F$.

We now want to prove that $f \mapsto m(f)$ is continuous at $f_0$. Let us choose $\eps>0$ with $\eps < r_{max}$.
We have $F(\eps)>F(0)$ since $F$ is strictly increasing. Let $\eta = |F(\eps)-F(0)|/3$ and assume
$\| f- f_0 \|_{\infty} < \eta$ with $f \in E_c(U)$. By lemma \ref{lem:minlip} we know that this implies
$|\min(f) - f_0(m_0)| <|F(\eps)-F(0)|/3$. Since $\min(f) = f(m(f))$ this yields
$$
|f_0(m(f)) - f_0(m_0)| \leq |f_0(m(f)) - f(m(f))| + |\min(f) - f_0(m_0)| \leqslant 2|F(\eps)-F(0)|/3
< |F(\eps)-F(0)|
$$
that is
$f_0(m(f)) - F(0)= |f_0(m(f)) - F(0)| < F(\eps) - F(0)$
hence $f_0(m(f)) < F(\eps)$. This implies by definition of $F$ that
$d(m_0,m(f))<\eps$, and therefore
$f \mapsto m(f)$ is continuous over $E_c(U)$.

\end{proof}

We notice that the map $m : E_c(U) \to U$ is \emph{not} Lipschitz. For instance, when $U  = [-1,2] \subset \R$,
$f_n : t \mapsto  t^2/n$,
$g_n : t \mapsto (t-1)^2/n$, we have $f_n,g_n \in E_c(U)$ for all $n$, $d(m(f_n),m(g_n)) = d(0,1) = 1$,
and $d(f_n,g_n) =\frac{1}{n}  \| t^2-(t-1)^2 \|_{\infty} \to 0$.

\subsection{Essential support}

We recall that, when $M$ is a metric
space, $L(\Omega,M)$ is the set of (equivalence classes of) Borel maps $\Omega \to M$ such that
for one (hence for every) point $x_0 \in M$ we have $\int d(f(t),x_0)\dd t< \infty$. If a topological group $G$ acts on $M$ by isometries,
it obviously acts on $L(\Omega,M)$ via $g.f : t \mapsto g. f(t)$ for $g \in G$. The metric on $L(\Omega,M)$ is given as usual
by $d(f_1,f_2) = \int d(f_1(t),f_2(t))\dd t$. For $x_0 \in M$ we let $\tilde{x}_0 : \Omega \to M$ be the
constant map $t \mapsto x_0$. For $f_0 \in L(\Omega,M)$, we denote $B(f_0,\alpha)$ the open ball $\{ f  \in L(\Omega,M) \ | \ \int d(f(t),f_0(t))\dd t < \alpha \}$. The goal of this section is to prove the following proposition, which enables us to associate continuously to
a distribution of mass $f : \Omega \to M$, mostly concentrated near a point, an open subset of $M$ which
lies inside some prescribed geodesically convex neighborhood of this point.

\begin{proposition} \label{prop:esssupport}
Let $M$ be a connected complete Riemannian manifold endowed with its intrinsic metric
and a volume form $\dd p$, $x_0 \in M$, and $U$ an open neighborhood
of $x_0$. We assume that $\dd p$ is nowhere-vanishing over $U$.
Let $G$
be a topological group acting by isometries on $M$. There exists 
a real number
$\alpha > 0$ and a continuous map $\Phi : B(\tilde{x}_0,\alpha) \to \mathcal{B}(U)\subset \mathcal{B}(M)$ such that,
if $f \in B(\tilde{x}_0,\alpha)$ and $g \in G$ satisfy $g.f \in B(\tilde{x}_0,\alpha)$, then
$\Phi(g.f) = g.\Phi(f)$, where the (partly defined) action of $G$ on $\mathcal{B}(U)$ is the one obviously
induced by the action of $G$ on $M$.
\end{proposition}

In order to simplify notations, we denote $0 = x_0$. Up to possibly replacing $U$ by a smaller open neighborhood,
we can assume as in the previous section that $U$ is taken to a ball $B(0,A)$ with $A>0$ small enough so that
it is geodesically convex and on which, for each $a \in U$, the map $x \mapsto d(x,a)^2$
is strictly convex on $U$. We can also assume that $A$ is strictly smaller
than the injectivity radius of $M$ at $x_0$ and, for some technical reason, that $A \leq 4$.

\begin{lemma} \label{lem:lemC}
Let $r_2 \in ]0,A[$ with $A$ as before. Let us choose $r_1 \in ]0,r_2/2[$.
Then, for all $x,y,a \in M$,
with $d(x,0) \in ]r_2,r_3[$,$d(y,0) \in ]r_2,A[$,$d(0,a) \in [0,r_1[$ and such that
$x$ lies on the minimizing geodesic from $0$ to $y$, we have
$d(y,a)^2-d(x,a)^2 \geq md(x,y)$ with $m = r_2-2r_1>0$.
\end{lemma}
\begin{proof}
Let us denote $\gamma : [0,d(0,y)] \to M$ the unique minimizing geodesic from $0$ to $y$,
parametrized according to arclength. By hypothesis we have $\gamma(d(0,x)) = x$. 
We choose $r_2 \in ]0,A[$ arbitrarily, and $r_1 \in ]0, r_2/2[$.
 Let
us consider the function $G(t) = d(\gamma(t),a)^2$. We have $G(0) = d(0,a)^2 < r_1^2$.
Let $x_2 = \gamma(r_2)$. 
We have $d(0,x_2) = r_2 > r_1 \geq d(0,a)$ and therefore
$d(a,x_2) \geq |d(0,x_2) - d(0,a)| = d(0,x_2)-d(0,a) \geq r_2 - r_1$
hence $G(r_2)=d(a,x_2)^2 \geq  (r_2-r_1)^2> r_1^2$ since $r_2 > 2 r_1$.
From this we deduce $G(r_2)>G(0)$.
By assumption on $A$ the map
 $t \mapsto G(t)$ is strictly convex on $B(0,A)$
hence $G$ is strictly increading on $[r_2,A]$, and moreover the slopes
$(G(t_2)-G(t_1))/(t_2-t_1)$ for $r_2 < t_1 < t_2 < A$ are greater than
the slope $(G(r_2)-G(0))/r_2 \geq ((r_2-r_1)^2-r_1^2)/r_2=r_2-2r_1$.
Letting $m = r_2 - 2 r_1 >0$, $t_2 = d(0,y)$, $t_1 = d(0,x)$
we get $d(y,a)^2-d(x,a)^2 \geq m d(x,y)$ and the conclusion.

\end{proof}

\includegraphics{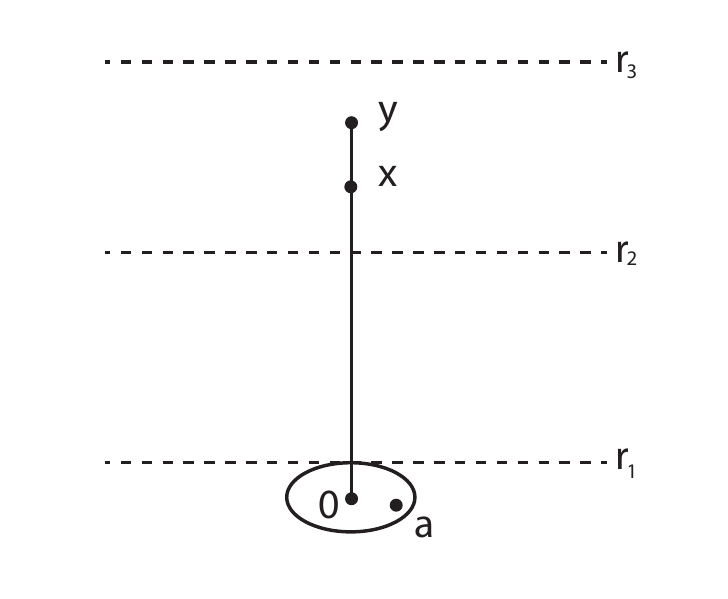}

For $f \in L(\Omega,M)$ we let $F(x) = \int d(x,f(t))^2\dd t$. This defines a continuous map $M \to \R_+$.
Note that, for every $\alpha,\beta>0$,
if $d(f,0) < \alpha$ then $\mu \{ t ; d(f(t),0) < \beta \} \geq 1 - \frac{\alpha}{\beta}$. We now choose $\alpha > 0$ small enough so
that $\alpha \leq \frac{A}{24diam(M)}$ and $1 - \frac{\alpha}{r_1} \geq \frac{\frac{m}{2} + 2 diam(M)}{m + 2 diam(M)}$.
The latter condition implies the following. Assume that $y,x$ are as in the previous lemma, that $d(0,f)<\alpha$ and let $E = \{ t ; d(f(t),0)<r_1\}$. Then
$$
F(y) - F(x) = \int_E (d(y,f(t))^2 - d(x,f(t))^2)\dd t + \int_{\Omega \setminus E} (d(y,f(t))^2 - d(x,f(t))^2)\dd t 
$$
hence 
$$
F(y) - F(x) \geqslant \mu(E)md(x,y) + \int_{\Omega \setminus E} (d(y,f(t))^2 - d(x,f(t))^2)\dd t \geqslant \frac{m}{2} d(x,y)
$$
as soon as 
$$
\int_{\Omega \setminus E} (d(y,f(t))^2 - d(x,f(t))^2)\dd t \leqslant
2 diam(M) d(x,y) (1 - \mu(E)) \leqslant \mu(E)md(x,y) - \frac{m}{2} d(x,y)
$$
which means $\mu(E) \geq \frac{\frac{m}{2} + 2 diam(M)}{m + 2 diam(M)}$, and by assumption this holds true.
Therefore we have $F(y) - F(x) \geq (m/2)d(x,y)$ as soon as $x,y$ are as in lemma \ref{lem:lemC},
and in particular $F$ is strictly increasing on the corresponding part of the geodesic.

From $|d(x,f(t))^2 - d(x,0)^2| \leq d(0,f(t)) 2 diam(M)$
we get
$$
-2 diam(M) d(0,f(t))  \leq d(x,f(t))^2 - d(x,0)^2 \leq 2 diam(M) d(0,f(t))
$$
for all $t \in \Omega$, and by integrating over $\Omega$ we get 
$|F(x) - d(x,0)^2| \leq 2 diam(M) d(0,f) \leq A/12$ for all $x \in M$,
since $\alpha \leq \frac{A}{24 diam(M)}$.

Let us set $r_2 = A/2$, and $r_1 = A/6$. With the notations of lemma \ref{lem:lemC}
we have $m = r_2 - 2 r_1 = A/6$. Let $r_0 = (r_2+r_3)/2 = 3A/4$,
and $\delta_0 = (r_3 -r_2)/2 = A/4$. By assumption on $A$ we have $\delta_0 \leq 1$.
Moreover $\delta_0$ is such that $|F(x) - d(x,0)^2| \leq A/12 = \delta_0/3$.

\begin{lemma}  \label{lem:riemphicont} {\ }
\begin{enumerate}
\item If $d(x,0) \leq r_2+\delta_0/3$, then $F(x) < r_0^2 - \delta_0/2$.
\item If $d(x,0) \geq r_3-\delta_0/3$, then $F(x) > r_0^2 + \delta_0$.
\item For all $v \in S^{n-1}$, there exists a unique $\varphi(v) \in [0, A[$ such that
$F(\exp_0(\varphi(v) v)) = r_0^2$. Moreover, we have $\varphi(v) \in ]r_2 + \delta_0/3,A- \delta_0/3[$.
\item The map $v \mapsto \varphi(v)$, $S^{n-1} \to ]r_2,A[$, is continuous.
\end{enumerate}
\end{lemma}
\begin{proof}
We have $F(x) \leq d(x,0)^2 + \delta_0/3$
hence if $d(x,0) \leq r_2+\delta_0/3 = r_0 - 2 \delta_0/3$ we have
$F(x) \leq (r_0 - 2 \delta_0/3)^2 + \delta_0/3 = r_0^2 - 4 \delta_0/3 + 4 \delta_0^2/9 + \delta_0/3
\leq r_0^2- 5 \delta_0/9  < r_0^2 - \delta_0/2$ (since $\delta_0^2 \leq \delta$).

Similarly, $F(x) \geq d(x,0)^2 - \delta_0/3$,
hence if $d(x,0) \geq r_3 - \delta_0/3 = r_0 + 2 \delta_0/3$, we have
$F(x) \geq r_0^2 + 4 \delta_0/3 + 4 \delta_0^2/9 - \delta_0/3
> r_0^2 + 4 \delta_0/3  - \delta_0/3 = r_0^2 +\delta_0$. This proves (1) et (2).

We now prove (3). It is clear that $F$ is continuous on $M$, therefore its composition with
the geodesic $t \mapsto \exp_0(t v)$ is continuous, too. Assume $t < r_3=A$. We
have $d(\exp_0(tv),0) = t$ and therefore we deduce from (1) and (2) that,
on the one hand $F(\exp_0(tv)) \neq r_0^2$ when $t \not\in 
 ]r_2 + \delta_0/3,r_3- \delta_0/3[$, and that, on the other hand,
 there exists  $\varphi(v) \in ]r_2 + \delta_0/3,r_3- \delta_0/3[$
 such that $F(\exp_0(\varphi(v) v))=r_0^2$. Since 
 $t \mapsto F(\exp_0(tv))$ is strictly decreasing on $]r_2,r_3[$
 this proves that $\varphi(v)$ is uniquely determined, and this proves (3).

We now prove (4). Since $F$ and $(t,v) \mapsto \exp_0(tv)$ are continuous,
we know that $(t,v) \mapsto F(\exp_0(tv))$ is continuous, too.
Therefore (see e.g. \cite{BOURBTOP} TG X.29, cor. 2 du théorème 3)
we notice that the map $h : v \mapsto (t \mapsto F(\exp(tv))$, $S^{n-1} \to C^0([0,r_3],\R^+)$
is continuous, too. We denote $d$ a metric on the sphere $S^{n-1}$ compatible
with its natural topology. Let us choose $v_0 \in S^{n-1}$, and prove
that $\varphi$ is continuous at $v_0$. Let $\eps >0$ small enough so
that $\eps < \delta_0/2$. By continuity of $h$ there exists $\eta >0$ such that,
if $v \in S^{n-1}$ satisfies $d(v,v_0) \leq \eta$, then
$|F(\exp_0(tv_0)) - F(\exp_0(tv))| \leq \eps/3$ for all $t \in [0,r_3]$.
Since $r_0^2 - \eps > r_0^2- \delta_0/2$ and $r_0^2 + \eps < r_0^2+\delta_0$,
by (1) and (2) there exists $a,b$ with $ r_2 < a < \varphi(v_0) < b < r_3$
such that $F(\exp_0(av_0)) = r_0^2 - \eps$ and $F(\exp_0(bv_0)) = r_0^2 + \eps$.
From $|F(\exp_0(bv)) - F(\exp_0(bv_0))| \leq \eps/3$ we get
$F(\exp_0(bv)) \geq F(\exp_0(bv_0)) - \eps/3 =r_0^2 + 2 \eps/3$
and from $|F(\exp_0(av)) - F(\exp_0(av_0))| \leq \eps/3$
that $F(\exp_0(av)) \leq F(\exp_0(av_0)) + \eps/3 = r_0^2 - \eps + \eps/3 = r_0^2 - 2 \eps/3$.
Since $t \mapsto F(\exp_0(tv))$ is strictly increasing we get $\varphi(v) \in ]a,b[$.
Since we know $\varphi(v_0) \in ]a,b[$ this implies $|\varphi(v)- \varphi(v_0) | \leq b-a = 2 \eps$.
This proves that $\varphi$ is continuous at every $v_0 \in S^{n-1}$, whence (4).

\end{proof}

For $f \in B(\tilde{0},\alpha)$, and $F$ as before, we define
$$
\Phi(f) = \{ x \in M ; F(x) < r_0^2 \}.
$$
Since $F$ is continuous it is an open subset of $M$. By lemma \ref{lem:riemphicont} (1) it
is non-empty, and therefore of positive measure. By lemma \ref{lem:riemphicont} (2)
it is included inside $B(0,A)$ hence $\Phi(f) \in \mathcal{B}(U)$. Let $g \in G$ be such
that $g.f \in B(\tilde{0},\alpha)$, and $F_g : x \mapsto \int d(g.f(t),x)^2 \dd t$
the associated map. Since $G$ acts by isometries we have
$d(g.f(t),x) = d(f(t),g^{-1}.x)$ for all $t,x$ hence $F_g(x) = F(g^{-1}.x)$.
Then, by definition 
$$\Phi(g.f) = \{ x \in M ; F_g(x) < r_0^2 \}
= \{ x \in M ; F(g^{-1}.x) < r_0^2 \} = \{ x \in M ; g^{-1}.x \in \Phi(f)\}
= g. \Phi(f).
$$
It remains to prove that $\Phi$ is continuous. This is done by the next lemma, where we prove that $\Phi$ is
$K'$-Lipschitz on $B(\tilde{0},\alpha)$ for some $K'>0$. Recall that $\mu_M$ denote the measure on $M$ hence also on $U$ associated to
the volume form $\dd p$. Note that $\Phi(f)$ always contains the ball with radius $r_2$ centered at $0$,
and is always contained inside the one with radius $r_3$.

\begin{lemma}
There exists $K,K' > 0$ such that,
for all $f_1,f_2 \in B(\tilde{0},\alpha)$, and $\varphi_1,\varphi_2 : S^{n-1} \to ]r_2,r_3[$ the (continuous) maps associated to them
by lemma \ref{lem:riemphicont},
we have
\begin{enumerate}
\item $\| \varphi_1 - \varphi_2 \|_{\infty} \leq \frac{4diam(M)}{m} d(f_1,f_2)$
\item $\mu_M( \Phi(f_1) \Delta \Phi(f_2)) \leq K \| \varphi_1 -  \varphi_2 \|_{\infty}$
\item $\mu_M( \Phi(f_1) \Delta \Phi(f_2)) \leq K' d(f_1,f_2)$
\end{enumerate}
\end{lemma}
\begin{proof}
We first prove (1). Note that this part does not involve $K,K'$.
First of all, we know that
$\| F_1 -  F_2 \|_{\infty} \leq 2 diam(M) d(f_1,f_2)$. Let us choose $v \in S^{n-1}$.
We have $| F_2(\exp_0(\varphi_1(v) v))-r_0^2| = | F_2(\exp_0(\varphi_1(v) v)) - F_1(\exp_0(\varphi_1(v) v))| \leq 2 diam(M) d(f_1,f_2)$.
We deduce from this that $| F_2(\exp_0(\varphi_1(v) v))-F_2(\exp_0(\varphi_2(v) v))| = | F_2(\exp_0(\varphi_1(v) v))-r_0^2| \leq 2 diam(M) d(f_1,f_2)$.
But we know that
$|F_2(\exp_0(\varphi_1(v) v))-F_2(\exp_0(\varphi_2(v) v))| \geq \frac{m}{2}d(\exp_0(\varphi_1(v)v),\exp_0(\varphi_2(v)v) = \frac{m}{2} |\varphi_1(v) - \varphi_2(v)|$, which proves
 (1).

Since (3) is a trivial consequence of (1) and (2), we only prove (2). 
Let us consider the diffeomorphism $G_0 : (t,v) \mapsto \exp_0(tv)$
for $(t,v) \in ]0,A^+[ \times S^{n-1}$, where $A^+ > A = r_3$ is the injectivity radius at $0$,
and let us write the pullback of the volume form $\dd p$ on $M$
as $G_0^* \dd p = a(t,v) \dd t \dd v$, with $a : ]0,A^+[ \times S^{n-1} \to \R_+$ a continuous map,
where $\dd t, \dd v$ are the natural
volume forms on $\R$ and $S^{n-1} \subset T_0 M$.
We let $\| a \|_{\infty}$ denote the supremum of $|a|$ on $[r_2,r_3]$.
Let us set $K = \| a \|_{\infty} \la(S^{n-1})$ where $\la$
is the Lebesgue measure on $T_0 M \simeq \R^n$.

By abuse of notation we will denote $[a,b] = [\min(a,b),\max(a,b)]$ even when $a > b$. Then,
for $E= \Phi(f_1) \Delta \Phi(f_2) \subset G_0([r_2,r_3] \times S^{n-1})$ we have that
$\mu_M(\Phi(f_1) \Delta \Phi(f_2))$ is equal to
$$  \int_{[r_2,r_3] \times S^{n-1}} a(t,v)\un_{G_0^{-1}(E)}(t,v) \dd t \dd v
= \int_{[r_2,r_3] \times S^{n-1}} a(t,v)\un_{\{ (u,w) \ | \ u \in [\varphi_1(w),\varphi_2(w)]\} }(t,v) \dd t \dd v
$$
whence
$$
\mu_M(C(f_1) \Delta C(f_2)) \leq 
\| a \|_{\infty} \int_{[0,r_3] \times S^{n-1}} \un_{\{ (u,w) \ | \ u \in [\varphi_1(w),\varphi_2(w)]\} }(t,v) \dd t \dd v$${}$$
\leq \| a \|_{\infty} \int_{S^{n-1}} \left( \int_0^{r_3}  \un_{\{ (u,w) \ | \ u \in [\varphi_1(w),\varphi_2(w)]\} }(t,v) \dd t \right) \dd v
= \| a \|_{\infty} \int_{S^{n-1}} |\varphi_2(v)- \varphi_1(v)| \dd v 
$$
and this yields $\mu_M(\Phi(f_1) \Delta \Phi(f_2)) \leq \| a \|_{\infty} \| \varphi_2 - \varphi_1 \|_{\infty} \la (S^{n-1}) = K \| \varphi_2 - \varphi_1 \|_{\infty}$. This proves (2).

\end{proof}

\subsection{Main result}

If $M$ is a metric space, we still denote $L(M) = L(\Omega,M)$ the set
of (equivalence classes of) Borel maps $f : \Omega \to M$ such that $\int d(f(t),p) \dd t < \infty$
for one (and hence for all) $p \in M$. If $G$ is a topological group acting isometrically on $M$,
and $M$ is bounded, then $G$ acts isometrically on $L(\Omega,M)$, under $g.f = (t \mapsto g.f(t))$.
Moreover,
if $G$ is a metric group, then $L(\Omega,G)$ acts isometrically on $L(\Omega,M)$
by $g.f = (t\mapsto g(t).f(t))$. We first note the following.

\begin{proposition}  Let $G$ be a locally compact metric group with a countable base acting by isometries on 
the compact
metric space $M$. 
Assume that $G$ is separable of finite dimension.
If the action of $G$ is transitive, then so is the induced action of $L(\Omega,G)$
on $L(\Omega,M)$.
\end{proposition}
\begin{proof}
Let us choose $x_0 \in M$, and let $f_0 \in L(\Omega,M)$ be the constant map $t \mapsto x_0$.
We let $G_0 = \{ g \in G ; g.x_0 = x_0 \}$ denote the fixer of $x_0$. It is a closed subgroup of
$G$, and the map $g \mapsto g.x_0$ induces an injective continuous map $\psi : G/G_0 \to M$.
Since the action of $G$ on $M$ is transitive we know that $\psi$ is bijective. Under our assumptions
 (\cite{HELGASON1}, \S 3 theorem 3.2) we know that it is an homeomorphism. Now, the projection map $G \to G/G_0$
 admits local cross-sections (see \cite{MOSTERT}) since $G$ is metric separable of finite dimension. 
 For each $x \in G/G_0$, there exists $\eps_x > 0$ and a local section $s_x : B(x,\eps_x) \to G$
 of the projection map, where $B(x,r)$ denotes the open ball of radius $r$ for the induced metric.
 Since the collection $B(x,\eps_x/2), x \in G/G_0$ is a covering of the compact
 space $G/G_0$, there exists $x_1,\dots,x_n$ such that $B_1,\dots,B_n$ is a covering of $G/G_0$,
 where $B_i = B(x_i,\eps_{x_i}/2)$. We denote $s_i = s_{x_i}$.
 
 Let $f : \Omega \to M$ be a Borel map. Then we define $g : \Omega \to G$
 by $g(t) = s_i(  \psi^{-1} ( f(t)))$ if $i$ is the minimal $r \in \{1,\dots, n\}$ such that $\psi^{-1}(f(t)) \in B_r$.
 Letting $X_i = f^{-1}(\psi(B_i))$ we know that $X_i$ is Borel since $\psi$ is an homeomorphism and $B_i$
 is open, and therefore so is $Y_i = X_i \setminus \bigcup_{r < i} X_i$. Since $g$ coincides on $Y_i$
 on the Borel map $s_i \circ \psi^{-1} \circ f$ we get that that $g$ is Borel.
 Since each $s_i$ is continuous on the compact set $\overline{B_i}$, the map $t \mapsto g(t)$
 is bounded, and therefore $g \in L(\Omega,B)$. One checks readily that $g.f_0 = f$, and this proves the claim.
\end{proof}

We want to provide a geometric proof of the following theorem (which readily follows from Gleason's general
theorem on actions of compact Lie groups).

\begin{theorem}
Let $G$ be a compact Lie group endowed with a Haar measure. Then,
the natural projection map $L(\Omega,G) \to L(\Omega,G)/G$ admits local cross-sections.
\end{theorem}

Let $x_0 \in G$, and $f_0 \in L(\Omega,G)$. There exists $g_0 \in L(\Omega,G)$ such that
$f_0g_0 = \tilde{x}_0$. Let $\bar{x}_0$, $\bar{f}_0$ the images of $\tilde{x}_0$ and $f_0$, respectively,
inside $L(\Omega,G)/G$. If we can find $\alpha>0$ and a local section $s : B(\bar{x}_0,\alpha) \to L(\Omega,G)$,
then the map $f \mapsto s(\overline{f g_0})g_0^{-1} $ defines a section $B(f_0,\alpha) \to L(\Omega,G)$.
Therefore, we can assume $f_0 = x_0$, and even $x_0 = 1$.

We denote $\dd p$ the volume form associated to the Haar measure of $G$.
Let $U$ be the open neighborhood of $p_0 = 1$ provided by proposition \ref{prop:masscont}. We have a continuous map $c : \mathcal{B}(U) \to
U$, where $c(E)$ is the unique element of $U$ minimizing $F : x \mapsto \int_E d(x,p)^2 \dd p$. If $g \in G$ and $E \in \mathcal{B}(U)$
satisfy $g.E \subset U$, then $c(g.E) \in U$, and $c(g.E)$ minimizes 
$$x \mapsto \int_{g.E} d(x,p)^2 \dd p = \int_E d(x,g.p)^2 \dd p
= \int_E d(g^{-1}.x,p)^2 \dd p = F(g^{-1}.x)$$
over $M$. But $F(g^{-1}.x)$ is minimal iff $g^{-1}.x = c(E)$ that is $x = g .c(E)$, whence $c(g.E) = g.c(E)$.

By proposition \ref{prop:esssupport} there exists $\alpha >0$ and a continuous map $\Phi : B(\tilde{x}_0,\alpha) \to \mathcal{B}(U)$
such that if $f \in B(\tilde{x}_0,\alpha)$ and $g \in G$ satisfy $g.f \in B(\tilde{x}_0,\alpha)$, then
$\Phi(g.f) = g.\Phi(f)$. From this we get a continuous map $\Psi : B(\tilde{1},\alpha) \to U \subset G$ defined
by $\Psi(f) = c(\Phi(f))$ with the property that, for every $f,g$ with $f \in B(\tilde{x}_0,\alpha)$ and $g \in G$ satisfying
$g.f \in B(\tilde{x}_0,\alpha)$ we have $\Psi(g.f) = g.\Psi(f)$.
Now let us consider the continuous map $\sigma : B(\tilde{1},\alpha) \to L(\Omega,G)$ defined by $f \mapsto \Psi(f)^{-1} .f$
and denote $B(\bar{1},\alpha)$ the ball in $L(\Omega,G)/G$ of center the image $\bar{1}$ of $\tilde{1}$ and radius $\alpha$.
In order to check that $\sigma$ factorizes through the natural projection map $B(\tilde{1},\alpha) \to B(\bar{1},\alpha)$
it is sufficient to have $\sigma(g.f) = \sigma(f)$ whenever $g \in G$ satisfies $g.f \in B(\tilde{1},\alpha)$, and we already checked
that this is true. Therefore $\sigma$ provides a local cross-section and this proves the theorem.

\begin{corollary} If $G$ is a compact Lie group endowed with a
bi-invariant metric, then $L(\Omega,G)/G$ is a classifying space for the topological group $G$
inside the category of paracompact spaces.
\end{corollary}
\begin{proof}
By the theorem we know that the projection map
 $L(\Omega,G) \to L(\Omega,G) /G$ admits a local cross-section. 
 Proposition \ref{prop:genclassifiant} implies the conclusion.
 \end{proof}

\section{Exponentiations}
\label{sect:exponential}

In this section we prove results of the form $L(\Omega,L(\Omega,G)) \simeq
L(\Omega \times \Omega,G)$, for a sufficiently large class of groups $G$.
By construction, this bijection is a $G$-equivariant isometry. We start by some
preliminaries, then establish the case where $G$ is discrete (but not necessarily countable), and  then when $G$ is a Borel subset of a compact metric space. This covers the case
of compact Lie groups and of metric profinite groups. Finally we investigate some
interesting subspaces of $L(\Omega,L(\Omega,G))$.

\subsection{Dense subsets of $L(2)$}

\label{sect:prelimL2}

We denote $\mathcal{L}(2) = \mathcal{L}(\Omega,2)$ and
$L(2) = L(\Omega,2) = L(\Omega,\{ 0, 1 \})$.
We identify them with the (classes of) Borel subsets of 
 $[0,1]$, endowed with the pseudo-distance $(X,Y) \mapsto \la(X \Delta Y)$,
 where $\la$ is the Lebesgue measure on $[0,1]$.

We let $\mathcal{P}_n \subset \mathcal{L}(2)$
the set of all the Borel subsets which are unions of at most $n$ closed
intervals of $[0,1]$. Clearly $\mathcal{P}_n \subset \mathcal{P}_{n+1}$ for
all $n \geq 1$. We let $\mathcal{G}_n \subset \mathcal{L}(2)$ the set
of maps $[0,1] \to \{0,1 \}$ which are constant on each interval
of the form $[\frac{k}{n},\frac{k+1}{n}[$, for $0 \leq k < n$. 
We let $P_n, G_n$ denote the images of $\mathcal{P}_n, \mathcal{G}_n$
in $L(2)$. Clearly $G_n \subset P_n$.

\begin{lemma} \label{lem:pavescompacts}
For all $n$, $\mathcal{P}_n$ is a closed subset of $\mathcal{L}(2)$,
and is metacompact. In particular its image inside $L(2)$ is compact.
\end{lemma}
\begin{proof}
It is enough to show that the image $P_n$ of $\mathcal{P}_n$
inside $L^1$ is compact. Let $B = [0,1]^{2n}$, endowed with the
usual topology, and $D \subset B$ being defined by
$$
\{ (x_1,\dots,x_n,e_1,\dots,e_n) \ | \ x_i + e_i \leq 1 \}.
$$
Since $D$ is closed inside $B$, it is compact. We consider
$\Phi : D \to \mathcal{L}^1$ defined by
$$
X = (x_1,\dots,x_n,e_1,\dots,e_n) \mapsto \bigcup_{i=1}^n [x_i,x_i+e_i].
$$
If $X' = (x'_1,\dots,x'_n,e'_1,\dots,e'_n) \in D$, with $|x'_i - x_i| \leq \eps$, $|e'_i-e_i| \leq \eps$, then
$$
\Phi(X) \cup \Phi(X') \subset \bigcup [x_i-\eps, x_i + e_i + \eps]
$$
et 
$$
\Phi(X) \cap \Phi(X') \supset \bigcup [x_i+\eps, x_i + e_i - \eps]
$$
hence, if $t \in  \Phi(X) \Delta \Phi(X')$, then
$\exists i \ \ x_i - \eps \leq t \leq x_i + e_i + \eps$
and $\forall i \ \ t < x_i+\eps $  ou $t > x_i + e_i-\eps$.
It follows that there exists $i$ such that $t \in [x_i-\eps,x_i+\eps] \cup 
[x_i + e_i-\eps, x_i + e_i+\eps]$.
Therefore,
$$
\Phi(X) \Delta \Phi(X') \subset \bigcup_{i=1}^n \left([x_i-\eps,x_i+\eps] \cup 
[x_i + e_i-\eps, x_i + e_i+\eps] \right) 
$$
whence $\la(\Phi(X) \Delta \Phi(X') ) \leq 4n \eps$,
thus proving that $\Phi$ is continuous. Then $P_n$ is the image of a compact set under a continous map, and therefore is compact.
\end{proof}

\begin{lemma} \label{lem:pavesdenses} $\bigcup_{n \geq 1} P_n$ is dense inside $L(2)$. If $(y_n)_n$ is a positive integer-valued sequence tending to $\infty$,
then $\bigcup_n G_{y_n}$ is dense inside $L(2)$.
\end{lemma}
\begin{proof} Let $B \in \mathcal{L}(2)$ be a Borel set. Since $B$
is a Lebesgue-measurable set, its
Lebesgue measure coincides with its exterior
measure, namely
$
\la(B) = \inf \{ \la(C) ; B \subset C =  \bigcup_{i \geq 1}E_i, E_i\  \mathrm{intervals} \ \}$. Let $\eps > 0$, and $C = \bigcup_{i \geq 1} E_i$
such that $B \subset C$ and $\la(B \Delta C) \leq \eps$. Letting
$C_n = \bigcup_{1 \leq i \leq n} E_i$ we have $\la(C_n \Delta C) \to 0$
hence there
exists $n$ such that $\la(C_n \Delta C) \leq \eps$. Therefore
$\la(B \Delta C_n) \leq 2 \eps$. Since $C_n \in \mathcal{P}_n$
this proves the first claim. Let $m \in \N$ and $A \in \mathcal{P}_m$.
Then it is clear that, for $n$ large enough, $A$ can be approximated
by an element of $\mathcal{G}_{y_n}$, and this proves the second claim.

\end{proof}

\subsection{Discrete Exponentiation theorem}

The goal of this section is to prove the following result, which can be viewed as a partial
analog, for $X$ discrete, of the exponentiation theorem $(X^Y)^Z \simeq X^{Y \times Z}$ in
topology (e.g. in the realm of Hausdorff locally compact spaces).

\begin{theorem} \label{theoskolem} Let $X,Y$ be two non-atomic probability spaces ($X = Y = \Omega$), and $(D,e)$ be a pointed discrete metric space.
There exists a morphism $\mathcal{L}(X \times Y,D) \to \mathcal{L}(X,\mathcal{L}(Y,D))$ inducing
a bijective isometry $L(X\times Y,D) \to L(X,L(Y,D))$. This isometry is equivariant with respect to the obvious actions of $\mathfrak{S}(D)$.
\end{theorem}

Let $f \in \mathcal{L}(X \times Y,D) $, that is a Borel map $X \times Y \to D$.
By proposition \ref{prop:imagedenombrable} it has countable image $f(I) =D_0$, hence $f \in \mathcal{L}(X\times Y,D_0)$.
As a consequence, every $f^x : y \mapsto f(x,y)$
is a Borel map $Y \to D_0$ if we can prove that $(f^x)^{-1}(\{ d_0 \})$
is Borel for every $d_0 \in D_0$. But $(f^x)^{-1}(\{ d_0 \}) = \{ y ; (x,y) \in f^{-1}(\{ d_0 \}) \})$
is a section of the Borel set $f^{-1}(\{ d_0 \}) \}) \subset X \times Y$, and is
therefore Borel.

This defines a map $F : x \mapsto f^x$, $X \to \mathcal{L}(Y,D_0) \subset \mathcal{L}(Y,D)$. We want to
show that this is a Borel map, with respect to the (non-Hausdorff) topology on $\mathcal{L}(Y,D)$ defined by the natural
pseudo-metric. For this we need to show that $F^{-1}(U)$ is Borel for every open subset $U$ in $\mathcal{L}(Y,D)$ or, equivalently of $\mathcal{L}(Y,D_0)$, since $F$ takes values
inside $\mathcal{L}(Y,D_0)$.
Since $D_0$ is countable, $\mathcal{L}(Y,D_0)$ is separable (see proposition \ref{prop:LEcontractcomplete}) ; because it is
pseudo-metric, its topology admits a countable basis of open sets made of open balls. We thus need to show
that $F^{-1}(U)$ is Borel for every ball 
$U$, with center $g \in \mathcal{L}(Y,D_0)$ and radius $\alpha > 0$.
Then $F^{-1}(U) = \{ x \in X ; d(f^x,g) < \alpha \} = \{ x \in X ; \mu( \{ y; d(f(x,y),g(y)) \neq 0 \}) < \alpha \}$.
Let $A = \{ (x,y) \in X \times Y ; d(f(x,y),g(y)) = 1 \}$. Since $f : X \times Y \to D_0$ and $G : (x,y) \mapsto g(y)$ are
Borel,
then so is
$f \times G : X \times Y \to D_0 \times D_0$. Since $d : D_0 \times D_0 \to \{ 0,1 \}$ is continuous,
we get that $d \circ (f_0 \times G)$ is Borel-measurable, therefore $A = (d \circ (f_0 \times G))^{-1}(\{ 1 \})$ is a Borel set
($\{ 1 \}$ being open in $\{0,1\}$). It follows that all its sections $A_x = \{ y \in Y; (x,y) \in A \} = \{ y \in Y ; d(f_0(x,y),g(y))=1 \}$
are Borel-measurable, and so is the map $\mathcal{A} : x \mapsto \mu(A_x)$, $\mathcal{A} : X \to \R_+$. It follows that
$F^{-1}(U) = \mathcal{A}^{-1}([0,\alpha[)$ is a Borel set. This proves that $F$ is 
a Borel map.

We have thus defined a map $\mathcal{L}(X \times Y,D) \to \mathcal{L}(X,\mathcal{L}(Y,D))$ given by $f \mapsto F$.
Since the Borel subsets of $\mathcal{L}(Y,D)$ are the inverse images by the natural projection of the Borel subsets of $L(Y,D)$,
we have a natural map $\mathcal{L}(X,\mathcal{L}(Y,D)) \to \mathcal{L}(X,L(Y,D))$. By composition,
we get a map $\Psi : \mathcal{L}(X \times Y,D) \to \mathcal{L}(X,\mathcal{L}(Y,D)) \to \mathcal{L}(X,L(Y,D)) \to L(X,L(Y,D))$.
Let now $f,h \in \mathcal{L}(X \times Y, D)$ having the same image under $\Psi$ inside
$L(X \times Y,D)$, meaning that $\{ (x,y) ; f(x,y) \neq h(x,y) \}$ has measure $0$, or in
other terms $\int_{X \times Y} d(f(x,y),h(x,y)) \dd x \dd y = 0$.
We can assume $f,h \in \mathcal{L}(X \times Y,D_0)$
for the same countable subset $D_0$ of $D$.
By the Fubini-Tonelli theorem we get $\int_X \left( \int_Y d(f(x,y),h(x,y))\dd y \right) \dd x = 0$, that is
$$
0 = \int_X \left( \int_Y d(f^x(y),h^x(y)) \dd y \right) \dd x = \int_X d_{L(Y,D)} (f^x,h^x) \dd x = d_{L(X,L(Y,D))} (\Psi(f),\Psi(h)).
$$
The map $\Psi$ thus induces a map $L(X \times Y,D) \to L(X,L(Y,D))$, that we still denote by $\Psi$, and which is clearly isometric by the same computation
following from Fubini-Tonelli's theorem. We want to prove that it is surjective.

Let $f \in \mathcal{L}(X,\mathcal{L}(Y,D))$, that we identify with a map $X \times Y \to D$. We need to prove that
there exists $\hat{f} \in \mathcal{L}(X \times Y,D)$ with the same image inside $L(X,L(Y,D))$. 
Let $f_0$ be the image of $f$ inside $\mathcal{L}(X,L(Y,D))$. Since $L(Y,D)$ is metric, 
 by corollary \ref{cor:probalindelof} we know that $f_0(X)$ is separable. Let $(g_n)_{n \in \N}$
 be a dense sequence inside $f_0(X)$, and $\tilde{g}_n$ a representative of $g_n$ inside
 $\mathcal{L}(Y,D)$. By proposition \ref{prop:imageseparable} we know
 that each $\tilde{g}_n(Y)$ is countable. Let $D_0 = \bigcup_n \tilde{g}_n(Y)$. 
 Since $L(Y,D_0)$ is closed inside $L(Y,D)$, and $g_n \in L(Y,D_0)$ for all $n$,
 we get $f_0(X) \subset L(Y,D_0)$, hence $f \in L(X,L(Y,D_0))$. We can thus assume that $D$
 is countable.

We show that we can reduce our problem to the case $D_0 = \{0,1 \}$. Indeed, let us
associate to $d \in D_0 \setminus \{ e \}$ the map $\varphi_d : D_0 \to \{0,1 \}$ given
by $d \mapsto 1$ and $d' \mapsto 0$ if $d'\neq d$. The set $\{0,1\}$ being endowed with the discrete metric,
this is a 1-Lipschitz map, and therefore by lemma \ref{lemnatuprojlipsch} there are 1-Lipschitz induced maps $\check{\varphi}_d : \mathcal{L}(X,\mathcal{L}(Y,D_0))
\to \mathcal{L}(X,\mathcal{L}(Y,\{0,1\}))$.

Let us assume that we know how to build a $f_d \in \mathcal{L}(X \times Y,\{0,1\})$ having the same image as $\check{\varphi}_d(f)$ inside $L(X,L(Y,\{0,1\}))$. This provides $\tilde{f}_d \in \mathcal{L}(X \times Y,
\{e,d\}) \subset \mathcal{L}(X \times Y,
D)$ by using the 1-Lipschitz obvious map $\{e,d \} \to D$.
We can assume that $D$ is an interval of $\N = \Z_{\geq 0}$, with $e$ being identified to $0$. We let $f^{(0)} \in \mathcal{L}(X \times Y,D)$ be the constant map $(x,y) \mapsto e$ and in general $f^{(n)}(x,y) = f(x,y)$ if $f(x,y) \leq n$, $f^{(n)}(x,y) = e$ otherwise. We have $f^{(n+1)} = \max (f^{(n)}, \tilde{f}_{n+1})$
hence by induction the $f^{(n)}$ are Borel. Then $f$ is a limit of a sequence of Borel maps $X \times Y \to D$ and therefore is Borel.

We can thus assume $D = \{0,1 \}$. In this case this will be a consequence of the following proposition.

\begin{proposition} \label{prop:approxborelcarre}
Let $\la$ be the Lebesgue measure on $[0,1]$ and $A \subset [0,1]^2$ such that
\begin{itemize}
\item for all $x \in [0,1]$ the section $A_x = \{ y \in [0,1] ; (x,y) \in A \}$ is Borel
\item the map $x \mapsto A_x$ is a Borel map $[0,1] \to L^1(\Omega,\{0,1 \})$
\end{itemize}
then there exists $B \subset [0,1]^2$ Borel
such that $\int_0^1 \la(A_x \Delta B_x) \dd x = 0$.
\end{proposition}
\begin{proof}
Let $\mathcal{T}$ denote the set of subsets $A \subset [0,1]^2$ satisfying these two assumptions. We denote $\la_2$ the Lebesgue measure on $[0,1]^2$.
By the Fubini-Tonelli theorem, $\mathcal{T}$ contains the Borel subsets of $[0,1]^2$. Moreover, it is a $\sigma$-algebra, and there is a measure $m : \mathcal{T} \to \R_+$ given by $A \mapsto \int_0^1 \la(A_x) \dd x$ that coincides
with the Lebesgue measure $\la_2$ on the Borel sets, again by Tonelli's theorem. We have a pseudo-distance $d : (X,Y) \mapsto m(X \Delta Y)$ on $\mathcal{T}$
that extends the $\mathcal{L}^1$
pseudo-distance on the Borel $\sigma$-algebra.

As in section \ref{sect:prelimL2} we let $\mathcal{P}_n$ denote the set of all the Borel subsets which are unions of at most $n$ closed intervals
in $[0,1]$, and $P_n$ its image in $L^1$. By lemma \ref{lem:pavescompacts}
the $P_n$ are compact, and by lemma \ref{lem:pavesdenses} the set $\bigcup_n P_n$
is dense. We let $\mathcal{P}_n(\eps) = \{ x \in \mathcal{L}^1 ; d(x,P_n) < \eps \}$.

Let $\alpha > 0$. Since $\bigcup_{n \geq 1} \mathcal{P}_n$
is dense inside $\mathcal{L}^1$, we have $\mathcal{L}^1 = \bigcup_{n \geq }
\mathcal{P}_n(\alpha/3)$. Since $x \mapsto d(x,P_n)$ is continuous,
each $\mathcal{P}_n(\alpha/3)$ is Borel, and we get
$\la(\mathcal{P}_n(\alpha/3)) \to 0$. Therefore, there exists $n_0$
such that $\la( ^c \mathcal{P}_n(\alpha/3) ) \leq \alpha/3$.

Since $\mathcal{P}_{n_0}$ is metacompact, there exists $Q_1,\dots,Q_m \in
\mathcal{P}_{n_0}$ such that $\mathcal{P}_{n_0} \subset \bigcup_{1 \leq i \leq m} Q_i(\alpha/3)$,
where $Q_i(\eps) = \{ x \in \mathcal{L}_1 ; d(x,Q_i) \leq \eps \}$.
Therefore $\mathcal{L}^1 = \bigcup_{1 \leq i \leq m} Q_i(2 \alpha/3)$,
and there is a well-defined map $p : \mathcal{L}^1 \to \{1,\dots, m \}$
given by $x \mapsto \min \{ i ; x \in Q_i(2 \alpha/3) \}$.
We have $p^{-1}(\{ i \}) = Q_i(2 \alpha/3) \setminus \bigcup_{1 \leq j < i } Q_j(\alpha)$. Since the $Q_i(\eps)$ are closed, each $p^{-1}(\{ i \})$
is Borel, and therefore $p$ is a Borel map. 

We let $B_{\alpha} = \{ (x,y);  A_x \in \mathcal{P}_{n_0}(\alpha/3) , y \in Q_{p(A_x)}\}$. We have
$$
B_{\alpha} = \bigcup_{i=1}^{m} \{ x ; A_x \in \mathcal{P}_{n_0}(\alpha/3)
\cap Q_i(2 \alpha/3) \} \times Q_i.
$$
Since $x \mapsto A_x$ is Borel, $\{ x ; A_x \in \mathcal{P}_{n_0}(\alpha/3)
\cap Q_i(2 \alpha/3) \} $ is Borel and therefore $B_{\alpha}$
is a Borel subset of $[0,1]^2$.

We have 
$$
d(B_{\alpha},A) \leq \int_{x \in ^c \mathcal{P}_{n_0}(\alpha/3)} d((B_{\alpha})_x,A_x) \dd x + \int_{x \in \mathcal{P}_{n_0}(\alpha/3)} d((B_{\alpha})_x,A_x) \dd x \leq \alpha/3 +  \int_{x \in \mathcal{P}_{n_0}(\alpha/3)} d((B_{\alpha})_x,A_x) \dd x. $$
Moreover, since $d((B_{\alpha})_x,A_x) = d(Q_i,A_x)\leq  2 \alpha/3$
for all $x \in \mathcal{P}_{n_0}(\alpha/3)$ such that $p(A_x) = i$,
we get
$$
\int_{x \in \mathcal{P}_{n_0}(\alpha/3)} d((B_{\alpha})_x,A_x) \dd x
\leq \sum_{i=1}^m \la(\{x ; A_x \in \mathcal{P}_{n_0}(\alpha/3) \cap p^{-1}(\{ i \})\}) \times (2 \alpha/3) \leq 2 \alpha/3 \times \la([0,1]) = 2 \alpha/3.
$$
hence $d(B_{\alpha}, A) \leq \alpha$.

From this we get that the sequence $B_{\frac{1}{n}}$ converges
to $A$ for the distance $d$. In particular, it is a Cauchy sequence
for the distance $d$, and therefore in the usual $L^1$. Since $L^1([0,1]^2,
\{ 0, 1\})$ is complete, there is a Borel subset $B$ of $[0,1]^2$
such that $B_{\frac{1}{n}} \to B$. But then $d(A,B_{\frac{1}{n}}) \leq 1/n$
converges both to $0$ and to $d(A,B)$, and this proves the claim.

\end{proof}

\begin{remark} 
Under the assumption of the proposition, $A$ itself does not need to
be Lebesgue-measurable. For instance, under the Continuum Hypothesis
one could endow $[0,1]$ with some well-ordering $\preceq$
by identifying it with the first non-countable ordinal, and
one may consider $A = \{ (x,y) \in [0,1]^2 \mid y \prec x \}$.
Then, for every $x \in [0,1]$, $A_x = \{ y \in [0,1] : y \prec x \}$
is countable and therefore the assumptions are satisfied. However,
it is not Lebesgue-measurable, because otherwise one
would have $\la_2(A) = \int_x \la(A_x) \dd x = 0$ and $\la_2(A) = \int_y \la(\{ x ; y \prec x \}) \dd y = 1$,
contradicting the Fubini-Tonelli theorem. This example was communicated to me by A. Rivi\`ere.

Another example, not using the Continuum Hypothesis, was constructed for me
by G. Godefroy and J. Saint-Raymond. It runs as follows. Let $\mathcal{K}^+$ denote
the set of compact subsets of $[0,1]^2$ of positive measure. Its cardinality is
the continuum, and therefore we can describe it as $\{ K_{\alpha}, \alpha \in [0,1] \}$.
Let $p_1 : [0,1]^2 \to [0,1]$ be the projection on the first coordinate. Each $p_1(K_{\alpha})$
has to be uncountable, for otherwise $K_{\alpha}$ would have measure $0$. Since $p_1(K_{\alpha})$ is compact, this implies $|p_1(K_{\alpha})| = \mathfrak{c}$. We endow $[0,1]$
with a well-ordering $\prec$ (of type $\mathfrak{c}$). We then construct elements $x_{\alpha}, y_{\alpha}$
by transfinite induction on $([0,1],\prec)$. Assuming the $x_{\beta},y_{\beta}$ constructed
for $\beta \prec \alpha$, we chose $x_{\alpha} \in p_1(K_{\alpha}) \setminus \{ x_{\beta}, \beta \prec \alpha \}$, which is non-empty since $\alpha$ is smaller than the continuum. We then choose
$y_{\alpha}$ such that $(x_{\alpha},y_{\alpha}) \in K_{\alpha}$ and set $A = \{ (x_{\alpha},y_{\alpha}) ; \alpha \in [0,1] \}$. Then $|A_x| \leq 1$ for all $x \in [0,1]$,
therefore $A$ satisfies our assumptions. Moreover, if it is measurable, it should have measure $0$
by Fubini-Tonelli ; on the other hand, since $A$ meets all compacts of positive measure in $[0,1]$, it should have exterior measure $1$, thus proving that $A$ is not Lebesgue-measurable.

\end{remark}

By this proposition, identifying an element $f \in \mathcal{L}(X,\mathcal{L}(Y,D))$ with the characteristic function of $A \subset [0,1]^2$,
we get that there exists $B \subset [0,1]^2$
 Lebesgue-measurable with the same image in $L(X,L(Y,D))$.
Moreover there exists two Borel sets $B_1 \subset A \subset B_2$ such that
$B_2 \setminus B_1$ has measure $0$, and their characteristic function
also have the same image.
This concludes the
proof of the theorem.

\subsection{Boreleanity of Borel inverse images}

This section is devoted to the proof of the following technical proposition, that we will need in the sequel.
The proof we provide here was found for us by G. Godefroy.

\begin{proposition} \label{prop:supportborel} 
The map $\mathcal{L}([0,1],\R_+) \to L(2)$ defined by associating 
to a Borel map $g : [0,1] \to \R_+$ the (class of the characteristic function of)
the Borel set $\{ t ; g(t) > 0 \}$ is Borel.
\end{proposition}

In order to prove this proposition, we first recall the following classical lemma, for which we could not
find a proper reference.

\begin{lemma} \label{lem:fctseparantes} Let $Y$ be a standard Borel space, and $X$ a
measurable space. Assume we are given a countable collection
of Borel maps $\varphi_n : Y \to \R$ which separate the points of $Y$
(that, is $\forall y_1,y_2 \in Y \ y_1 \neq y_2 \Rightarrow \exists n \ \ \varphi_n(y_1) \neq \varphi_n(y_2)$). Then, a map $f : X \to Y$
is measurable if and only if $\varphi_n \circ f$ is measurable for all $n \in \N$.
\end{lemma}
\begin{proof}
Let $\iota : Y \to \R^{\N}$ given by $\iota(y) = (\varphi_n(y))_{n \in \N}$, where
$\R^{\N}$ is given the product topology. By
assumption, $\iota$ is 1-1. Then $\iota$ is measurable and injective.
By \cite{CHRISTENSEN} theorem 2.4 this implies that $\iota$ is
a Borel isomorphism $Y \to \iota(Y)$. In particular, $f : X \to Y$
is measurable if and only if $\iota \circ f : X \to \R^{\N}$ is measurable.
But this means that $\varphi_n \circ f$ is measurable for all $n \in \N$.
\end{proof}

Now, $L^1([0,1],\R_+)$ is a Polish space and is therefore standard Borel. We
consider the countable collection of maps 
$\varphi_I : f \mapsto \int_I f$ 
where $I$ runs among the open sub-intervals of $[0,1]$
with rational endpoints.

It is separating because, if $\int_I f = \int_I g$ then $\int_I(f-g) = 0$,
hence $\int(f-g)^+ = \int(f-g)^-$ and we need to prove $\int_I f = \int_I g
\Rightarrow f =g$ only when $f,g \geq 0$. But then the finite positive measures $E \mapsto \int_E f$ and $E \mapsto \int_E g$ coincide
on the $\pi$-system of all intervals with rational bounds. Since this
$\pi$-system generates the Borel $\sigma$-algebra, they coincide
on all Borel sets $E$. Therefore we need to prove $\forall E\ \int_E f = 0 \Rightarrow f = 0$. Considering the Borel
subsets $E_1 = \{ f > 0 \}$ and $E_2 = \{ f < 0 \}$ we get $f = 0$ and the conclusion.

Then, we prove that the maps $\Psi_I = \varphi_I \circ F$ are continuous (and therefore Borel). Indeed, we
have
$$
\Psi_I : g \mapsto \int_I F(g) = \la(\{ t; g(t)>0 \})
$$ 
where $\la$ is the Lebesgue measure on $[0,1]$. Since the topology of $\R$ is generated by the open subintervals $]a,+\infty[$, we need to check that the
$E_{I,a} = \{ g \in L^1([0,1],\R^+) ; \Psi_I(g) > a \}$ are open. If they are not, there
would exists a sequence $g_n \not\in E_{I,a}$ converging to $g \in E_{I,a}$
in $L^1(([0,1],\R^+)$. Since $g \in E_{I,a}$ we have $\la(I \cap \{ t; g(t)>0 \}) > a$ hence there exists $\eps>0$ such that $\la(I \cap \{ t; g(t)>\eps \}) > 
a + \eps$. Since $g_n$ converges to $g$ inside $L^1([0,1],\R^+)$,
there exists a subsequence pointwise converging to $g$ almost everywhere
which belongs to $L^{\infty}([0,1],\R^+)$
Without loss of
generality we can replace $g_n$ by this subsequence, and therefore
$g_n$ converges to $g$ almost everywhere. Since $\la([0,1])<\infty$ This implies that $g_n$
converges to $g$ in measure
meaning that, for all $\eta > 0$, $\la\{ t ; |g_n(t)- g(t)| \geq \eta \} \to 0$. In particular there exists $n_0$ such that
$\la \{ t; |g_n(t)-g(t)| \geq \eps/2 \} \leq \eps/2$ for all $n \geq n_0$.
But $I \cap \{ t ; g_n(t)>0 \}$ contains $(I \cap \{ t; g(t)>\eps \}) \setminus
\{ t; |g_n(t)-g(t)| \geq \eps/2 \}$, which has measure at least $a+\eps - \eps/2 > a$. Therefore $g_n \in E_{I,a}$ for $n \geq n_0$, a contradiction.
This concludes the proof of the proposition.

\subsection{Continuous exponentiation theorem}

\begin{proposition} \label{prop:defexpcont}
Let $X,Y$ be two non-atomic probability spaces ($X = Y = \Omega$), and $E$ be a separable metric space.
The map $f \mapsto (x \mapsto (y \mapsto f(x,y))$ defines a map
$L(X \times Y,E) \to L(X,L(Y,E))$. This map is an isometry.
\end{proposition}
\begin{proof} This amounts to saying that, if $f$ is Borel, then every $f^x : y \mapsto f(x,y)$
is Borel $Y \to E$, and that the map $x \mapsto f^x$ is Borel $X \to L(Y,E)$.

First assume $E \subset \R_+$. 
Let $f \in \mathcal{L}(X \times Y, E)$, that is a Borel map $X \times Y \to E \subset \R_+$. 
The Fubini-Tonelli
theorem states that every $f^x : y \mapsto f(x,y)$ is Borel. 
This defines a map $F : x \mapsto f^x$, $X \to \mathcal{L}(Y,E)$. We want to show that
it is a Borel map with respect to the (non-Hausdorff) topology on $\mathcal{L}(Y,E)$
defined by the natural pseudo-metric. For this we need to show that $F^{-1}(U)$ is Borel
for every open subset $U$ in $\mathcal{L}(Y,E)$. Since $\mathcal{L}(Y,E)$ is pseudo-metric
and separable (because $E$ is separable, see prop. \ref{prop:LEcontractcomplete}),
its topology admits a countable basis of open balls. We thus need to prove that $F^{-1}(U)$ is Borel
for every open ball $U$, with center $g_0 \in \mathcal{L}(Y,E)$ and radius $\eps > 0$.
Then $F^{-1}(U) = \{ x ; \int d(f^x(y), g_0(y)) \dd y < \eps \}$.
Let $g : X \times Y \to \R_+$ be defined by $g(x,y) = d_{E}(f(x,y),g_0(y))$. Being a composite
of Borel and a continuous maps, it is Borel. By the Fubini-Tonelli theorem,
$G : x \mapsto \int g(x,y) \dd y$ is Borel, and therefore $F^{-1}(U) = G^{-1}([0,\eps[)$
is Borel. This proves the claim in the case $E \subset \R_+$.

We now consider the general case. Let us choose a countable dense subset $(x_i)_{i \in I}$ inside $E$, and let $\varphi_i : E \to \R_+$
denote the function $x \mapsto d(x,x_i)$. These maps are $1$-Lipshitz and separate the points of $E$.
The maps $L(Y,\varphi_i) : g \mapsto \varphi_i \circ g$, $\mathcal{L}(Y,E) \to \mathcal{L}(Y,\R_+)$
are also well-defined, 1-Lipschitz and therefore Borel, and the induced maps $\Phi_i : L(Y,E) \to L(Y,\R_+)$
separate the points of $L(Y,E)$. Indeed, if $f ,g : Y \to E$ satisfy $\Phi_i(f) = \Phi_i(g)$
for all $i \in I$, that is $\int_{Y} |d(f(t),x_i)- d(g(t),x_i)|\dd t = 0$, 
this means that the sets $Y_i = \{ t; d(f(t),x_i) \neq d(g(t),x_i) \}$ have measure $0$. But then
$Y' = \bigcup_{i \in I} Y_i$ also has measure $0$, and for all $t \in Y\setminus Y'$
we have $\forall i \in I \ d(f(t),x_i) = d(g(t),x_i)$, which implies $f(t)=g(t)$ by density. Finally this implies $d(f,g)=0$
and this proves that this family separates the points.

Therefore, the induced map $\iota : L(Y,E) \to \prod_{i \in I} L(Y,\R_+)$ is injective,
and continuous where the topology at the range is the product topology. 
By \cite{CHRISTENSEN} theorem 2.4 this implies that $\iota$ is
a Borel isomorphism $L(Y,E) \to \iota(L(Y,E))$. In particular, $f : X \to L(Y,E)$
is measurable if and only if $\iota \circ f : X \to L(Y,\R_+)^{I}$ is measurable. 
But this means that $\Phi_i \circ f$ is measurable for all $i \in I$.

 To each map $f \in \mathcal{L}(X \times Y, E)$
we associate $x \mapsto f^x$ with $f^x : Y \to E$, $y \mapsto f(x,y)$ as before.
Each $f^x$ is Borel iff each $\varphi_i \circ f^x = (\varphi_i \circ f)^x : Y \to \R_+$ is Borel
by lemma \ref{lem:fctseparantes}, and if it is the case $f \mapsto f^x$ is Borel iff $ \varphi_i \circ f^x = \Phi_i( f)^x$
is Borel for all $i$. But since $f$ is Borel all $\varphi_i \circ f$ are Borel and so are
the $(\varphi_i \circ f)^x$ for all $x$, and then $\Phi_i(f)$ is Borel $X \times Y \to \R_+$
which implies that $x \mapsto \Phi_i(f)^x$ is Borel by the case $E \subset \R_+$.

Therefore we have a well-defined map $L(X \times Y,E) \to L(X,L(Y,E))$. Let
$f,g \in L(X \times Y,E)$. Then $d(f,g) = \int_{X \times Y} d(f(x,y),g(x,y)) \dd x \dd y
= \int_X (\int_Y d(f(x,y),g(x,y)) \dd y) \dd x = \int_X d(f^x,g^x)\dd x = d(x \mapsto f^x, x \mapsto g^x)$
by the Fubini-Tonelli theorem applied to the function $(x,y) \mapsto d(f(x,y),g(x,y))$ which
is Borel $X \times Y \to \R_+$. This proves that our map is an isometry.

\end{proof}

By the proposition, for each separable metric space $E$, we have an `exponential map' $\exp_E : L(\Omega\times \Omega,E)
\to L(\Omega,L(\Omega,E))$. We first consider the case where $E$ is a closed bounded interval.

\begin{proposition} \label{prop:skolemintervalle} The exponential map $\exp_E : L(\Omega\times \Omega,E)
\to L(\Omega,L(\Omega,E))$
is surjective when $E = [0,m]$ for some $m >0$.
\end{proposition}

We set $E = [0,m]$. We need the following lemma.

\begin{lemma} \label{lem:approxL1rats} Let $g : [0,1] \to [0,M[ \subset \R^+$ bounded and Borel. Let $(r_n)_{n \geq 1}$
be a bijection $\N \setminus \{ 0 \} \to \Q_{\geq 0}\cap[0,M[$. For all $n \geq m \geq 1$
we set $A(m,n) = \{ t \in [0,1] \ | \ \forall k \leq n, k \neq m \Rightarrow |g(t)-r_k| > |g(t) - r_m| \}$ ; 
for all $n \geq m_2 > m_1 \geq 1$ we set 
$B(m_1,m_2,n) = \{ t \in [0,1] \ | \ \forall k \leq n, k \not\in \{ m_1,m_2 \}  \Rightarrow |g(t)-r_k| > |g(t) - r_{m_1}| = |g(t) - r_{m_2}| \}$ 
and 
$$h_n = \sum_{m \leq n} r_m \un_{A(m,n)} + \sum_{1 \leq m_1 < m_2 \leq n} \frac{r_{m_1} + r_{m_2}}{2} \un_{B(m_1,m_2,n)}.$$
 Then the sequence $h_n$
 converges
to $g$, and $\int |h_n - g|$ converges to $0$.
\end{lemma}
\begin{proof}
Each $A(m,n)$ is the intersection of a finite number of subsets
of the form $\{ t; |g(t) - r_k| - |(g(t) - r_m| > 0 \}$,
which are the inverse images of the Borel set $[0,+\infty[$
by a Borel map $[0,1] \to \R_+$. Therefore each $A(m,n)$ is Borel, and so are
the $B(m_1,m_2,n)$ by a similar argument. This implies that
$h_n$ is Borel. Note that
all the sets $A(m,n)$, $B(m_1,m_2,n)$ are disjoint for a given $n$.

Let $t \in [0,1]$, $\eps > 0$. Let $n_0$ such that $|g(t) - r_{n_0}| < \eps$.
Then, for $n \geq n_0$, we have $h_n(t) = r_{n_0}$ hence
$|h_n(t) - g(t)| < \eps$, except if there exists $m \leq n$
such that $|g(t)- r_m | < |h_n(t) - r_{n_0}| < \eps$. We can choose $m\leq n$
such $|g(t)- r_m |$ is minimal. If there exists only one such $m$,
then $t \in A(m,n)$ and $h_n(t) = r_m$, whence $|h_n(t) - g(t) | < \eps$.
If not, there may exist only two such $m$'s, call them $m_1$ and $m_2$,
$m_1 < m_2$, and we have $g(t) = (r_{m_1}+r_{m_2})/2$.
In this case $h_n(t) = g(t)$ whence $|h_n(t) - g(t)| < \eps$.
This proves that $h_n(t)$ converges to $g(t)$ for every $t \in [0,1]$.
Since $h_n \leq r_1+g$ the conclusion follows from Lebesgue's dominated
convergence theorem.

\end{proof}

We now prove the surjectivity of $\exp_E$. Let  $f \in \mathcal{L}(X, \mathcal{L}(Y,E))$. 
We first choose an (injective) enumeration $(r_n)$ of the positive rationals 
of $[0,M[$ as in the lemma, with $M = m+1$.
For each $x \in X$, we can apply lemma \ref{lem:approxL1rats} to $f(x) \in \mathcal{L}(Y,[0,M[)$.
We define $A_x(m,n)$, $B_x(m_1,m_2,n)$
as the Borel subsets associated to $f(x)$ by the statement of the lemma,
and $\mathcal{A}(m,n) = \{ (x,y) ; y \in A_x(m,n) \}$,
$\mathcal{B}(m_1,m_2,n) = \{ (x,y) ; y \in B_x(m_1,m_2,n) \}$.
When $C$ is a subset of $X \times Y$,
we let $C^x = \{ y ; (x,y) \in C \}$.
We prove that the maps 
$
f_{m,n} : x \mapsto \mathcal{A}_{m,n}^x = A_x(m,n) 
$
are Borel. For this, we denote $f_k : t \mapsto |f(t) - r_k| -|f(t)-r_m|$. Then $A_x(m,n) = \bigcap_{\stackrel{k \leq n}{k \neq m}} f_k^{-1}
(]0,+\infty[)$. The map $f \mapsto f_k$
is Borel $L^1([0,1],\R) \to L^1([0,1],\R)$, so is
$g \mapsto \max(g,0)$, $L^1([0,1],\R)\to L^1([0,1],\R_+)$
and so is $g \mapsto g^{-1}(]0,+\infty[)$ by proposition \ref{prop:supportborel}.
Therefore their composite $\Phi_k : L^1([0,1],\R) \to L(2)$
is Borel, and so is the
map $G = \bigoplus_k \Phi_k :  L^1([0,1],\R) \to L(2)^{m-1}$.
Since, identifying $L(2)$ with the set of (equivalence classes of) Borel
subspaces of $[0,1]$,  the intersection map $L(2)^q \to L(2)$, $(B_i)_{1 \leq i \leq q} \mapsto \bigcap_i B_i$ is continuous, then by composing it with $G$
we get that the $f_{m,n}$ are Borel. Similarly, one shows that
the maps $x \mapsto \mathcal{B}_{m_1,m_2,n}^x$ are Borel.

Because of this, by proposition
\ref{prop:approxborelcarre}, there exists Borel subsets of $X \times Y$,
$\tilde{\mathcal{A}}(m,n)$ and $\tilde{\mathcal{B}}(m_1,m_2,n)$,
such that 
$\int \mu_Y( A_x(m,n)\Delta \tilde{\mathcal{A}}_{m,n}^x) \dd x = 0$
and 
$\int \mu_Y(B_x(m_1,m_2,n)\Delta \tilde{\mathcal{B}}_{m_1,m_2,n}^x)\dd x= 0$, where $\mu_Y$
is the measure on $Y$. We let then
$\int_x \la(B_x(m_1,m_2,n)\Delta \tilde{\mathcal{B}}_{m_1,m_2,n}^x)= 0$. We let then
$$
H_n =  \sum_{m \leq n} r_m \un_{\tilde{\mathcal{A}}(m,n)} + \sum_{1 \leq m_1 < m_2 \leq n} \frac{r_{m_1} + r_{m_2}}{2} \un_{\tilde{\mathcal{B}}(m_1,m_2,n)}.
$$
The $H_n$ clearly belong to $L^1(X \times Y \to \R)$, and they
form a Cauchy sequence, because
$$
\int_{X \times Y} |H_{n_1}(z) - H_{n_2}(z)| = \int_X \int_Y |H_{n_1}(z) - H_{n_2}(z)| 
= \int_X \int_Y |h_{n_1}(x,y) - h_{n_2}(x,y)| \dd y \dd x
$$
where
$$h_n(x,y) = \sum_{m \leq n} r_m \un_{A_x(m,n)}(y) + \sum_{1 \leq m_1 < m_2 \leq n} \frac{r_{m_1} + r_{m_2}}{2} \un_{B_x(m_1,m_2,n)}(y).$$
Since $h_n(x,\cdot)$ converges to $f(x)$ pointwise and is uniformly bounded, we get by Lebesgue's dominated convergence theorem
that, for all $x$,  
$\int_Y | h_n(x,y)- f(x)(y)| \dd y$
converges to $0$.
Since this is uniformly bounded
this implies that $\int_X \int_Y |h_n(x,y) - f(x)(y)| \dd y \dd x$ converges to $0$.
Since 
$$
 \int_X \int_Y |h_{n_1}(x,y) - h_{n_2}(x,y)| \dd y \dd x
\leqslant
\int_X \int_Y |h_{n_1}(x,y) - f(x)(y)| \dd y \dd x
+
\int_X \int_Y |h_{n_2}(x,y) - f(x)(y)| \dd y \dd x
$$
we get that the sequence $H_n$ is a Cauchy sequence inside $L^1(X \times Y,\R)$ and therefore converges to some Borel map $H : X \times Y \to \R$
inside $L^1(X \times Y,\R)$. Since the measure of $X \times Y$ is finite
this implies that we can replace $H_n$ by a subsequence so that
$H_n$ converges almost everywhere to $H$.
Since $\int_X \int_Y |H_n(x,y)-H(x,y)| \dd y \dd x = 
\int_X \int_Y |h_n(x,y)-H(x,y)| \dd y \dd x \to 0$
and $\int_X \int_Y |h_n(x,y) - f(x)(y)| \dd y \dd x \to 0$
we get that $H$ has for image $f$ inside $L(X,L(Y,\R_+))$
and this proves 
proposition \ref{prop:skolemintervalle}.

\begin{theorem}\label{theoskolemcompact} The exponential map $\exp_E : L(\Omega \times \Omega, E)
\to L(\Omega,L(\Omega,E))$ is a bijective isometry when $E$ is a Borel subspace
of a compact metric space.
\end{theorem}

\begin{proof} We let $H \subset \ell^2(\N)$ denote the Hilbert cube
$H = \{ \underline{y} = (y_n)_{n \in \N} \ | \ 0 \leq y_n \leq 1/n \}$. We first prove that $\exp_H$ is surjective.
For each $n \in \N$ let $\pi_n : H \to [0,1/n]$ denote the natural projection $\underline{y} \mapsto y_n$.
Since it is $1$-Lipschitz, the map $\Pi_n : \varphi \mapsto \pi_n \circ \varphi$, $\mathcal{L}(\Omega,H) \to \mathcal{L}(\Omega,[0,1/n])$
is also well-defined and 1-Lipschitz. Let $g : x \mapsto g_x$ be an element of $\mathcal{L}(\Omega,\mathcal{L}(\Omega,H))$.
The map $X_n : x \mapsto \Pi_n \circ g_x$ belongs to $\mathcal{L}(\Omega,\mathcal{L}(\Omega,[0,1/n]))$. Let $f_n \in \mathcal{L}(\Omega \times \Omega,[0,1/n])$
whose image in $L(\Omega,L(\Omega,[0,1/n]))$ coincides with (the one of) $X_n$. The map $\Omega \times \Omega \to H$
defined by $f(x,y) = (f_n(x,y))_{n \in \N}$ is Borel by lemma \ref{lem:fctseparantes}, since the maps $\pi_n$ form
a countable collection of separating maps. We need to prove that its image in $L(\Omega,L(\Omega,H))$ coincides with (the one of) $g$.
But
$$
\int_{\Omega} d(g_x,f^x) \dd x = \int_{\Omega}\int_{\Omega} d(g_x(y),f^x(y))\dd y \dd x
= \int_{\Omega}\int_{\Omega} \sqrt{\sum_n d(\pi_n \circ g_x(y),\pi_n \circ f^x(y))^2} \dd y \dd x
$$
There exists $X_0(n) \subset \Omega$ of full measure such that $X_n(x)$ coincides almost everywhere with $f_n^x$
for all $x \in X_0(n)$. Then, letting $X_0 = \bigcap_n X_0(n)$, we have $\mu(X_0) = 1$. For $x \in X_0$
there exists $Y_n(x) \subset \Omega$ of full measure such that $f_n^x(y) = g_x(y)$ for all $ y \in Y_n(x)$. 
Again, $Y(x) = \bigcap_n Y_n(x)$ has measure $1$ for every $x \in X_0$. It follows that
$$
\int_{\Omega}\int_{\Omega} \sqrt{\sum_n d(\pi_n \circ g_x(y),\pi_n \circ f^x(y))^2} \dd y \dd x
= 
\int_{X_0}\int_{Y(x)} \sqrt{\sum_n 0} \dd y \dd x = 0
$$
and this proves that $\exp_H$ is surjective.

Now assume that $\exp_E$ is known to be surjective for some separable bounded metric space $E$,
and let $F \subset E$ be a nonempty Borel subset endowed with the induced
metric. Let  $g : x \mapsto g_x$ be an element of $\mathcal{L}(\Omega,\mathcal{L}(\Omega,F)) \subset \mathcal{L}(\Omega,\mathcal{L}(\Omega,E))$.
There exists $f \in \mathcal{L}(\Omega \times \Omega,E)$ whose image inside $L(\Omega,L(\Omega,E))$ coincides with the
image of $g$. This means that there exists $X_0 \subset \Omega$ of full measure such that $f^x$ and $g_x$ define the same element
in $L(\Omega,E)$ for all $x \in X_0$. Let $A = \{ (x,y) \in \Omega \times \Omega \ | \ f(x,y) \not\in F \}$.
Since $F$ and $f$ are Borel, $A = f^{-1}(E \setminus F)$ is a Borel subset of $\Omega \times \Omega$. 
We pick $c \in F$ and define $f_0$ by $f_0(x,y) = f(x,y)$ if $(x,y) \in A$ and
$f_0(x,y) = c$ otherwise. Since $A$ is Borel this defines a Borel map $\Omega \times \Omega \to F$,
whose image in $\mathcal{L}(\Omega,\mathcal{L}(\Omega,F))$
satisfies $d(x \mapsto f_0^x, g) 
= \int_{x \in \Omega} \int_{y \in \Omega} d(f_0(x,y),g_x(y)) \dd y \dd x
= \int_{x \in X_0} \int_{y \in \Omega} d(f_0(x,y),g_x(y)) \dd y \dd x$.
But for $x \in X_0$ we know that $g_x(y)$ coincides with $f(x,y)$ and thus belongs to $F$ almost everywhere in $y$,
whence $\int_{y \in \Omega} d(f_0(x,y),g_x(y)) \dd y = \int_{y \in \Omega} d(f(x,y),f(x,y)) \dd y =0$ for
all $x \in X_0$. This proves $d(x \mapsto f_0^x, g)=0$ and thus $g = \exp_F(f_0)$.

Now consider an arbitrary compact metric space $K$ with diameter $\delta$, and $(x_n)_{n \in \N}$ a dense countable collection of points.
We consider the classical map $\Delta : K \to H$, $x \mapsto (d(x,x_n)/(2^n\delta))_n$.
It is injective, $2/\delta\sqrt{3}$-Lipschitz, and embeds $K$ as a compact subspace $K_H = \Delta(K)$ of $H$. Since
$K$ and $K_H$ are compact, the map $\Delta : K \to K_H$ is an homeomorphism whose inverse
$\Delta^{-1} : K_H \to K$ is uniformly continuous. Therefore, by lemma \ref{lem:BLipUC} we get
natural uniformly continuous homeomorphisms $\Delta : L(\Omega,L(\Omega,K)) \to L(\Omega,L(\Omega,K_H))$ and 
$\Delta : L(\Omega\times \Omega,K) \to L(\Omega \times \Omega, K_H)$, which satisfy $\exp_K = \Delta^{-1} \circ \exp_{K_H} \circ \Delta$.
Since $\exp_H$ and therefore $\exp_{K_H}$ is surjective,
this proves that $\exp_E$ is surjective for every compact metric space, and therefore
for every Borel subset of an arbitrary compact metric space.

\end{proof}

We remark that this result can be applied to compact Lie groups, as well as to metric profinite groups.

\subsection{Dense subsets of $L(\Omega,L(\Omega,E))$}

\subsubsection{Staircase maps}

Recall that, if $X$ and $Y$ are two measured space, then
a staircase map from $X$ and $Y$ is a map $f : X \to Y$ such that there
exists a finite partition $X$ into measurable subsets $X_1,\dots,X_r$
such that $f$ is constant on each $X_i$.

\begin{proposition} \label{prop:densestaircase} Let $D$ be a discrete metric space. The set of staircase maps is dense inside $L(\Omega,L(\Omega,D))$.
\end{proposition}
\begin{proof}
We identify $\Omega$ with $I = [0,1]$ endowed with the Lebesgue measure, and
fix some element $e \in D$.
By theorem \ref{theoskolem} we have a natural isomorphism $L(I \times I,D) \to L(I,L(I,D))$. By 
proposition \ref{prop:imagedenombrable}
it is clear that
staircase maps in $L(I \times I,D)$ are dense. So it is enough to show that every staircase map in $L(I\times I,D)$
can be approximated by a map which is staircase as an element of $L(I,L(I,D))$. 
First assume that this holds true for maps of the form 
$$
\begin{array}{lclcccl}
\varphi_{A,x} & :&  t & \mapsto & x & \mbox{if}& t \in A \\
& & t & \mapsto & e & \mbox{if}& t \not\in A
\end{array}
$$
where $A$ is a Borel set in $I \times I$ and $x \in D$. Given a staircase map $\varphi : I \times I \to D$
associated to a partition $A_1 \sqcup A_2 \sqcup \dots \sqcup A_r$ of $I \times I$ such that $\varphi(t) = d_i \in D$
whenever $t \in A_i$, we can endow $D$ with some total ordering such that $e$ is minimal ; given
staircase maps $(f_{i,n})_{n \geq 0}$ tending to $\varphi_{A_i,d_i}$ as $n \to \infty$,
we may consider the sequence $f_n = \max(f_{1,n},\dots,f_{r,n})$. We may assume that the $d_i$'s are distinct,
that is $i\neq j \Rightarrow d(d_i,d_j)=1$.

It is clear that the $f_{i,n}$ are again staircase maps inside $L(I,L(I,D))$. We need to prove that $f_n \to f$. Let $\eps >0$. We know that, for $n \geq n_0$, we have $d(f_{i,n},\varphi_{A_i,d_i}) \leq \eps/r^2$,
for $i=1,\dots,r$. We have
$$
d(f,f_n) = \sum_{i=1}^r \int_{A_i} d(f(t),f_n(t)) \dd t = \sum_{i=1}^r \int_{A_i} d(\varphi_{A_i,d_i}(t) ,f_n(t))\dd t.= \sum_{i=1}^r \int_{A_i} d(d_i,f_n(t))\dd t.
$$
Let us fix $i \in \{1,\dots,r\}$. The condition $d(f_{i,n},\varphi_{A_i,d_i}) \leq \eps/r^2$ implies
that the set $\{ t \in A_i ; f_{i,n}(t) \neq d_i \}$ has measure at most $\eps/r^2 $. For
an arbitrary $j \neq i$ and $t \in A_i$ 
we have that
$f_n(t) \neq d_i$ implies that either $f_{i,n}(t) \neq \varphi_{A_i,d_i}(t)=d_i$
or there exists $j \neq i$ such that $f_{j,n}(t) \neq \varphi_{A_j,d_j}(t)=e$. It follows that
the set $\{ t \in A_i ; f_{jn}(t) \neq e \}$ has measure at most $\sum_j d(f_{j,n},\varphi_{A_j,d_j}) \leq r \times \eps/r^2 = \eps/r$.
Therefore,
the set $\{ t \in A_i ; f_n(t) = d_i \}$ has measure at least $\mu(A_i) - \eps/r$. 

Since this holds for all $i \in \{1,\dots,r \}$, summing up we get $d(f,f_n) \leq r \times \eps/r = \eps$.

We are then reduced to show that every $\varphi_{A,d}$ can be approximated by staircase maps in $L(I,L(I,D))$.
For this we can assume $d=1$, $e= 0$, and $D = \{0,1 \} \subset \R$. We need to prove that, given a Borel subset $A$
of $I\times I = [0,1]\times [0,1]$, the characteristic function $\un_A \in L^1([0,1]^2,\R)$
can be approximated by staircase functions in $L^1(I,L^1(I,\R))$ which take values only in $\{0,1\}$.
We let $\mathcal{T}$ denote the set of Borel sets $A$ with this property. It clearly includes the sets of the form
$U \times V$ for $U,V$ Borel subsets of $I$, hence $\mathcal{T}$ will contain all Borel subsets of $[0,1]$ if we can prove
that it is a $\sigma$-algebra (since the Borel $\sigma$-algebra of $[0,1]$ is the cartesian square of the Borel $\sigma$-algebra
of $[0,1]$).

We first prove that $\mathcal{T}$ is stable under finite unions. Indeed, if $A = A_1 \cup A_2 \cup \dots \cup A_r$
with each $A_i$ in $\mathcal{T}$, let us consider $f = 1_A$, $f_i = 1_{A_i}$. By assumption,
$f_i$ is the limit inside $L^1(I\times I,\R)$ of staircase maps $f_{i,n} = L(I,L(I,\{0,1\}))$.
We have $f = \max(f_1,\dots,f_r)$, and the maps $g_n = \max(f_{1,n},\dots,f_{r,n})$
are staircase maps in $L(I,L(I,\{0,1\}))$. By the same argument as above we have $g_n \to f$
and this proves $A \in \mathcal{T}$.

We now prove that it is stable under countable unions. We assume $A = \bigcup_{i \in \N} A_i$. Because we proved
that $\mathcal{T}$ is stable under finite unions, we can assume $A_0 \subset A_1 \subset \dots$. We let $f = \un_A$,
$f_i = \un_{A_i}$. Then $f= \sup_i(f_i)$, hence $(f_i) \to f$. Some $\eps>0$ being chosen,
we have $i_0 \in \N$ such that $d(f,f_{i_0}) \leq \eps/2$ ; $i_0$ being fixed, there is $n_0 \in \N$ such that $d(f_{i_0},f_{i_0,n_0}) \leq \eps/2$, and this provides a staircase map such that $d(f,f_{i_0,n_0}) \leq \eps$. Therefore, $A \in \mathcal{T}$.

Finally, if $A \in \mathcal{T}$, then $\,^c A=I\times I \setminus A \in \mathcal{T}$. This holds true because the map from $L(I,L(I,\{0,1\}))$
to itself given by $f \mapsto 1-f$ is continuous and preserves the set of staircase maps.

This proves that $\mathcal{T}$ contains all Borel sets and concludes the proof.
\end{proof}

\section{Iterations and Speculations}

\label{sect:iteration}

Let $\Gamma$ be an \emph{abelian} discrete metric group, and $G = L(\Omega,\Gamma)/\Gamma$. We know that $G$
is a $K(\Gamma,1)$ and a complete metric group, separable if and only if $\Gamma$ is countable. Therefore
$L(\Omega,G)/G$ could be expected to be a classifying space for $G$. If the projection map $L(\Omega,G) \to L(\Omega,G)/G$
induces a long exact sequence in homotopy (for instance if it is a Serre fibration or a Dold fibration or a quasifibration),
then $L(\Omega,G)/G$ has $\pi_2 = \Gamma$ for only nontrivial homotopy
group. If in addition $L(\Omega,G)/G$ has the homotopy type of a CW-complex, then it is a $K(\Gamma,2)$.

Let us define $E_1(\Gamma) = L(\Omega,\Gamma)$, $B_1(\Gamma) = L(\Omega,\Gamma)/\Gamma$
and by induction $E_{n+1}(\Gamma) = L(\Omega,B_n(\Gamma))$, $B_{n+1}(\Gamma) = E_{n+1}(\Gamma)/B_n(\Gamma)$.
By convention, $B_0(\Gamma) = \Gamma$. We consider the following hypothetical
properties.

\bigskip

{\bf Properties $C_n(m)$}
\begin{enumerate}
\item $B_{n}(\Gamma)$ is a classifying space for $B_{n-1}(\Gamma)$
\item The projection map $E_n(\Gamma) \to B_n(\Gamma)$ is a quasifibration.
\item $\pi_n(B_{n}(\Gamma)) = \Gamma$, $\pi_k(B_{n}(\Gamma)) = 1$ for $k \neq n$
\item $B_n(\Gamma)$ has the homotopy type of a CW-complex.
\item $B_n(\Gamma)$ is a $K(\Gamma, n)$
\item The projection map $E_n(\Gamma) \to B_{n}(\Gamma)$ admits local cross-sections.
\item The projection map $E_n(\Gamma) \to B_{n}(\Gamma)$ is a Serre fibration.
\item $B_{n}(\Gamma)$ is locally contractible.
\end{enumerate}

\medskip
Properties $C_1(m)$ are known to be true. We have $C_n(3) \& C_n(4) \Rightarrow C_n(5)$,
and $C_n(2) \& C_{n-1}(3) \Rightarrow C_n(3)$. We have $C_n(6) \Rightarrow C_n(2)$, 
$C_n(6) \Rightarrow C_{n}(1) \Rightarrow C_{n}(4)$. But so far we have no indication
on the validity of any of the $C_n(m)$, $n \geq 2$.
\medskip

For every $n \geq 1$, the space $\Omega^n$ is a probability space. For $i \in \{1, \dots, n \}$
we define $L_{i,n}(\Gamma)$ as the subspace of (classes of) functions  $f \in L(\Omega^n, \Gamma)$
such that $f(x_1,\dots,x_n)$ is independent of $x_i$. It is a closed subgroup of $L(\Omega^n, \Gamma)$,
isomorphic to $L(\Omega^{n-1},\Gamma)$ as metric group. We let $L_n(\Gamma)$ denote the
subgroup of $L(\Omega^n,\Gamma)$ generated by the subgroups  $L_{1,n}(\Gamma)\dots L_{n,n}(\Gamma)$.

If $\Gamma$ is countable, $L(\Omega^n,\Gamma)$ is separable and therefore Polish, and therefore every quotient map $L(\Omega^n,\Gamma) \to L(\Omega^n,\Gamma)/H$
admits a (global) Borel cross-section (\cite{SRI} theorem 5.4.2 p. 196).

\begin{lemma} \label{lem:projiter} Assume that $G$ is a quotient of $L(\Omega,\Gamma)$.
Let $N$ be a closed normal subgroup of $G$.
The map $L(\Omega,G) \to L(\Omega,G/N)$ is continuous and surjective with kernel $L(\Omega,N)$. It induces an isomorphism of topological groups
$$
\frac{L(\Omega,G)}{L(\Omega,N)} \to L(\Omega,G/N)
$$
which is an isometry.
\end{lemma}
\begin{proof}
The map $L(\Omega,G) \to L(\Omega,G/N)$ is continuous with kernel $L(\Omega,N)$. We prove that it is surjective. First assume that $\Gamma$ is countable. Then 
$G$ is Polish, whence the projection map $G\to G/N$ admits a (global)
Borel cross-section $s : G/N \to G$. Then every Borel map $f : \Omega \to G/N$ is the image of $s \circ f$ under $L(\Omega,G) \to L(\Omega,G/N)$, and therefore
$L(\Omega,G) \to L(\Omega,G/N)$ is onto.
In the general case, let us consider a Borel map $f : \Omega \to G/N$. By proposition \ref{prop:imageseparable} we know that $f(\Omega)$ is separable. Let $(g_n)_{n \in \N}$ be a family in $L(\Omega,\Gamma)$ whose image $\overline{g_n}$ in $G/N$ is dense. By proposition \ref{prop:imagedenombrable} each $g_n$ has countable image and therefore the exists a countable $D_0 \subset \Gamma$ such that $g_n \in L(\Omega,D_0)$ for all $n$.
Let $\Gamma_0 = \langle D_0 \rangle$, and $G_0<G$ the image of $L(\Omega,\Gamma_0)$ under the projection map $L(\Omega,\Gamma)\onto G$.
Every $g_n$ belongs to $L(\Omega,\Gamma_0)$, hence their image in $L(\Omega,G/N)$ lie in $L(\Omega,G_0/(G_0 \cap N))$, which is closed. It follows that $f$ lies in $L(\Omega,G_0/(G_0\cap N))$
and therefore it admits a preimage inside $L(\Omega,G_0) \subset L(\Omega,G)$, since $\Gamma_0$ is countable. This proves the surjectivity.

From this we get a bijective morphism of topological groups $\Phi : \frac{L(\Omega,G)}{L(\Omega,N)} \to L(\Omega,G/N)$.
It remains to prove that this is an isometry, which in particular will imply that is an isomorphism of topological groups. Since it is continuous it suffices to prove that $d(\Phi(f_1),\Phi(f_2)) = d(f_1,f_2)$ for all $f_1,f_2$ belonging to some dense subset
of $L(\Omega,G)/L(\Omega,N)$. Applying our result to the quotient map $L(\Omega,\Gamma) \to G$, we know that the continuous map $L(\Omega,L(\Omega,\Gamma))\to L(\Omega,G)$
is surjective. Since the staircase maps are dense inside $L(\Omega,L(\Omega,\Gamma))$ by proposition \ref{prop:densestaircase}, so are their images inside $L(\Omega,G)$.
This proves that staircase maps are dense in $L(\Omega,G)$, and so it suffices to prove $d(\Phi(\overline{f_1}),\Phi(\overline{f_2})) = d(\overline{f_1},\overline{f_2})$ for $f_k : \Omega \to G$, $k=1,2$ two staircase maps, and $\overline{f_k}$ their images in $L(\Omega,G)/L(\Omega,N)$.
Since $d$ is bi-invariant and $\Phi$ is an homomorphism and a product of two staircase maps is a staircase map it suffices to prove $d(\Phi(\overline{f}),1) = d(\overline{f},1)$ for $f$ an arbitrary staircase map. 

Let $f$ be such a staircase map. There is a partition $\Omega = A_1 \sqcup \dots \sqcup A_m$ of Borel sets and $a_1,\dots, a_m \in G$ such
that $f(A_i) = \{ a_i \}$ for $i = 1,\dots, m$. Then
$$
\begin{array}{lclcl}
d(\overline{f},1) &=& \inf \{ d(b.f,1) \ | \ b \in L(\Omega,N) \} \\
&=& \inf \{ \int_{\Omega} d(b(t)f(t),1) \ | \ b \in L(\Omega,N) \} \\
&=& \inf \{ \sum_{i=1}^m \int_{A_i} d(b(t),a_i^{-1})\dd t \ | \ b \in L(\Omega,N) \} \\
&=& \sum_{i=1}^m \inf \{  \int_{A_i} d(b(t),a_i^{-1})\dd t \ | \ b \in L(A_i,N) \} \\
&=& \sum_{i=1}^m \inf \{  \int_{A_i} d(b,a_i^{-1})\dd t \ | \ b \in N \} \\
&=& \sum_{i=1}^m \mu(A_i) \inf \{   d(ba_i,1)\dd t \ | \ b \in N \} \\
&=& \sum_{i=1}^m \int_{A_i} \inf \{   d(ba_i,1)\dd t \ | \ b \in N \} \dd t \\
&=& \int_{\Omega} \inf \{   d(bf(t),1)\dd t \ | \ b \in N \} \dd t \\
&=& d(\Phi(\overline{f}),1)\\
\end{array}
$$
and this proves the claim.
\end{proof}

\begin{lemma}  \label{lem:iteriso2} Assume that $G$ is a (Hausdorff) quotient of $L(\Omega,\Gamma)$.
Let $N$ be a closed normal subgroup of $G$.
The map $L(\Omega,G) \to L(\Omega,G/N)/(G/N)$ is continuous and surjective with kernel $G.L(\Omega,N)$. It induces an isomorphism of
topological groups
$$
\frac{L(\Omega,G)}{G.L(\Omega,N)} \to \frac{L(\Omega,G/N)}{G/N}
$$
which is an isometry.
\end{lemma}
\begin{proof}
By lemma \ref{lem:projiter} the natural map $L(\Omega,G)/L(\Omega,N) \to L(\Omega,G/N)/(G/N)$ is surjective
and therefore so is $L(\Omega,G) \to L(\Omega,G/N)/(G/N)$. Its kernel is clearly $G.L(\Omega,N)$. It remains to prove that the
induced continuous algebraic isomorphism $\Phi : L(\Omega,G)/G.L(\Omega,N) \to L(\Omega,G/N)/(G/N)$ is an isometry, which will imply in particular
that it is an isomorphism of topological groups.
As in the proof of lemma \ref{lem:projiter} we only need to check $d(\Phi(\overline{f}),1) = d(\overline{f},1)$ for $f : \Omega \to G$
a staircase map. We have a partition $\Omega = A_1 \sqcup \dots \sqcup A_m$ and $a_1,\dots,a_m \in G$ with $f(A_i) = \{ a_i \}$.
Then
$$
\begin{array}{lcl}
d(\bar{f},1) &=& \inf \{ d(gbf ,1) \ | \ g \in G, b \in L(\Omega,N) \} \\
&=& \inf \{ \int_{\Omega} d(gb(t)f(t) ,1)\dd t \ | \ g \in G, b \in L(\Omega,N) \} \\
&=& \inf \{ \sum_{i=1}^m  \int_{A_i} d(gb(t)f(t) ,1)\dd t \ | \ g \in G, b \in L(\Omega,N) \} \\
&=& \inf_{g \in G} \inf \sum_{i=1}^m \{ \int_{A_i} d(b(t) ,(ga_i)^{-1})\dd t \ | \  b \in L(A_i,N) \} \\
&=& \inf_{g \in G} \inf \sum_{i=1}^m \{ \int_{A_i} d(b ,(ga_i)^{-1})\dd t \ | \  b \in N \} \\
&=& \inf_{g \in G} \inf \sum_{i=1}^m \{ \int_{A_i} d_{G/N}(\bar{g} ,\overline{a_i})\dd t   \} \\
&=& \inf_{g \in G} \inf \{ \int_{\Omega} d_{G/N}(\bar{g} ,\overline{f(t)})\dd t \} \\
&=& \inf_{\overline{g} \in G/N} \inf \{ \int_{\Omega} d_{G/N}(\bar{g} ,\overline{f(t)})\dd t  \} \\
&=& \inf_{\overline{g} \in G/N} \inf \{ d_{L(\Omega,G/N)}(\bar{g} ,\overline{f(t)}) \} \\
&=& d(\Phi(\overline{f}),1) \\
\end{array}
$$
and this proves the claim.

\end{proof}

\begin{proposition} 
For every $n \geq 1$, $L_n(\Gamma)$ is a closed subgroup of $L(\Omega^n,\Gamma)$,
and $B_n(\Gamma) \simeq L(\Omega^n,\Gamma)/L_n(\Gamma)$ by an isometric isomorphism.
\end{proposition} 
\begin{proof} The proof is by induction on $n$, the case $n =1$ being obvious. Then, by definition $B_{n+1}(\Gamma) = L(\Omega,B_n(\Gamma))/B_n(\Gamma)$.
By the induction assumption we can identify $B_n(\Gamma)$ and $L(\Omega^n,\Gamma)/L_n(\Gamma)$.
We now consider the map 
$$L(\Omega,L(\Omega^n,\Gamma)) \to L(\Omega,L(\Omega^n,\Gamma)/L_n(\Gamma)) \to  \frac{L(\Omega,L(\Omega^n,\Gamma)/L_n(\Gamma))}{L(\Omega^n,\Gamma)/L_n(\Gamma)}
$$
It is a continuous map, and surjective by lemma \ref{lem:projiter}.
 By lemma \ref{lem:iteriso2} its kernel is $L(\Omega^n,\Gamma).L_n(\Gamma)$ and the induced algebraic isomorphism is an isometry.
  By theorem \ref{theoskolem} we have a natural identification $L(\Omega,L(\Omega^n,\Gamma)) \simeq L(\Omega\times \Omega^n, \Gamma) = L(\Omega^{n+1},\Gamma)$
  and this isomorphism identifies $L(\Omega^n,\Gamma).L_n(\Gamma)$ with $L_{n+1}(\Gamma)$, and this proves the claim.
\end{proof}

Let us now consider the following commutative diagram, where all the plain arrows represent the
natural maps described above.

$$
\xymatrix{
L(\Omega,L(\Omega,\Gamma)) \ar[d]_{\pi_h} \ar[r]^{\simeq} & L(\Omega^2,\Gamma) \ar[dd] \\
L(\Omega,B_1(\Gamma)) \ar[d] \ar@{.>}[ur]_{s_h} & \\
\frac{L(\Omega,B_1(\Gamma))}{B_1(\Gamma)}  \ar[r]^{\simeq} & \frac{L(\Omega^2,\Gamma)}{L_2(\Gamma)}
}
$$
Let us assume that $\Gamma$ is countable. This implies that all the spaces involved in the above diagram are separable.
By propositions \ref{prop:globborelsect}  and \ref{prop:existlocalisosectcovering} we have a global continuous cross section of $\pi_h$,
which combined with the isometric isomorphism $L(\Omega,L(\Omega,\Gamma)) \to L(\Omega^2,\Gamma)$ provides the map $s_h$.
A consequence of this is that property $C_2(6)$ (i.e. existence of a local cross-section for $E_2(\Gamma) \to B_2(\Gamma)$) is equivalent
to the existence of a local cross-section for $L(\Omega^2,\Gamma) \to L(\Omega^2,\Gamma)/L_2(\Gamma)$.
Another consequence is
 that $L(\Omega^2,\Gamma) \to L(\Omega^2,\Gamma)/L_2(\Gamma)$ has
the homotopy lifting property (HLP) with respect to a space $X$ iff $E_2(\Gamma) \to B_2(\Gamma)$) has the HLP
with respect to $X$. In particular one of the two is a Serre fibration iff the other one is.

In particular, if it is a Serre fibration, then $L_2(\Gamma)$
and $B_1(\Gamma)$ have the same weak homotopy type. Actually it is pretty straightforward to check directly that $L_2(\Gamma)$ is a $K(\Gamma,1)$, and therefore
is homotopically equivalent to $B_1(\Gamma)$, even for $\Gamma$ not necessarily countable.
\begin{proposition}
Let $\Gamma$ be a discrete abelian group. Then $L_2(\Gamma)$ is a $K(\Gamma,1)$ with universal cover
the natural map $ L(\Omega,\Gamma) \times L(\Omega,\Gamma) \to L_2(\Gamma)$
defined by $(\varphi, \psi) \mapsto \left( (x,y) \mapsto \varphi(x)\psi(y)\right)$. It can be endowed with a bi-invariant complete metric.
\end{proposition}
\begin{proof}
Let us denote $\pi : L(\Omega,\Gamma) \times L(\Omega,\Gamma) \to L_2(\Gamma)$ the above map. It is obviously a morphism
of topological groups, whose kernel is $\{ (g,g^{-1}) ; g \in \Gamma \} \subset L(\Omega,\Gamma)^2 \} \simeq \Gamma$.
Let us define a metric on $L(\Omega,\Gamma)^2$ via $d((f_1,f_2),(f'_1,f'_2)) = d(f_1,f'_1)+d(f_2,f'_2)$.
This metric is clearly bi-invariant and defines the original topology of $L(\Omega,\Gamma)^2$. Let $\underline{f} = (f_1,f_2) \in L(\Omega,\Gamma)^2$
and $U = \{ \underline{f'}  \in L(\Omega,\Gamma)^2 \ | \ d(\underline{f},\underline{f'}) < 1/2 \}$. Then, for all $\underline{f'}, \underline{f''} \in U$
if $\underline{f''} = \gamma. \underline{f'}$ with $\gamma \in \Gamma \setminus \{ 1 \}$ then
$2 = d(\underline{f},\gamma.\underline{f}) \leq d(\underline{f},\underline{f}'')+d(\underline{f}'',\gamma.\underline{f})
\leq d(\underline{f},\underline{f}'')+d(\gamma.\underline{f}',\gamma.\underline{f})
\leq d(\underline{f},\underline{f}'')+d(\underline{f}',\underline{f}) < 2 \times (1/2) <2$,
a contradiction. This proves that $\pi$ is a covering map. Since $L_2(\Gamma) \subset L(\Omega^2,\Gamma)$
is metric and therefore paracompact, this implies that $L_2(\Gamma)$ is a classifying space for $\Gamma$
hence a $K(\Gamma,1)$.

\end{proof}

Going one step further already presents some difficulties. We still assume that $\Gamma$ is countable. One has the following commutative diagram.

$$
\xymatrix{
 & L(\Omega,L(\Omega^2,\Gamma)) \ar[d]_{\pi_h} \ar[r]^{\simeq} & L(\Omega^3,\Gamma) \ar[dd] \\
L(\Omega,B_2(\Gamma)) \ar[d] \ar[r]^{\simeq} & L(\Omega,\frac{L(\Omega^2,\Gamma)}{L_2(\Gamma)})  & \\
\frac{L(\Omega,B_2(\Gamma))}{B_2(\Gamma)}  \ar[rr]^{\simeq} & &\frac{L(\Omega^3,\Gamma)}{L_3(\Gamma)}
}
$$

The previous argument cannot be extended `as is', because we do not know whether the map $L(\Omega^2,\Gamma) \to L(\Omega^2,\Gamma)/\Gamma$
admits local isometric cross-sections : we do not even know whether it admits a  continuous local cross-section. Therefore, we cannot
apply proposition \ref{prop:globborelsect}. This diagram however proves
that, \emph{if}  $L(\Omega^3,\Gamma) \to L(\Omega^3,\Gamma)/L_3(\Gamma)$ admits a local cross-section, or has the HLP with respect to some space $X$,
\emph{then} so does the map $L(\Omega,B_{2}(\Gamma)) \to B_3(\Gamma)$. This can be generalized through a straightforward induction.

\begin{lemma} If the map $L(\Omega^n, \Gamma) \to L(\Omega^n,\Gamma)/L_n(\Gamma)$ admits a local cross-section, or
is a Serre fibration, then so does the map $L(\Omega,B_{n-1}(\Gamma)) \to B_n(\Gamma)$.
\end{lemma}

This technical obstruction raises the question of whether the following strenghenings of property $C_n(6)$ hold true :

\bigskip

\begin{tabular}{ll}
$C_n(6)^+$ & The projection map $L(\Omega^n,\Gamma) \to L(\Omega^n,\Gamma)/L_n(\Gamma)$ admits local cross-sections \\
$C_n(6)^{++}$ & The projection map $L(\Omega^n,\Gamma) \to L(\Omega^n,\Gamma)/L_n(\Gamma)$ admits local isometric cross-sections. \\
$C_n(7)^+$ & The projection map $L(\Omega^n,\Gamma) \to L(\Omega^n,\Gamma)/L_n(\Gamma)$ is a Serre fibration \\
\end{tabular}

\bigskip

If $C_n(7)^+$ holds true, then $L(\Omega^n,\Gamma)/L_n(\Gamma) \simeq B_n(\Gamma)$ is a weak $K(\Gamma,n)$,
and moreover $L_n(\Gamma)$ is a weak $K(\Gamma,n-1)$.

A related question is whether the $B_n(\Gamma)$ provide an $\Omega$-(pre)spectrum. We finally note that, thanks to the isomorphism $\Omega^n \simeq \Omega$
these possible Eilenberg-MacLane spaces can be seen either as quotients (under the guise $L(\Omega^n)/L_n(\Gamma)$)
or as subspaces (under the guise $L_n(\Gamma)$) of the \emph{same} complete metric space $L(\Omega,\Gamma)$. We hope to come back to exploring the \emph{geometry}
of this space in future works.

\end{document}